\documentclass[11pt,reqno]{amsart}
\usepackage{amsmath,amssymb,verbatim,geometry}
\usepackage[pagebackref,pdftex,hyperindex]{hyperref}

\newcommand{\A}{\mathbb{A}}

\newcommand{\C}{\mathbb{C}}
\newcommand{\G}{\mathbb{G}}
\newcommand{\Q}{\mathbb{Q}}
\newcommand{\R}{\mathbb{R}}
\newcommand{\Z}{\mathbb{Z}}
\newcommand{\N}{\mathbb{N}}
\renewcommand{\P}{\mathbb{P}}

\newcommand{\fa}{\mathfrak{a}}
\newcommand{\fb}{\mathfrak{b}}
\newcommand{\fc}{\mathfrak{c}}
\newcommand{\fm}{\mathfrak{m}}

\newcommand{\tX}{\widetilde{X}}
\newcommand{\tB}{\widetilde{B}}

\newcommand{\tL}{\tilde{L}}

\newcommand{\tN}{\tilde{N}}

\newcommand{\tcL}{\tilde{\cL}}
\newcommand{\tcX}{\widetilde{\cX}}

\newcommand{\tmu}{\tilde{\mu}}

\newcommand{\cB}{\mathcal{B}}

\newcommand{\cF}{\mathcal{F}}

\newcommand{\cH}{\mathcal{H}}

\newcommand{\cL}{\mathcal{L}}
\newcommand{\cM}{\mathcal{M}}

\newcommand{\cO}{\mathcal{O}}

\newcommand{\cV}{\mathcal{V}}

\newcommand{\cX}{\mathcal{X}}
\newcommand{\cY}{\mathcal{Y}}

\newcommand{\Xan}{X^{\mathrm{an}}}

\newcommand{\Xdiv}{X^{\mathrm{div}}}

\newcommand{\Lan}{L^{\mathrm{an}}}

\renewcommand{\a}{\alpha}

\renewcommand{\d}{\delta}
\newcommand{\e}{\varepsilon}

\newcommand{\la}{\lambda}
\newcommand{\om}{\omega}

\newcommand{\Ga}{\Gamma}
\newcommand{\La}{\Lambda}

\newcommand{\cf}{{\rm cf.\ }} 
\newcommand{\eg}{{\rm e.g.\ }} 
\newcommand{\ie}{{\rm i.e.\ }} 
\newcommand{\triv}{\mathrm{triv}}

\newcommand{\red}{\mathrm{red}} 
\newcommand{\amp}{\mathrm{amp}}

\renewcommand{\DH}{\mathrm{DH}}

\newcommand{\lo}{\mathrm{log}}

\renewcommand{\=}{:=}

\newcommand{\an}{\mathrm{an}}

\DeclareMathOperator{\Spec}{Spec}

\DeclareMathOperator{\MA}{MA}

\DeclareMathOperator{\DF}{DF}

\DeclareMathOperator{\supp}{supp}
\DeclareMathOperator{\vol}{vol}
\DeclareMathOperator{\Pic}{Pic}

\DeclareMathOperator{\ord}{ord}
\DeclareMathOperator{\Hom}{Hom}
\DeclareMathOperator{\Aut}{Aut}
\DeclareMathOperator{\Proj}{Proj}

\DeclareMathOperator{\ratrk}{rat.rk}

\DeclareMathOperator{\CH}{CH}
\DeclareMathOperator{\hCH}{\widehat{CH}}

\DeclareMathOperator{\GL}{GL}

\DeclareMathOperator{\depth}{depth}
\DeclareMathOperator{\codim}{codim}
\DeclareMathOperator{\lct}{lct}

\DeclareMathOperator{\trdeg}{tr.deg}
\DeclareMathOperator{\ch}{ch}

\newcommand{\NA}{\mathrm{NA}}

\newcommand{\Td}{\mathrm{Td}}

\newcommand{\cro}[1]{[\![#1]\!]}
\newcommand{\lau}[1]{(\!(#1)\!)}

\numberwithin{equation}{section}       

\newtheorem{prop} {Proposition} [section]
\newtheorem{thm}[prop] {Theorem} 
\newtheorem{defi}[prop] {Definition}
\newtheorem{lem}[prop] {Lemma}
\newtheorem{cor}[prop]{Corollary}
\newtheorem{prop-def}[prop]{Proposition-Definition}

\newtheorem*{thmA}{Theorem A}

\newtheorem*{thmC}{Theorem C}

\newtheorem*{corB}{Corollary B} 
 
\newtheorem*{corD}{Corollary D} 
\newtheorem{exam}[prop]{Example}
\newtheorem{rmk}[prop]{Remark}
\theoremstyle{remark}

\newtheorem*{ackn}{Acknowledgment} 

\title[Uniform K-stability]{Uniform K-stability, Duistermaat-Heckman measures and singularities of pairs}
\date{\today}

\author{S{\'e}bastien Boucksom
  \and
  Tomoyuki Hisamoto
  \and 
  Mattias Jonsson}
\address{CNRS-CMLS\\
  \'Ecole Polytechnique\\
  F-91128 Palaiseau Cedex\\
  France}
\email{sebastien.boucksom@polytechnique.edu}
\address{Dept of Mathematics\\
  University of Michigan\\
  Ann Arbor, MI 48109--1043\\
  USA}
\address{Mathematical Sciences\\
  Chalmers University of Technology
  and University of Gothenburg\\
  SE-412 96 G\"oteborg\\
  Sweden}
\email{mattiasj@umich.edu}
\address{Graduate School of Mathematics\\
Nagoya University\\
Furocho\\
Chikusa\\
Nagoya\\ 
Japan}
\email{hisamoto@math.nagoya-u.ac.jp}

\begin{document}

\begin{abstract} The purpose of this paper is to set up a formalism
  inspired by non-Archimedean geometry to study K-stability. We first
  provide a detailed analysis of Duistermaat-Heckman measures in the
  context of test configurations for arbitrary polarized schemes,
  characterizing in particular almost trivial test
  configurations. Second, for any normal polarized variety (or, more generally, polarized pair in the sense of the Minimal Model Program), we introduce and study non-Archimedean analogues of certain classical functionals in K\"ahler geometry. These functionals are defined on the space of test configurations, and the Donaldson-Futaki invariant is in particular interpreted as the non-Archimedean version of the Mabuchi functional, up to an explicit error term. 
Finally, we study in detail the relation between uniform K-stability and singularities of pairs, reproving and strengthening Y.~Odaka's results in our formalism. This provides various examples of uniformly K-stable varieties. 
\end{abstract}

\maketitle

\setcounter{tocdepth}{1}
\tableofcontents
%
%
%
%



\section*{Introduction}
%
%
%
%
Let $(X,L)$ be a polarized complex manifold, \ie a smooth complex projective variety $X$ endowed with an ample line bundle $L$. Assuming for simplicity that the reduced automorphism group $\Aut(X,L)/\C^*$ is discrete (and hence finite), the Yau-Tian-Donaldson conjecture predicts that the first Chern class $c_1(L)$ contains a constant scalar curvature K\"ahler metric (cscK metric for short) iff $(X,L)$ satisfies a certain algebro-geometric condition known as \emph{K-stability}. Building on~\cite{Don1,AP}, it was proved in~\cite{Sto} that K-stability indeed follows from the existence of a cscK metric. When $c_1(X)$ is a multiple of $c_1(L)$, the converse was recently established (\cite{CDS15}, see also~\cite{Tian15}); in this case a cscK metric is the same as a K\"ahler-Einstein metric. 

In the original definition of~\cite{Don2}, $(X,L)$ is K-semistable if the Donaldson-Futaki invariant $\DF(\cX,\cL)$ of every (ample) test configuration $(\cX,\cL)$ for $(X,L)$ is non-negative, and K-stable if we further have $\DF(\cX,\cL)=0$ only when $\cX=X\times\C$ is trivial (and hence $\cL=p_1^*L$ with $\C^*$ acting through a character). However, as pointed out in~\cite{LX}, $(X,L)$ always admits test configurations $(\cX,\cL)$ with $\cX$ non-trivial, but \emph{almost trivial} in the sense that its normalization $\tcX$ is trivial. Such test configurations automatically satisfy $\DF(\cX,\cL)=0$, and the solution adopted in~\cite{Sto2,Oda4} was therefore to replace `trivial' with `almost trivial' in the definition of K-stability. 

On the other hand, G.~Sz\'ekelyhidi~\cite{Sze1,Sze2} proposed that a
\emph{uniform} notion of K-stability should be used to formulate the
Yau-Tian-Donaldson conjecture for general polarizations. In this
uniform version, $\DF(\cX,\cL)$ is bounded below by a positive
multiple of the $L^p$-norm $\|(\cX,\cL)\|_p$. Since uniform
K-stability should of course imply K-stability, one then faces the
problem of showing that test configurations with norm zero are almost trivial. 

In the first part of the paper, we prove that this is
indeed the case. In fact, the $L^p$-norm $\|(\cX,\cL)\|_p$ of a test
configuration $(\cX,\cL)$ can be computed via the 
\emph{Duister\-maat-Heckman measure} $\DH_{(\cX,\cL)}$ associated to the test
configuration. We undertake a quite thorough study of 
Duistermaat-Heckman measures and prove in particular that $\DH_{(\cX,\cL)}$ is a 
Dirac mass iff $(\cX,\cL)$ is almost trivial.

The second main purpose of the paper is to introduce a \emph{non-Archimedean}
perspective on K-stability, in which test configurations for $(X,L)$
are viewed as non-Archimedean metrics on (the Berkovich
analytification with respect to the trivial norm of) 
$L$. We introduce non-Archimedean analogues of many
classical functionals in K\"ahler geometry, and interpret uniform
K-stability as the non-Archimedean counterpart of the coercivity of
the Mabuchi K-energy. 

Finally, in the third part of the paper, we use this formalism to analyze the interaction between singularities of pairs (in the sense of the Minimal Model Program) and uniform K-stability, revisiting Y.~Odaka's work~\cite{Oda1,Oda3,OSa,OSu}. 

\medskip
We now describe the contents of the paper in more detail.
\subsection*{Duistermaat-Heckman measures} 
Working, for the moment, over any arbitrary alge\-braically closed ground field, let $(X,L)$ 
be a polarized scheme, \ie a (possibly non-reduced) scheme $X$ together with an 
ample line bundle $L$ on $X$.
Given a $\G_m$-action on $(X,L)$, let $H^0(X,mL)=\bigoplus_{\la\in\Z}H^0(X,mL)_\la$ be the weight decomposition. For each $d\in\N$, the finite sum $\sum_{\la\in\Z}\la^d \dim H^0(X,mL)_\la$ is a polynomial function of $m\gg 1$, of degree at most $\dim X+d$ (cf.~Theorem~\ref{thm:equivRR}, as well as Appendix B). 

Setting $N_m:=\dim H^0(X,mL)$, we get, as a direct consequence,
the existence of the \emph{Duistermaat-Heckman measure}
$$
\DH_{(X,L)}:=\lim_{m\to\infty}\frac{1}{N_m}\sum_{\la\in\Z}\dim H^0(X,mL)_\la\d_{m^{-1}\la},
$$
a probability measure with compact support in $\R$ describing the asymptotic distribution as $m\to\infty$ of the (scaled) weights of $H^0(X,mL)$, counted with multiplicity. The \emph{Donaldson-Futaki invariant} $\DF(X,L)$ appears in the subdominant term of the expansion
$$
\frac{w_m}{mN_m}=\frac{1}{N_m}\sum_{\la\in\Z}m^{-1}\la\dim H^0(X,mL)_\la=\int_\R\la\,\DH_{(X,L)}(d\la)-(2m)^{-1}\DF(X,L)+O(m^{-2}),
$$
where $w_m$ is the weight of the induced action on the determinant line $\det H^0(X,mL)$. 

\smallskip
Instead of a $\G_m$-action on $(X,L)$, consider more generally a
\emph{test configuration} $(\cX,\cL)$ for $(X,L)$, \ie a
$\G_m$-equivariant partial compactification of
$(X,L)\times(\A^1\setminus\{0\})$. It comes with a proper, flat,
$\G_m$-equivariant morphism $\pi\colon\cX\to\A^1$, together with a
$\G_m$-linearized $\Q$-line bundle $\cL$ extending $p_1^*L$ on
$X\times(\A^1\setminus\{0\})$. When the test configuration is ample,
\ie $\cL$ is $\pi$-ample, the central fiber $(\cX_0,\cL_0)$ is a
polarized $\G_m$-scheme, and the Duistermaat-Heckman measure
$\DH_{(\cX,\cL)}$ and Donaldson-Futaki invariant $\DF(\cX,\cL)$ are
defined to be those of $(\cX_0,\cL_0)$. 

In the previous case, where
$(X,L)$ comes with a $\G_m$-action, the Duistermaat-Heckman measure
and Donaldson-Futaki as defined above coincide with those of the
corresponding \emph{product} test configuration $(X,L)\times\A^1$ with the
diagonal action of $\G_m$. 
Such a test configuration is called \emph{trivial} if the action on
$X$ is trivial. 

Our first main result may be summarized as follows. 
\begin{thmA} Let $(X,L)$ be a polarized scheme and $(\cX,\cL)$ an ample test configuration for $(X,L)$, with Duistermaat-Heckman measure $\DH_{(\cX,\cL)}$. 
\begin{itemize}
\item[(i)] The absolutely continuous part of $\DH_{(\cX,\cL)}$ has piecewise polynomial density, and its singular part is a finite sum of point masses.
\item[(ii)] The measure $\DH_{(\cX,\cL)}$ is a finite sum of point masses iff $(\cX,\cL)$ is \emph{almost trivial} in the sense that the normalization of each top-dimensional irreducible component of $\cX$ is trivial. 
\end{itemize}  
\end{thmA} 
The piecewise polynomiality in (i) generalizes a well-known property of
Duistermaat-Heckman measures for polarized complex manifolds with a
$\C^*$-action~\cite{DH}. In (ii), the normalization of $\cX$ is viewed
as a test configuration for the normalization of $X$. The notion of
almost triviality is compatible with the one introduced
in~\cite{Sto2,Oda4} for $X$ reduced and equidimensional, cf.~Proposition~\ref{prop:almosttriv}. 

In Theorem~A, $X$ is a possibly non-reduced scheme.
If we specialize to the case when $X$ is a (reduced, irreducible)
variety, Theorem~A and its proof yield the following characterization
of almost trivial test configurations:
\begin{corB}
  Let $(X,L)$ be a polarized variety and $(\cX,\cL)$ an ample test
  configuration for $(X,L)$, with Duistermaat-Heckman measure
  $\DH_{(\cX,\cL)}$. Then the following conditions are equivalent:
  \begin{itemize}
  \item[(i)]
    the Duistermaat-Heckman measure $\DH_{(\cX,\cL)}$ is a Dirac mass;
  \item[(ii)]
    for some (or, equivalently, any) $p\in[1,\infty]$, we have
    $\|(\cX,\cL)\|_p=0$;
  \item[(iii)]
    $(\cX,\cL)$ is almost trivial, that is, the normalization 
    $(\tcX,\tcL)$ is trivial.
  \end{itemize}
\end{corB}
Here the $L^p$-norm $\|(\cX,\cL)\|_p$ is defined,
following~\cite{Don3,WN12,His12}, as
the $L^p$ norm of $\la\mapsto\la-\bar\la$ with respect to $\DH_{(\cX,\cL)}$,
where $\bar\la$ is the barycenter of this measure.
%
%
%
\subsection*{Uniform K-stability and non-Archimedean functionals}
A polarized scheme $(X,L)$ is \emph{K-semistable} if $\DF(\cX,\cL)\ge 0$ for
each ample test configuration. It is \emph{K-stable} if, furthermore, $\DF(\cX,\cL)=0$ only when $(\cX,\cL)$ is almost trivial, in the sense of Theorem~A~(ii). 

Assume from now on that $X$ is irreducible and normal. 
By Corollary~B, the almost triviality of an ample test configuration 
can then be detected by the $L^p$-norm $\|(\cX,\cL)\|_p$ with
$p\in[1,\infty]$.
We say that $(X,L)$ is \emph{$L^p$-uniformly K-stable} if $\DF(\cX,\cL)\ge\d\|(\cX,\cL)\|_p$ for some uniform constant $\d>0$. For $p=1$, we simply speak of \emph{uniform K-stability}, which is therefore implied by $L^p$-uniform K-stability since $\|(\cX,\cL)\|_p\ge\|(\cX,\cL)\|_1$. 

These notions also apply when $((X,B);L)$ is a \emph{polarized pair}, consisting of a normal polarized variety and a $\Q$-Weil divisor on $X$ such that $K_{(X,B)}:=K_X+B$ is $\Q$-Cartier, using the \emph{log Donaldson-Futaki invariant} $\DF_B(\cX,\cL)$ of a test configuration $(\cX,\cL)$. We show that $L^p$-uniformly K-stability can in fact only hold for $p\le \tfrac{n}{n-1}$ (cf.~Proposition~\ref{prop:thresh}). 

One of the points of the present paper is to show that ($L^1$-)uniform K-stability of polarized pairs can be understood in terms of the non-Archimedean counterparts of well-known functionals in K\"ahler geometry. In order to achieve this, we interpret a test configuration for $(X,L)$ as a \emph{non-Archimedean metric} on the Berkovich analytification of $L$ with respect to the \emph{trivial} norm on the ground field, see~\S\ref{sec:NAmetrics}. In this language,
ample test configurations become positive metrics.

Several classical functionals on the space of Hermitian metrics in K\"ahler geometry have natural counterparts in the non-Archimedean setting. 
For example, the \emph{non-Archimedean Monge-Amp\`ere energy} is
$$
E^{\NA}(\cX,\cL)=\frac{(\bar\cL^{n+1})}{(n+1)V}=\int_\R\la\,\DH_{(\cX,\cL)}(d\la),
$$
where $V=(L^n)$, $(\bar\cX,\bar\cL)$ is the natural $\G_m$-equivariant compactification of $(\cX,\cL)$ over $\P^1$ and $\DH_{(\cX,\cL)}$ is the Duistermaat-Heckman measure of $(\cX,\cL)$.
The \emph{non-Archimedean J-energy} is 
$$
J^{\NA}(\cX,\cL)
=\la_{\max}-E^{\NA}(\cX,\cL)
=\la_{\max}-\int_\R\la\,\DH_{(\cX,\cL)}(d\la)\ge 0,
$$
with $\la_{\max}$ the upper bound of the support of
$\DH_{(\cX,\cL)}$. We show that this quantity is equivalent to the
$L^1$-norm in the following sense:
$$
c_n J^\NA(\cX,\cL)\le\|(\cX,\cL)\|_1\le 2 J^\NA(\cX,\cL)
$$
for some numerical constant $c_n>0$.  

Given a boundary $B$, we define the \emph{non-Archimedean Ricci
  energy} $R_B^{\NA}(\cX,\cL)$ in terms of intersection numbers on a
suitable test configuration dominating $(\cX,\cL)$. The \emph{non-Archimedean entropy} $H_B^{\NA}(\cX,\cL)$ is defined in terms of the log discrepancies with respect to $(X,B)$ of certain divisorial valuations, and will be described in more detail below. 

The \emph{non-Archimedean Mabuchi functional} is now defined so as to satisfy the analogue of the \emph{Chen-Tian formula} (see~\cite{Che2} and also~\cite[Proposition 3.1]{BB})
$$
M_B^{\NA}(\cX,\cL)=H_B^{\NA}(\cX,\cL) +R_B^{\NA}(\cX,\cL)+\bar S_B E^{\NA}(\cX,\cL)
$$
with 
$$
\bar S_B:=-nV^{-1}\left(K_{(X,B)}\cdot L^{n-1}\right), 
$$ 
which, for $X$ smooth over $\C$ and $B=0$, gives the mean value of the scalar curvature of any K\"ahler metric in $c_1(L)$. The whole point of these constructions is that $M_B^{\NA}$ is essentially the same as the log Donaldson-Futaki invariant.\footnote{The interpretation of the Donaldson-Futaki invariant as a non-Archimedean Mabuchi functional has been known to Shou-Wu Zhang for quite some time, cf.~\cite[Remark 6]{PRS}.} 
We show more precisely that every normal, ample test configuration $(\cX,\cL)$ satisfies
\begin{equation}\label{equ:Mslope}
\DF_B(\cX,\cL)=M_B^{\NA}(\cX,\cL)+V^{-1}\left((\cX_0-\cX_{0,\red})\cdot\cL^n\right). 
\end{equation}
Further, $M^{\NA}_B$ is homogeneous with respect to $\G_m$-equivariant 
base change, a property which is particularly useful in relation with semistable reduction, and fails for the Donaldson-Futaki invariant when the central fiber is non-reduced. Using this, we show that uniform K-stability of $((X,B);L)$ is equivalent to the apparently stronger condition $M_B^{\NA}\ge\d J^{\NA}$, which we interpret as a counterpart to the \emph{coercivity} of the Mabuchi energy in K\"ahler geometry~\cite{Tian97}. 

\medskip
The relation between the non-Archimedean functionals above and their classical
counterparts will be systematically studied in~\cite{BHJ2}. Let us indicate the main idea.
Assume $(X,L)$ is a smooth polarized complex variety, and $B=0$.
Denote by $\cH$ the space of K\"ahler metrics on $L$
and by $\cH^{\NA}$ the space of non-Archimedean metrics. 
The general idea is that $\cH^{\NA}$ plays the role of the `Tits boundary' of the (infinite dimensional) symmetric space $\cH$. Given an ample test configuration $(\cX,\cL)$ (viewed as an element of $\cH^{\NA}$) and a smooth ray $(\phi_s)_{s\in(0,+\infty)}$ corresponding to a smooth $S^1$-invariant metric on $\cL$, we shall prove in~\cite{BHJ2} that 
\begin{equation}\label{e301}
\lim_{s\to+\infty}\frac{F(\phi_s)}{s}=F^{\NA}(\cX,\cL),
\end{equation}
where $F$ denotes the Monge-Amp\`ere energy, $J$-energy, entropy, or 
Mabuchi energy functional and $F^{\NA}$ is the corresponding non-Archimedean functional defined above.
In the case of the Mabuchi energy, this result is closely related to~\cite{PT1,PT2,PRS}.
%
%
%
%
\subsection*{Singularities of pairs and uniform K-stability}
A key point in our approach to K-stability is to relate the birational geometry of $X$ and that of its test configurations using the language of \emph{valuations}. 

More specifically, let $(X,L)$ be a normal polarized variety, and
$(\cX,\cL)$ a normal test configuration. Every irreducible component
$E$ of $\cX_0$ defines a divisorial valuation $\ord_E$ on the function
field of $\cX$. Since the latter is canonically isomorphic to
$k(X\times\A^1)\simeq k(X)(t)$, we may consider the restriction
$r(\ord_E)$ of $\ord_E$ to $k(X)$; this is proved to be a divisorial valuation as well when $E$ is non-trivial, \ie not the strict transform of the central fiber of the trivial test configuration. 

This correspondence between irreducible components of $\cX_0$ and divisorial valuations on $X$ is analyzed in detail in~\S\ref{sec:valtest}. In particular, we prove that the Rees valuations of a closed subscheme $Z\subset X$, \ie the divisorial valuations associated to the normalized blow-up of $X$ along $Z$, coincide with the valuations induced on $X$ by the normalization of the deformation to the normal cone of $Z$. 

Given a boundary $B$ on $X$, we define the non-Archimedean entropy of a normal test configuration $(\cX,\cL)$ mentioned above as
$$
H^{\NA}_B(\cX,\cL)=V^{-1}\sum_E A_{(X,B)}(r(\ord_E))(E\cdot\cL^n),
$$
the sum running over the non-trivial irreducible components of $\cX_0$ and $A_{(X,B)}(v)$ denoting the log discrepancy of a divisorial valuation $v$ with respect the pair $(X,B)$. Recall that the pair $(X,B)$ is log canonical (lc for short) if $A_{(X,B)}(v)\ge 0$ for all divisorial valuations on $X$, and Kawamata log terminal (klt for short) if the inequality is everywhere strict. Our main result here is a characterization of these singularity classes in terms of the non-Archimedean entropy functional. 

\begin{thmC} Let $(X,L)$ be a normal polarized variety, and $B$ an effective boundary on $X$. Then 
$(X,B)$ is lc (resp.\ klt) iff $H_B^{\NA}(\cX,\cL)\ge 0$ (resp.\ $>0$) for every non-trivial normal, ample test configuration $(\cX,\cL)$. In the klt case, there automatically exists $\d>0$ such that $H_B^{\NA}(\cX,\cL)\ge\d J^{\NA}(\cX,\cL)$ for all $(\cX,\cL)$. 
\end{thmC}
The strategy to prove the first two points is closely related to that of~\cite{Oda3}. In fact, we also provide a complete proof of the following mild generalization (in the normal case) of the main result of~\textit{loc.cit}:
$$
((X,B);L)\text{ K-semistable }\Longrightarrow (X,B)\text{ lc}.  
$$
The non-normal case is discussed in \S\ref{sec:slc}. If $(X,B)$ is not lc (resp.\ not klt), then known results from the Minimal Model Program allow us to construct a closed subscheme $Z$ whose Rees valuations have negative (resp.\ non-positive) discrepancies; the normalization of the deformation to the normal cone of $Z$ then provides a test configuration $(\cX,\cL)$ with $H^{\NA}_B(\cX,\cL)<0$ (resp.\ $\le 0$). 
To prove uniformity in the klt case, we exploit the 
the strict positivity of the \emph{global log canonical threshold} $\lct((X,B);L)$ of $((X,B);L)$. 

As a consequence, we are able to analyze uniform K-stability in the 
\emph{log K\"ahler-Einstein case}, \ie when $K_{(X,B)}$ is numerically proportional to $L$. 

\begin{corD} Let $(X,L)$ be a normal polarized variety, $B$ an effective boundary, and assume that $K_{(X,B)}\equiv\la L$ with $\la\in\Q$. 
\begin{itemize}
\item[(i)] If $\la>0$, then $((X,B);L)$ is uniformly K-stable iff $(X,B)$ is lc;
\item[(ii)] If $\la=0$, then $((X,B);L)$ is uniformly K-stable iff $(X,B)$ is klt; 
\item[(iii)] If $\la<0$ and $\lct((X,B);L)>\frac{n}{n+1}|\la|$, then $((X,B);L)$ is uniformly K-stable. 
\end{itemize}
\end{corD}
This result thus gives `uniform versions' of~\cite{Oda1,OSa}. 

In the last case, when $-K_{(X,B)}$ is ample, we also prove that uniform
K-stability is equivalent to \emph{uniform Ding stability},
defined as $D^\NA_B\ge\delta J^\NA$, where $D^\NA$ is the
\emph{non-Archimedean Ding functional} that appeared in the work 
of Berman~\cite{Berm16}; see also~\cite{BBJ15,Fuj15b,Fuj16}.
%
%
\subsection*{Relation to other works} 
Since we aim to give a systematic introduction to uniform K-stability, and to set
up some non-Archimedean terminology, we have tried to make the exposition 
as self-contained as possible. This means that we reprove or slightly generalize some 
already known results~\cite{Oda1,Oda3,OSa,Sun,OSu}. 

During the preparation of the present paper, we were informed of
R.~Dervan's independent work~\cite{Der1}(see also~\cite{Der2}), which
has a substantial overlap with the present paper. First, when $X$ is
normal, ample test configurations with trivial norm were also
characterized in~\cite[Theorem 1.3]{Der1}. Next, the \emph{minimum
  norm} introduced in~\textit{loc.cit} turns out to be equivalent to
our non-Archimedean J-functional, up to multiplicative constants
(cf.~Remark~\ref{rmk:derJ}). As a result, \emph{uniform K-stability
  with respect to the minimum norm} as in~\cite{Der1} is the same as
our concept of uniform K-stability. Finally, Corollary C above is to a
large extent contained in~\cite[\S3]{Der1} and~\cite{Der2}. 

Several papers exploring K-stability through valuations have
appeared since the first version of this paper. We mention in
particular~\cite{Fuj15a,Fuj15b,Fuj16,FO16,Li15,Li16,Liu16}.
%
%
\subsection*{Structure of the paper}
Section~\ref{sec:prelim} gathers a number of preliminary facts on filtrations and valuations, with a special emphasis on the Rees construction and the relation between Rees valuations and integral closure. 

In Section~\ref{sec:test} we provide a number of elementary facts on test
configurations, and discus in particular some scheme theoretic aspects. 

Section~\ref{sec:DHDF} gives a fairly self-contained treatment of Duistermaat-Heckman measures and Donaldson-Futaki invariants in the context of polarized schemes. The existence of asymptotic expansions for power sums of weights is established in Theorem~\ref{thm:equivRR}, following an idea of Donaldson. 

The correspondence between irreducible components of the central fiber of a normal test configuration and divisorial valuations on $X$ is considered in Section~\ref{sec:valtest}.
In particular, Theorem~\ref{thm:reescone} relates Rees valuations and the deformation to the normal cone. 

Section~\ref{sec:DHfiltr} contains an in-depth study of
Duistermaat-Heckman measures in the normal case, leading to the proof
of Theorem~A and Corollary~B. 
 
Certain non-Archimedean metrics on $L$ are introduced in
Section~\ref{sec:NAmetrics} as equivalence classes of test
configurations. This is inspired by~\cite{siminag,nama,simons}.

In Section~\ref{S201} we introduce non-Archimedean analogues of
the usual energy functionals and in Section~\ref{sec:Kstab} we use these to 
define and study uniform K-stability. In the Fano case, we relate
(uniform) K-stability to the notion of (uniform) Ding stability.

Section~\ref{sec:kstabsing} is concerned with the interaction between
uniform K-stability and singularities of pairs. Specifically, Theorem~\ref{thm:lc} and Theorem~\ref{thm:klt} establish Theorem~C as well as the generalization of~\cite{Oda3} mentioned above. Corollary~D is a combination of Corollary~\ref{cor:canpol}, Corollary~\ref{cor:CY} and Proposition~\ref{prop:alpha}.

Finally, Appendix A provides a proof of the two-term Riemann-Roch theorem on a normal variety, whose complete proof we could not locate in the literature, and Appendix B summarizes Edidin and Graham's equivariant version of the Riemann-Roch theorem for schemes, yielding an alternative proof of Theorem~\ref{thm:equivRR}. 

\begin{ackn} The authors would like to thank Robert Berman, Michel Brion, Ama\"el Broustet, Kento Fujita, Vincent Guedj, Marco Maculan, Mircea Musta\c{t}\u{a}, Yuji Odaka and Ahmed Zeriahi for helpful conversations. We are also grateful to the anonymous referee for several very useful comments. S.B. was partially supported by the ANR projects MACK and POSITIVE\@. T.H. was supported by JSPS Research Fellowships for Young Scientists (25-6660).
M.J. was partially supported by NSF grant DMS-1266207,
the Knut and Alice Wallenberg foundation,
and the United States---Israel Binational Science Foundation.
\end{ackn}

%
%
%
%
\section{Preliminary facts on filtrations and valuations}\label{sec:prelim}
We work over an algebraically closed field $k$, whose characteristic
is arbitrary unless otherwise specified. Write $\G_m$ for the
multiplicative group over $k$ and $\A^1=\Spec k[t]$ for the affine
line. The \emph{trivial absolute value} $|\cdot|_0$ on $k$ is defined
by $|0|_0=0$ and $|c|_0=1$ for $c\in k^*$.

All schemes are assumed to be separated and of finite type over $k$. We restrict the use of \emph{variety} to denote a reduced and irreducible scheme. A reduced scheme is thus a finite union of varieties, and a normal scheme is a disjoint union of normal varieties. 

By an \emph{ideal} on a scheme $X$ we mean a coherent ideal sheaf, whereas
a \emph{fractional ideal} is a coherent $\cO_X$-submodule 
of the sheaf of rational functions. 

If $X$ is a scheme and $L$ a line bundle on $X$, then $\G_m$-action on
$(X,L)$ means a $\G_m$-action on $X$ together with a
$\G_m$-linearization of $L$. This induces an action on 
$(X,rL)$ for any $r\in\Z_{>0}$. If $L$ is a $\Q$-line bundle on $X$,
then a $\G_m$-action on $(X,L)$ means a 
compatible family of actions on $(X,rL)$ for all 
sufficiently divisible $r\in\Z_{>0}$.

A polarized scheme (resp.\ variety) is a pair $(X,L)$ where $X$ is a
projective scheme (resp.\ variety) and $L$ is an ample $\Q$-line
bundle on $X$.
%
%
%
%
\subsection{Norms and filtrations}\label{sec:filtr}   
Let $V$ be a finite dimensional $k$-vector space. In this paper, a \emph{filtration} of $V$ will mean a decreasing, left-continuous, separating and exhaustive $\R$-indexed filtration $F^\bullet V$. In other words, it is a family of subspaces $(F^\lambda V)_{\la\in\R}$ of $V$ such that
\begin{itemize}
\item[(i)] $F^\lambda V\subset F^{\lambda'}V$ when $\la\ge\la'$; 
\item[(ii)] $F^\la V=\bigcap_{\la'<\la} F^{\la'}V$; 
\item[(iii)] $F^\la V=0$ for $\la\gg 0$; 
\item[(iv)] $F^\la V=V$ for $\la\ll 0$.
\end{itemize}
A \emph{$\Z$-filtration} is a filtration $F^\bullet V$ such that $F^\la  V=F^{\lceil\la\rceil} V$ for $\la\in\R$. Equivalently, it is a family of subspaces $(F^\la  V)_{\la\in\Z}$ satisfying (i), (iii) and (iv) above. 
  
With these conventions, filtrations are in one-to-one correspondence with non-Archimedean
norms on $V$ compatible with the trivial absolute value on $k$,
\ie functions $\|\cdot\|\colon V\to\R_+$ such that
\begin{itemize}
\item[(i)] $\|s+s'\|\le\max\left\{\|s\|,\|s'\|\right\}$ for all $s,s'\in V$;
\item[(ii)] $\|c s\|=|c|_0\cdot\|s\|=\|s\|$ for all $s\in V$ and $c\in k^*$;
\item[(iii)] $\|s\|=0\Longleftrightarrow s=0$. 
\end{itemize}
The correspondence is given by
\begin{equation*}
  -\log \|s\|=\sup\left\{\la\in\R\mid s\in F^\la  V\right\}
  \quad\text{and}\quad
  F^\la  V=\left\{s\in V\mid \|s\|\le e^{-\la}\right\}.
\end{equation*}
The \emph{successive minima} of the filtration $F^\bullet V$ is the decreasing sequence 
$$
\la_{\max}=\la_1\ge\dots\ge\la_N=\la_{\min} 
$$ 
where $N=\dim V$, defined by
$$
\la_j=\max\left\{\la\in\R\mid\dim F^\la  V\ge j\right\}.
$$
From the point of view of the norm, they are indeed the analogues of
the (logarithmic) successive minima in Minkowski's geometry of
numbers. Choosing a basis $(s_j)$ compatible with the flag $F^{\la_1} V\subset\dots\subset F^{\la_N}V$ diagonalizes the associated norm $\|\cdot\|$, in the sense that
$$
\|\sum c_i s_i\|=\max|c_i|_0 e^{-\la_i}.  
$$

\medskip

Next let $R:=\bigoplus_{m\in\N} R_m$ be a graded $k$-algebra with finite dimensional graded pieces $R_m$. A filtration $F^\bullet R$ of $R$ is defined as the data of a filtration $F^\bullet R_m$ for each $m$, satisfying
$$
F^\la  R_m\cdot F^{\la'}R_{m'}\subset F^{\la+\la'}R_{m+m'}
$$
for all $\la,\la'\in\R$ and $m,m'\in\N$. The data of $F^\bullet R$ is equivalent to the data of a non-Archimedean submultiplicative norm $\|\cdot\|$ on $R$, \ie a non-Archimedean norm $\|\cdot\|_m$ as above on each $R_m$, satisfying
$$
\|s\cdot s'\|_{m+m'}\le\|s\|_m\|s'\|_{m'}
$$
for all $s\in R_m$, $s'\in R_{m'}$. We will use the following terminology. 

\begin{defi}\label{defi:fingen} We say that a $\Z$-filtration $F^\bullet R$ of a graded algebra $R$ is \emph{finitely generated} if the bigraded algebra 
$$
\bigoplus_{(\la,m)\in\Z\times\N} F^\la  R_m
$$ 
is finitely generated over $k$. 
\end{defi}
The condition equivalently means that the graded $k[t]$-algebra
$$
\bigoplus_{m\in\N}\left(\bigoplus_{\la\in\Z}t^{-\la}F^\la  R_m\right)
$$
is finitely generated. 
%
%
%
\subsection{The Rees construction}\label{sec:reesfiltr}   
We review here a classical construction due to Rees, which yields a geometric interpretation of $\Z$-filtrations. 

Start with a $\G_m$-linearized vector bundle $\cV$ on $\A^1$, and set $V=\cV_1$. The weight decomposition
$$
H^0(\A^1,\cV)=\bigoplus_{\la\in\Z} H^0(\A^1,\cV)_\la
$$
yields a $\Z$-filtration $F^\bullet V$, with $F^\la  V$ defined as the image of the weight-$\la$ part of $H^0(\A^1,\cV)$ under the restriction map $H^0(\A^1,\cV)\to V$. Since $t$ has weight $-1$ with respect to the $\G_m$-action on $\A^1$, multiplication by $t$ induces an injection $F^{\la+1} V\subset F^\la  V$, so that this is indeed a decreasing filtration. 

Conversely, consider a $\Z$-filtration $F^\bullet V$ of a $k$-vector space $V$. Then $\bigoplus_{\la\in\Z} t^{-\la}F^\la  V$ is a torsion free, finitely generated $k[t]$-module. It can thus be written as the space of global sections of a unique vector bundle $\cV$ on $\A^1=\Spec k[t]$. The grading provides a $\G_m$-linearization of $\cV$, and the corresponding weight spaces are given by $H^0(\A^1,\cV)_\la\simeq t^{-\la} F^\la  V$. 

\begin{lem}\label{lem:graded} In the above notation, we have a $\G_m$-equivariant vector bundle isomorphism,  
\begin{equation}\label{equ:reesgen}
\cV|_{\A^1\setminus\{0\}}\simeq V\times\left(\A^1\setminus\{0\}\right)
\end{equation}
as well as
\begin{equation}\label{equ:reessp}
\cV_0\simeq\mathrm{Gr}^F_\bullet V=\bigoplus_{\la\in\Z}F^\la  V/F^{\la+1}V. 
\end{equation}
\end{lem}
Intuitively, this says that $\cV$ may be thought of as a way to degenerate the filtration to its graded object. 
 
\begin{proof} To see that (\ref{equ:reesgen}) holds, consider the $k$-linear map 
$\pi\colon H^0(\A^1,\cV)\to V$ sending $\sum_\la t^{-\la}v_\la$ to $\sum_\la v_\la$.
This map is surjective since $F^\la  V=V$ for $\la\ll0$. 
If $\sum_\la t^{-\la}v_\la$ lies in the kernel, then $v_\la=w_{\la+1}-w_\la$ for 
all $\la$, where $w_\la=-\sum_{\mu\ge\la}v_\mu\in F^\la  V$. 
Conversely, any element of the form $\sum_\la t^{-\la}(w_{\la+1}-w_\la)$,
where $w_\la\in F^\la V$, is in the kernel of $\pi$, and the set
of such elements is equal to $(t-1)H^0(\A^1,\cV)$.
Thus $\pi$ induces an isomorphism between 
$\cV_1=H^0(\A^1,\cV)/(t-1)H^0(\A^1,\cV)$ and $V$,
which induces~\eqref{equ:reesgen} using the $\G_m$-action. 
The proof of (\ref{equ:reessp}) is similar. 
\end{proof}

Using this, it is easy to verify that the two constructions above are inverse to each other, and actually define an equivalence of categories between $\Z$-filtered, finite dimensional vector spaces $F^\bullet V$ and $\G_m$-linearized vector bundles $\cV$ on $\A^1$, related by the $\G_m$-equivariant isomorphism
$$
H^0(\A^1,\cV)\simeq\bigoplus_{\la\in\Z} t^{-\la} F^\la  V.
$$

Every filtered vector space admits a basis compatible with the filtration, and is thus (non-canonically) isomorphic to its graded object. On the vector bundle side, this yields (compare~\cite[Lemma 2]{Don3}):

\begin{prop}\label{prop:reesfiltr} Every $\G_m$-linearized vector bundle $\cV$  on $\A^1$ is $\G_m$-equivariantly trivial, \ie $\G_m$-isomorphic to $\cV_0\times\A^1$ with $\cV_0$ the fiber at $0$.  
\end{prop}
For line bundles, the trivialization admits the following particularly simple description. 
\begin{cor}\label{cor:trivline} Let $\cL$ be a $\G_m$-linearized line bundle on $\A^1$, and let $\la\in\Z$ be the weight of the $\G_m$-action on $\cL_0$. For each non-zero $v\in\cL_1$, setting $s(t):=t^{-\la}(t\cdot v)$ defines a weight-$\la$ trivialization of $\cL$. 
\end{cor}
\begin{proof} While this is a special case of the above construction, it can be directly checked as follows. The section $s'\in H^0(\A^1\setminus\{0\},\cL)$ defined by $s'(t):=t\cdot v$ defines a rational section of $\cL$. If we set $\mu:=\ord_0(s')$, then $v_0:=\lim_{z\to 0}z^{-\mu} s'(z)$ is a non-zero element of $\cL_0$, which satisfies
$$
t\cdot v_0=\lim_{z\to 0}z^{-\mu}\left((t z)\cdot v\right)=t^{\mu}\lim_{z\to 0}(t z)^{-\mu}\left((t z)\cdot v\right)=t^{\mu}v_0.
$$
It follows that $\mu$ coincides with the weight $\la$ of the $\G_m$-action on $\cL_0$. 
\end{proof}

We introduce the following piece of terminology. 
\begin{defi}\label{defi:weight} Let $W=\bigoplus_{\la\in\Z} W_\la$ be the weight decomposition of a $\G_m$-module. The \emph{weight measure} of $W$ is defined as the probability measure 
$$
\mu_W:=\frac{1}{\dim W}\sum_{\la\in\Z}(\dim W_\la)\d_\la.
$$
\end{defi}

For later use, we record the following immediate consequence of (\ref{equ:reessp}). 

\begin{lem}\label{lem:tail} Let $\cV$ be a $\G_m$-linearized vector bundle over $\A^1$, and $F^\bullet V$ the corresponding $\Z$-filtration of the fiber $V=\cV_1$. The weight measure $\mu_{\cV_0}$ of the $\G_m$-module $\cV_0$ then satisfies
$$
\mu_{\cV_0}\{x\ge\la\}=\frac{\dim F^{\lceil\la\rceil} V}{\dim V}
$$
for all $\la\in\R$. 
\end{lem}
%
%
%
%
\subsection{Valuations}\label{sec:val}
Let $K$ be a finitely generated field extension of $k$, with $n:=\trdeg K/k$, so that $K$ may be realized as the function field of a (normal, projective) $n$-dimensional variety. 

Since we only consider real-valued valuations, we simply call \emph{valuation} $v$ on $K$ a group homomorphism $v:K^*\to(\R,+)$ such that $v(f+g)\ge\min\left\{v(f),v(g)\right\}$ and $v|_{k^*}\equiv 0$~\cite{ZS}. It is convenient to set $v(0)=+\infty$. The \emph{trivial valuation} $v_\triv$ is defined by $v_\triv(f)=0$ for all $f\in K^*$. To each valuation $v$ is attached the following list of invariants. The \emph{valuation ring} of $v$ is 
$\cO_v:=\left\{f\in K\mid v(f)\ge 0\right\}$.
This is a local ring with maximal ideal 
$\fm_v:=\left\{f\in K\mid v(f)>0\right\}$,
and the \emph{residue field} of $v$ is $k(v):=\cO_v/\fm_v$.
The \emph{transcendence degree} of $v$ (over $k$) is 
$\trdeg(v):=\trdeg k(v)/k$.
Finally, the \emph{value group} of $v$ is $\Ga_v:=v(K^*)\subset\R$,
and the \emph{rational rank} of $v$ is 
$\ratrk(v):=\dim_\Q\left(\Ga_v\otimes\Q\right)$.

If $k\subset K'\subset K$ is an intermediate field extension, $v$ is a valuation on $K$ and $v'$ is its restriction to $K'$,  the Abhyankar-Zariski inequality states that
\begin{equation}\label{equ:Abhy}
\trdeg(v)+\ratrk(v)\le\trdeg(v')+\ratrk(v')+\trdeg K/K'.
\end{equation}
Taking $K'=k$, we get $\trdeg(v)+\ratrk(v)\le n$, 
and we say that $v$ is an \emph{Abhyankar valuation} if equality holds;
such valuations can be geometrically characterized, see~\cite{ELS,KK05,JM}.
In particular, the trivial valuation is Abhyankar; it is the unique valuation 
with transcendence degree $n$. We say that $v$ is \emph{divisorial} if $\ratrk(v)=1$
and $\trdeg(v)=n-1$. 
By a theorem of Zariski, this is the case iff there exists a normal projective variety $Y$ with $k(Y)=K$ and a prime divisor $F$ of $Y$ such that $v=c\ord_F$ for some $c>0$. We then have $k(v)=k(F)$ and $\Ga_v=c\Z$. 

If $X$ is a variety with $k(X)=K$, a valuation $v$ is \emph{centered on $X$} if there exists a scheme point $\xi\in X$ such that $v\ge 0$ on the local ring $\cO_{X,\xi}$ and $v>0$ on its maximal ideal. We also say $v$ is a valuation on $X$ in this case.
By the valuative criterion of separatedness, the point $\xi$ is unique, and is called the \emph{center} of $v$ on $X$. If $X$ is proper, the valuative criterion of properness guarantees that any $v$ is centered on $X$. If a divisorial valuation $v$ is centered on $X$, then $v=c\ord_F$ where $F$ is a prime divisor on a normal variety $Y$ with a proper birational morphism $\mu\colon Y\to X$; the center of $v$ on $X$ is then the generic point of $\mu(F)$.

For any valuation $v$ centered on $X$, we can make sense of $v(s)\in\R_+$ for a (non-zero) section $s\in H^0(X,L)$ of a line bundle $L$ on $X$, by trivializing $L$ at the center $\xi$ of $v$ on $X$
and evaluating $v$ on the local function corresponding to $s$ in this
trivialization. Since any two such trivializations differ by a unit at
$\xi$, $v(s)$ is well-defined, and $v(s)>0$ iff $s(\xi)=0$. 

Similarly, given an ideal $\fa\subset\cO_X$ we set
$$
v(\fa)=\inf\{v(f)\mid f\in\fa_\xi\}. 
$$
It is in fact enough to take the min over any finite set of generators of $\fa_\xi$. 
We also set $v(Z):=v(\fa)$, where $Z$ is the closed subscheme defined by $\fa$.

\medskip
Finally, for later use we record the following simple variant of~\cite[Theorem 10.1.6]{HS}. 

\begin{lem}\label{lem:irredundant} Assume that $X=\Spec A$ is affine. Let $S$ be a finite set of valuations on $X$, which is irredundant in the sense that for each $v\in S$ there exists $f\in A$ with $v(f)<v'(f)$ for all $v'\in S\setminus\{v\}$. Then $S$ is uniquely determined by the function $h_S(f):=\min_{v\in S} v(f)$. 
\end{lem}
\begin{proof} Let $S$ and $T$ be two irredundant finite sets of valuations with $h_S=h_T=:h$. For each $v\in S$, $w\in T$ set $C_v:=\left\{f\in A\mid h(f)=v(f)\right\}$ and $D_w:=\left\{f\in A\mid h(f)=w(f)\right\}$, and observe that these sets are stable under finite products. For each $v\in S$, we claim that there exists $w\in T$ with $C_v\subset D_w$. Otherwise, for each $w$ there exists $f_w\in C_v\setminus D_w$, \ie $v(f_w)=h(f_w)<w(f_w)$. Setting $f=\prod_w f_w$, we get for each $w'\in T$
$$
w'(f)=\sum_{w\in T} w'(f_w)>\sum_{w\in S} h(f_w)=\sum_{w\in S} v(f_w)=v(f)\ge h(f), 
$$
and taking the min over $w'\in T$ yields a contradiction. 

We next claim that $C_v\subset D_w$ implies that $v=w$. This will prove that $S\subset T$, and hence $S=T$ by symmetry. Note first that $v(f)=h(f)=w(f)$ for each $f\in C_v$. Now choose $g_v\in A$ with $v(g_v)<v'(g_v)$ for all $v'\ne v$ in $S$, so that $g_v\in C_v\subset D_w$. For each $f\in A$,  we then have $v(g_v^m f)<v'(g_v^m f)$ for $m\gg 1$, and hence $g_v^m f\in C_v\subset D_w$. It follows that 
$$
m v(g_v)+v(f)=v(g_v^m f)=w(g_v^m f)=m w(g_v)+w(f)=m v(g_v)+w(f), 
$$
and hence $v(f)=w(f)$. 
\end{proof} 
%
%
%
%
\subsection{Integral closure and Rees valuations}\label{sec:rees}
We assume in this section that $X$ is a normal variety. Let $Z\subset X$ a closed subscheme with ideal $\fa\subset\cO_X$. On the one hand,  the \emph{normalized blow-up} $\pi\colon\tX\to X$ along $Z$ is the composition of the blow-up of $Z$ in $X$ with the normalization morphism. On the other hand, the \emph{integral closure} $\overline\fa$ of $\fa$ is the set of elements $f\in\cO_X$ satisfying a monic equation $f^d+a_1f^{d-1}+\dots+a_d=0$ with $a_j\in\fa^j$. 

The following well-known connection between normalized blow-ups and integral closures shows in particular that $\overline\fa$ is a coherent ideal sheaf. 

\begin{lem}\label{lem:normblow} Let $Z\subset X$ be a closed subscheme, with ideal $\fa\subset\cO_X$, and let $\pi\colon\tX\to X$ be the normalized blow-up along $Z$. Then $D:=\pi^{-1}(Z)$ is an effective Cartier divisor with $-D$ $\pi$-ample, and we have for each $m\in\N$:
\begin{itemize}
\item[(i)] $\cO_{\tX}(-mD)$ is $\pi$-globally generated; 
\item[(ii)] $\pi_*\cO_{\tX}(-mD)=\overline{\fa^m}$; 
\item[(iii)] $\cO_{\tX}(-mD)=\cO_{\tX}\cdot\overline{\fa^m}=\cO_{\tX}\cdot\fa^m$; 
\end{itemize}
In particular, $\pi$ coincides with the normalized blow-up of $\overline{\fa}$, and also with the (usual) blow-up of $\overline{\fa^m}$ for any $m\gg 1$. 
\end{lem}
We recall the brief argument for the convenience of the reader. 

\begin{proof} Let $\mu\colon X'\to X$ be the blow-up along $Z$, so that $\mu^{-1}(Z)=D'$ is a Cartier divisor on $X'$ with $-D'$ $\mu$-very ample, and hence $\cO_{X'}(-mD')$ $\mu$-globally generated for all $m\in\N$. Denoting by $\nu\colon\tX\to X'$ the normalization morphism, we have $\nu^*D'=D$. Since $\nu$ is finite, it follows that $-D$ is $\pi$-ample and satisfies (i), which reads $\cO_{\tX}(-mD)=\cO_{\tX}\cdot\fa_m$ with 
$$
\fa_m:=\pi_*\cO_{\tX}(-mD). 
$$
It therefore remains to establish (ii). By normality of $\tX$, $\cO_{\tX}(-mD)$ is integrally closed, hence so is $\fa_m$. As $\fa\subset\fa_1$, we have $\fa^m\subset\fa_1^m\subset\fa_m$, and hence $\overline{\fa^m}\subset\fa_m$. 

The reverse inclusion requires more work; we reproduce the elegant geometric argument of~\cite[II.11.1.7]{Laz}. Fix $m\ge 1$. As the statement is local over $X$, we may choose a system of generators $(f_1,\dots,f_p)$ for $\fa^m$. This defines a surjection $\cO_X^{\oplus p}\to\fa^m$, which induces, after pull-back and twisting by $-lD$, a surjection
$$
\cO_{\tX}(-lD)^{\oplus p}\to\cO_{\tX}\left(-(m+l)D\right)=\fa^m\cdot\cO_{\tX}(-lD)
$$ 
for any $l\ge 1$. Since $-D$ is $\pi$-ample, Serre vanishing implies that the induced map 
$$
\fa_{l}^{\oplus p}=\pi_*\cO_{\tX}\left(-lD\right)^{\oplus p}\to\fa_{(m+l)}=\pi_*\cO_{\tX}\left(-(m+l)mD\right)
$$ 
is also surjective for $l\gg 1$, \ie $\fa^m\cdot\fa_l=\fa_{m+l}$. But since $\fa_{m+l}\supset\fa_m\cdot\fa_l\supset\fa^m\cdot\fa_l$, $\fa_m$ acts on the finitely generated $\cO_X$-module $\fa_l$ by multiplication by $\fa^m$, and the usual determinant trick therefore yields $\fa_m\subset\overline{\fa^m}$.
\end{proof} 

\begin{defi}\label{defi:rees} Let $Z\subset X$ be a closed subscheme with ideal $\fa$, and let $\pi\colon\tX\to X$ be the normalized blow-up of $Z$, with $D:=\pi^{-1}(Z)$. The \emph{Rees valuations} of $Z$ (or $\fa$) are the divisorial valuations $v_E=\frac{\ord_E}{\ord_E(D)}$, where $E$ runs over the irreducible components of $D$.\end{defi}

Note that $v_E(Z)=v_E(\fa)=v_E(D)=1$ for all $E$.
We now show that the present definition of Rees valuations coincides with the standard one in valuation theory (see for instance~\cite[Chapter 5]{HS}). The next result is a slightly more precise version of~\cite[Theorem 2.2.2, (3)]{HS}. 

\begin{thm}\label{thm:rees} The set of Rees valuations of $\fa$ is the unique finite set $S$ of valuations such that:
\begin{itemize}
\item[(i)] $\overline{\fa^m}=\bigcap_{v\in S}\left\{f\in\cO_X\mid v(f)\ge m \right\}$ for all $m\in\N$; 
\item[(ii)] $S$ is minimal with respect to (i). 
\end{itemize}
\end{thm}

\begin{proof}[Proof of Theorem~\ref{thm:rees}] For each finite set of valuations $S$, set $h_S(f):=\min_{v\in S}v(f)$. Using that $h_S(f^m)=m h_S(f)$, it is straightforward to check that any two sets $S$, $S'$ satisfying (i) have $h_S=h_{S'}$. If $S$ and $S'$ further satisfy (ii), then they are irredundant in the sense of Lemma~\ref{lem:irredundant}, which therefore proves that $S=S'$. 

It remains to check that the set $S$ of Rees valuations of $Z$ satisfies (i) and (ii). The first property is merely a reformulation of Lemma~\ref{lem:normblow}.  Now pick an irreducible  component $E$ of $D$. It defines a fractional ideal $\cO_{\tX}(E)$. Since $-D$ is $\pi$-ample, $\cO_{\tX}(-mD)$ and $\cO_{\tX}(-mD)\cdot\cO_{\tX}(E)$ both become $\pi$-globally generated for $m\gg 1$. Since $\cO_{\tX}(-mD)$ is strictly contained in $\cO_{\tX}(-mD)\cdot\cO_{\tX}(E)$, it follows that 
$\overline{\fa^m}=\pi_*\cO_{\tX}(-mD)$ is strictly contained in 
$$
\pi_*\left(\cO_{\tX}(-mD)\cdot\cO_{\tX}(E)\right)\subset\bigcap_{E'\ne E}\left\{f\in\cO_X\mid v_{E'}(f)\ge m\right\}, 
$$
which proves (ii). 
\end{proof}

\begin{exam}\label{ex:rees} The Rees valuations of an effective Weil divisor  
  $D=\sum_{i=1}^m a_i D_i$ on a normal variety $X$ are given by 
  $v_i:=\frac{1}{a_i}\ord_{D_i}$, $1\le i\le m$.
\end{exam}

We end this section on Rees valuations with the following result. 
\begin{prop}\label{prop:reesexc} Let $\pi:Y\to X$ be a projective birational morphism between normal varieties, and assume that $Y$ admits a Cartier divisor that is both $\pi$-exceptional and $\pi$-ample. Then $\pi$ is isomorphic to the blow-up of $X$ along a closed subscheme $Z$ of codimension at least $2$, and the divisorial valuations $\ord_F$ defined by the $\pi$-exceptional prime divisors $F$ on $Y$ coincide, up to scaling, with the Rees valuations of $Z$.
\end{prop}
This is indeed a direct consequence of the following well-known facts. 

\begin{lem}\label{lem:ample} Let $\pi:Y\to X$ be a projective birational morphism between varieties with $X$ normal. If $G$ is a $\pi$-exceptional, $\pi$-ample Cartier divisor, then:
\begin{itemize}
\item[(i)] $-G$ is effective; 
\item[(ii)] $\supp G$ coincides with the exceptional locus of $\pi$; 
\item[(iii)] for $m$ divisible enough, $\pi$ is isomorphic to the blow-up of the ideal $\fa_m:=\pi_*\cO_Y(mG)$, whose zero locus has codimension at least $2$.
\end{itemize}
Conversely, the blow-up of $X$ along a closed subscheme of codimension at least $2$ admits a $\pi$-exceptional, $\pi$-ample Cartier divisor. 
\end{lem}
Assertion (i) is a special (=ample) case of the Negativity Lemma~\cite[Lemma 3.39]{KM}. The simple direct argument given here is taken from the alternative proof of the Negativity Lemma provided in~\cite[Proposition 2.12]{BdFF}. 

\begin{proof} Set $\fa_m:=\pi_*\cO_Y(mG)$, viewed as a fractional ideal on $X$. Since $G$ is $\pi$-exceptional, every rational function in $\pi_*\cO_Y(mG)$ is regular in codimension $1$, and $\fa_m$ is thus an ideal whose zero locus has codimension at least $2$, by the normality of $X$. 

If we choose $m\gg 1$ such that $\cO_Y(mG)$ is $\pi$-globally generated, then we have $\cO_Y(mG)=\cO_Y\cdot\fa_m\subset\cO_Y$, which proves (i). 

By assumption, $\supp G$ is contained in the exceptional locus $E$ of $\pi$. Since $X$ is normal, $\pi$ has connected fibers by Zariski's main theorem, so $E$ is the union of all projective curves $C\subset Y$ that are mapped to a point of $X$. Any such curve satisfies $G\cdot C>0$ by the relative ampleness of $G$, and hence $C\subset \supp G$ since $-G$ is effective. 
Thus $\supp G=E$, proving~(ii). 

Finally, the relative ampleness of $G$, implies that 
the $\cO_X$-algebra $\bigoplus_{m\in\N}\fa_m$ is finitely generated, and its relative $\Proj$ over $X$ is isomorphic to $Y$. 
The finite generation implies $\bigoplus_{l\in\N}\fa_{ml}=\bigoplus_{l\in\N}\fa_m^l$ 
for all $m$ divisible enough, 
and applying $\Proj_X$ shows that $X$ is isomorphic to the blow-up of $X$ along $\fa_m$. 
\end{proof}
%
%
\subsection{Boundaries and log discrepancies}\label{sec:bound}
Let $X$ be a normal variety. In the Minimal Model Program  (MMP)
terminology, a \emph{boundary} $B$ on $X$ is a $\Q$-Weil divisor (\ie
a codimension one cycle with rational coefficients) such that
$K_{(X,B)}:=K_X+B$ is $\Q$-Cartier. Alternatively, one says that
$(X,B)$ is a \emph{pair} to describe this condition, and $K_{(X,B)}$
is called the \emph{log canonical divisor} of this pair. In
particular, $0$ is a boundary iff $X$ is $\Q$-Gorenstein. 
If $(X',B')$ and $(X,B)$ are pairs and $f\colon X'\to X$ is a morphism,
then we set $K_{(X',B')/(X,B)}=K_{(X',B')}-f^*K_{(X,B)}$.

To any divisorial valuation $v$ on $X$ is associated its \emph{log
  discrepancy} with respect to the pair $(X,B)$, denoted by
$A_{(X,B)}(v)$ and defined as follows. For any proper birational
morphism $\mu\colon Y\to X$, with $Y$ normal, and any prime divisor $F$ of $Y$ such that $v=c\ord_F$, we set 
$$
A_{(X,B)}(v):=c\left(1+\ord_F\left(K_{Y/(X,B)}\right)\right)
$$
with $K_{Y/(X,B)}:=K_{Y}-\mu^*K_{(X,B)}$. This is well-defined (\ie independent of the choice of $\mu$), by  compatibility of canonical divisor classes under push-forward. By construction, $A_{(X,B)}$ is homogeneous with respect to the natural action of $\R_+^*$ on divisorial valuations by scaling, \ie $A_{(X,B)}(c\,v)=c A_{(X,B)}(v)$ for all $c>0$. 

As a real valued function on $k(X)^*$, $c\,v$ converges pointwise to
the trivial valuation $v_{\triv}$ as $c\to 0$. It is thus natural to
set $A_{(X,B)}(v_{\triv}):=0$.  

The pair $(X,B)$ is \emph{sublc} if $A_{(X,B)}(v)\ge0$
for all divisorial valuations $v$. It is \emph{subklt} if the
inequality is strict. If $B$ is furthermore effective, then 
$(X,B)$ is \emph{lc} (or log canonical) and \emph{klt} (Kawamata log terminal), 
respectively. 
If $\mu\colon X'\to X$ is a birational morphism and $B'$ is defined by
$K_{(X',B')}=\mu^*K_{(X,B)}$ and $\mu_*B'=B$, then $A_{(X',B')}=A_{(X,B)}$, so
$(X',B')$ is subklt (resp.\ sublc) iff $(X,B)$ is subklt
(resp.\ sublc), but the corresponding
equivalence may fail for klt or lc pairs, since $B'$ is not necessarily 
effective even when $B$ is. 

If $(X,B)$ is a pair and $D$ is an effective $\Q$-Cartier divisor
on $X$, then we define the \emph{log canonical threshold} 
of $D$ with respect to  $(X,B)$ as
\begin{equation*}
  \lct_{(X,B)}(D):=\sup\left\{t\in\Q\mid(X,B+tD)\ \text{is subklt}\right\}, 
\end{equation*}
with the convention $\lct_{(X,B)}(D)=-\infty$ if $(X,B+tD)$ 
is not subklt for any $t$.
Assume $\lct_{(X,B)}(D)>-\infty$. Since 
$A_{(X,B+tD)}(v)=A_{(X,B)}(v)-t v(D)$
for all divisorial valuations $v$ on $X$, we have
\begin{equation*}
  \lct_{(X,B)}(D)=\inf_v\frac{A_{(X,B)}(v)}{v(D)},
\end{equation*}
where the infimum is taken over $v$ with $v(D)>0$.

When $k$ has characteristic zero, we can compute $\lct_{(X,B)}(D)$ 
using resolution of singularities.
Pick a birational morphism 
$\mu\colon X'\to X$, with $X'$ a smooth projective variety, such that 
if $B'$ and $D'$ are defined by $K_{(X',B')}=\mu^*K_{(X,B)}$, $\mu_*B'=B$ and 
$D':=\mu^*D$, then the union of the
supports of $B'$ and $D'$ has simple normal crossings.
Then
$\lct_{(X,B)}(D)=\lct_{(X',B')}(D')=\min_iA_{(X',B')}(\ord_{E_i})/\ord_{E_i}(D')$, 
where $E_i$ runs over the irreducible components of $D'$. 
%
%
%
%
\section{Test configurations}\label{sec:test}
In what follows, $X$ is a projective scheme over $k$, and $L$ is a $\Q$-line bundle on $X$.  
Most often, $L$ will be ample, but it is sometimes useful to consider the general case. Similarly, it will be convenient to allow some flexibility in the definition of test configurations.
\begin{defi}\label{defi:test1} A test configuration $\cX$ for $X$ consists of the following data:
\begin{itemize}
\item[(i)] a flat and proper morphism of schemes $\pi\colon\cX\to\A^1$; 
\item[(ii)] a $\G_m$-action on $\cX$ lifting the canonical action on $\A^1$;  
\item[(iii)] an isomorphism $\cX_1\simeq X$. 
\end{itemize}
\end{defi}
By Proposition~\ref{prop:basic} below, $\cX$ is automatically
a variety (\ie reduced and irreducible) when $X$ is.  
The central fiber $\cX_0:=\pi^{-1}(0)$ is an effective 
Cartier divisor on $\cX$ by the flatness of~$\pi$. 

Given test configurations $\cX$, $\cX'$ for $X$, the isomorphism $\cX'_1\simeq X\simeq\cX_1$ induces a canonical $\G_m$-isomorphism $\cX'\setminus\cX'_0\simeq\cX\setminus\cX_0$.
We say that $\cX'$ \emph{dominates} $\cX$ if this isomorphism extends
to a morphism $\cX'\to\cX$. When it is an isomorphism, we 
abuse notation slightly and write $\cX'=\cX$ (which is reasonable
given that the isomorphism is canonical). 
Any two test configurations $\cX_1$, $\cX_2$ can be dominated by a
third, for example the graph of $\cX_1\dashrightarrow\cX_2$.

\begin{defi}\label{defi:test2}
A test configuration $(\cX,\cL)$ for $(X,L)$ consists of a test
configuration $\cX$ for $X$, together with the following data:
\begin{itemize}
\item[(iv)] a $\G_m$-linearized $\Q$-line bundle $\cL$ on $\cX$;
\item[(v)] an isomorphism $(\cX_1,\cL_1)\simeq (X,L)$ extending the one in~(iii).
\end{itemize}
\end{defi}
By a $\G_m$-linearized $\Q$-line bundle $\cL$ as in~(iv), we mean that $r\cL$ is
an actual $\G_m$-linearized line bundle for some $r\in\Z_{>0}$ that is
not part of the data. The isomorphism in (v) then means
$(\cX,r\cL_1)\simeq(X,rL)$. 

We say that $(\cX,\cL)$ is ample, semiample,\dots (resp.~normal,
$S_1$,\dots) when $\cL$ (resp.~$\cX$) has the corresponding property. 

A \emph{pull-back} of a test configuration $(\cX,\cL)$ for $(X,L)$  
is a test configuration $(\cX',\cL')$ where $\cX'$ dominates $\cX$  
and $\cL'$ is the pull-back of $\cL$. 
 
For each $c\in\Q$, the $\G_m$-linearization of the $\Q$-line bundle $\cL$ may be twisted by $ t^c$, in the sense that the $\G_m$-linearization of $r\cL$ is twisted by the character $ t^{rc}$ with $r$ divisible enough. The resulting test configuration can be identified with $(\cX,\cL+c\cX_0)$. 

If $(\cX,\cL_1)$ and $(\cX,\cL_2)$ are test configurations for
$(X,L_1)$, $(X,L_2)$, respectively, and $c_1,c_2\in\Q_{>0}$, then 
$(\cX,c_1\cL_1+c_2\cL_2)$ is a test configuration for $(X,c_1L_1+c_2L_2)$.

If $(\cX,\cL)$ is a test configuration of $(X,\cO_X)$, then there
exists $r\in\Z_{>0}$ and a Cartier divisor $D$ on $\cX$ supported on
$\cX_0$ such that $r\cL=\cO_{\cX}(D)$.
%
%
\begin{exam}\label{ex:product} Every $\G_m$-action on $X$ induces a diagonal $\G_m$-action on $X\times\A^1$, thereby defining a \emph{product test configuration} $\cX$ for $X$. Similarly, a $\G_m$-linearization of $rL$ for some $r\ge 1$ induces a product test configuration $(\cX,\cL)$ for $(X,L)$, which is simply $(X,L)\times\A^1$ with diagonal action of $\G_m$. 
\end{exam}
We denote by $X_{\A^1}$ (resp.~$(X_{\A^1},L_{\A^1})$ the product test configuration induced by the trivial $\G_m$-action on $X$ (resp.~$(X,L)$). 
\begin{exam}\label{ex:defnorm} The \emph{deformation to the normal cone} of a closed subscheme $Z\subset X$ is the blow-up $\rho\colon\cX\to X_{\A^1}$ along $Z\times\{0\}$.
Thus $\cX$ is a test configuration dominating $X_{\A^1}$. By~\cite[Chapter 5]{Ful}, its central fiber splits as $\cX_0=E+F$, where $E=\rho^{-1}(Z\times\{0\})$ is the exceptional divisor and $F$ is the strict transform of $X\times\{0\}$, which is isomorphic to the blow-up of $X$ along $Z$. 
\end{exam}
\begin{exam}\label{E201}
  More generally we can blow up any $\G_m$-invariant ideal 
  $\fa$ on $X\times\A^1$ supported on the central fiber. We discuss this further in \S\ref{FlagIdeals}. 
\end{exam}
%
%
\subsection{Scheme theoretic features}
Recall that a scheme $Z$ satisfies \emph{Serre's condition $S_k$} iff 
\begin{equation*}
  \depth\cO_{Z,\xi}\ge\min\{\codim \xi,k\}
  \quad\text{for every point $\xi\in Z$}.
\end{equation*}
In particular, $Z$ is $S_1$ iff it has no embedded points. While we will not use it, one can show that $Z$ is $S_2$ iff it has no embedded points and satisfies the Riemann extension property across closed subsets of codimension at least $2$. 

On the other hand, $Z$ is \emph{regular in codimension $k$} ($R_k$ for short) iff $\cO_{Z,\xi}$ is regular for every $\xi\in\cX$ of codimension at most $k$. Equivalently $Z$ is $R_k$ iff its singular locus has codimension greater than $k$. 
Note that $Z$ is $R_0$ iff it is generically reduced.

Serre's criterion states that $Z$ is normal iff it is $R_1$ and $S_2$. Similarly, $Z$ is reduced iff it is $R_0$ and $S_1$ (in other words, iff $Z$ is generically reduced and without embedded points).  

\begin{prop}\label{prop:basic} Let $\cX$ be a test configuration for $X$. 
\begin{itemize}
\item[(i)] $\cX$ is reduced iff so is $X$. 
\item[(ii)] $\cX$ is $S_2$ iff $X$ is $S_2$ and $\cX_0$ is $S_1$ (\ie without embedded points). 
\item[(iii)] If $X$ is $R_1$ and $\cX_0$ is generically reduced  (that is, `without multiple components'), then $\cX$ is $R_1$.
\item[(iv)] If $X$ is normal and $\cX_0$ is reduced, then $\cX$ is normal. 
\item[(v)] Every irreducible component $\cY$ (with its reduced
  structure) of $\cX$ is a test configuration for a unique irreducible
  component $Y$ of $X$. Further, the multiplicities of $\cX$ along
  $\cY$ and those of $X$ along $Y$ are equal. 
\item[(vi)] $\cX$ is a variety iff so is $X$.
\end{itemize}
\end{prop}
Recall that the multiplicity of $X$ along $Y$ is defined as the length of $\cO_X$ at the generic point of $Y$. 
\begin{proof} Flatness of $\pi\colon\cX\to\A^1$ implies that $\cX_0$ is Cartier divisor and that 
every associated (\ie generic or embedded) point of $\cX$ belongs to $\cX\setminus\cX_0$ (cf.~\cite[Proposition III.9.7]{Har}). 
The proposition is a simple consequence of this fact and of the isomorphism $\cX\setminus\cX_0\simeq X\times(\A^1\setminus\{0\})$. 

More specifically, since $\A^1\setminus\{0\}$ is smooth, $\cX\setminus\cX_0$ is $R_k$ (resp.~$S_k$) iff $X$ is. Since $\cX_0$ is a Cartier divisor, we also have
$$
\depth\cO_{\cX_0,\xi}=\depth\cO_{\cX,\xi}-1
$$
for each $\xi\in\cX_0$, so that $\cX$ is $S_k$ iff $X$ is $S_k$ and $\cX_0$ is $S_{k-1}$. 

It remains to show that $\cX_0$ being generically reduced and $X$
being $R_1$ imply that $\cX$ is $R_1$. But every codimension one point
$\xi\in\cX$ either lies the open subset $\cX\setminus\cX_0$, in which case $\cX$ is regular at $\xi$, or is a generic
point of the Cartier divisor $\cX_0$. In the latter case, the closed
point of $\Spec\cO_{\cX,\xi}$ is a reduced Cartier divisor; hence $\cO_{\cX,\xi}$ is regular.

Now, $\cX\setminus\cX_0$ is Zariski dense in $\cX$ since $\cX_0$
is a Cartier divisor. Hence $\cX$ is isomorphic to
$X\times\A^1$ at each generic point, and (v) easily follows.
Finally,~(vi) is a consequence of~(i) and~(v).
\end{proof}
%
%
\subsection{Compactification}\label{sec:compact}
For many purposes it is convenient to compactify test configurations. The following notion
provides a canonical way of doing so.
\begin{defi}\label{defi:comp} The \emph{compactification} $\bar\cX$ of a test configuration $\cX$ for $X$ is defined by gluing together $\cX$ and $X\times(\P^1\setminus\{0\})$ along their respective open subsets $\cX\setminus\cX_0$ and $X\times(\A^1\setminus\{0\})$, which are identified using the canonical $\G_m$-equivariant isomorphism $\cX\setminus\cX_0\simeq X\times(\A^1\setminus\{0\})$. 
\end{defi}
The compactification comes with a $\G_m$-equivariant flat morphism $\bar\cX\to\P^1$, still denoted by $\pi$. By construction, $\pi^{-1}(\P^1\setminus\{0\})$ is $\G_m$-equivariantly isomorphic to $X_{\P^1\setminus\{0\}}$ over $\P^1\setminus\{0\}$. 

Similarly, a test configuration $(\cX,\cL)$ for $(X,L)$ admits a compactification $(\bar\cX,\bar\cL)$, where $\bar\cL$ is a $\G_m$-linearized $\Q$-line bundle on $\bar\cX$. Note that $\bar\cL$ is relatively (semi)ample iff $\cL$ is. 

\begin{exam}\label{ex:comp} When $\cX$ is the product test configuration defined by a $\G_m$-action on $X$, the compactification $\bar\cX\to\P^1$ may be alternatively described as the locally trivial fiber bundle with typical fiber $X$ associated to the principal $\G_m$-bundle $\A^2\setminus\{0\}\to\P^1$, \ie 
$$
\bar\cX=\left((\A^2\setminus\{0\})\times X\right)/\G_m
$$ 
with $\G_m$ acting diagonally. Note in particular that $\bar\cX$ is not itself a product in general. For instance, the $\G_m$-action $ t\cdot[x:y]=[t^d x:y]$ on $X=\P^1$ gives rise to the Hirzebruch surface $\bar\cX\simeq\P\left(\cO_{\P^1}\oplus\cO_{\P^1}(d)\right)$. 
\end{exam} 
%
%
\subsection{Ample test configurations and one-parameter subgroups}
Let $(X,L)$ be a polarized projective scheme. Fix $r\ge 1$ such that $rL$ is very ample, and consider the corresponding closed embedding $X\hookrightarrow\P V^*$ with $V:=H^0(X,rL)$.  

Every one-parameter subgroup ($1$-PS for short) $\rho\colon\G_m\to\mathrm{PGL}(V)$ induces a test configuration $\cX$ for $X$, 
defined as the schematic closure in $\P V^*\times\A^1$ of the image of the closed embedding $X\times\G_m\hookrightarrow\P V^*\times\G_m$ mapping $(x, t)$ to $(\rho(t)x, t)$. In other words, $\cX_0$ is defined as the `flat limit' as $t\to 0$ of the image of $X$ under $\rho(t)$, cf.~\cite[Proposition 9.8]{Har}. By Proposition~\ref{prop:basic}, the schematic closure is simply given by the Zariski closure when $X$ is reduced. 

If we are now given $\rho\colon\G_m\to\GL(V)$, then $\cO_\cX(1)$ is $\G_m$-linearized, and we get an ample test configuration $(\cX,\cL)$ for $(X,L)$ by setting $\cL:=\tfrac 1r\cO_\cX(1)$. 

\smallskip

Conversely, every ample test configuration is obtained in this way, as was originally pointed out in~\cite[Proposition 3.7]{RT}. Indeed, let $(\cX,\cL)$ be an ample test configuration, and pick $r\ge 1$ such that $r\cL$ is (relatively) very ample. The direct image $\cV:=\pi_*\cO_\cX(r\cL)$ under $\pi\colon\cX\to\A^1$ is torsion-free by flatness of $\pi$, and hence a $\G_m$-linearized vector bundle on $\A^1$ with an equivariant embedding $\cX\hookrightarrow\P(\cV^*)$ such that $r\cL=\cO_\cX(1)$.  

By Proposition~\ref{prop:reesfiltr}, $\cV$ is $\G_m$-equivariantly isomorphic to $V\times\A^1$ for a certain $\G_m$-action $\rho\colon\G_m\to\GL(V)$, and it follows that $(\cX,\cL)$ is the ample test configuration attached to $\rho$. 
%
%
\subsection{Trivial and almost trivial test configurations}\label{sec:triv}
The normalization $\nu\colon\tX\to X$ of a (possibly non-reduced) scheme $X$ is defined as the normalization of the reduction $X_\red$ of $X$. Denoting by $X_\red=\bigcup_\a X^\a$ the irreducible decomposition, we have $\tX=\coprod_\a \tX^\a$, the disjoint union of the normalizations $\tX^\a\to X^\a$. 

If $L$ is a $\Q$-line bundle and $\tL:=\nu^*L$, we call $(\tX,\tL)$ the \emph{normalization} of $(X,L)$. If $L$ is ample, then so is $\tL$ (cf.~\cite[\S4]{Har2}). The normalization $(\tcX,\tcL)$ of a test configuration $(\cX,\cL)$ is similarly defined (the flatness of $\tcX\to\A^1$ being a consequence of~\cite[Proposition III.9.7]{Har}), and is a test configuration for $(\tX,\tL)$. By Proposition~\ref{prop:basic}, we have $\tcX=\coprod_\a\tcX^\a$ with $(\tcX^\a,\tcL^\a)$ a test configuration for $(\tX^\a,\tL^\a)$. 

\begin{defi}\label{defi:triv} A test configuration $(\cX,\cL)$ for
  $(X,L)$ is \emph{trivial} if $\cX=X_{\A^1}$. We say that $(\cX,\cL)$
  is \emph{almost trivial} if the normalization $\tcX^\a$ of each
  top-dimensional irreducible component $\cX^\a$ is trivial. 
\end{defi}
Note that (almost) triviality does not a priori bear on $\cL$. However, we have:
\begin{lem}\label{lem:triv} A test configuration $(\cX,\cL)$ is almost
  trivial iff for each top-dimensional irreducible component $X^\a$ of
  $X$, the corresponding irreducible component $\tcX^\a$ of the normalization of $\cX$ satisfies $(\tcX^\a,\tcL^\a+c_\a\tcX^\a_0)=(\tX^\a_{\A^1},\tL^\a_{\A^1})$ for some $c_\a\in\Q$.
\end{lem}
\begin{proof} We may assume that $\cX$ (and hence $X$) is normal and irreducible. Replacing $\cL$ with $\cL-L_{\A^1}$, we may also assume that $L=\cO_X$, and we then have $\cL=D$ for a unique $\Q$-Cartier divisor $D$ supported on $\cX_0$. If $(\cX,\cL)$ is almost trivial, then $\cX=X_{\A^1}$, and $\cX_0=X\times\{0\}$ is thus irreducible. It follows that $D$ is a multiple of $\cX_0$, hence the result. 
\end{proof}

The next result shows that the current notion of almost triviality is compatible with the one introduced in~\cite{Sto2,Oda4}. 

\begin{prop}\label{prop:almosttriv} Assume that $L$ is ample, and let $(\cX,\cL)$ be an ample test configuration for $(X,L)$. 
\begin{itemize}
\item[(i)] If $X$ is normal, then $(\cX,\cL)$ is almost trivial iff $X_{\A^1}$ dominates $\cX$.
\item[(ii)] If $X$ is reduced and equidimensional, then $(\cX,\cL)$ is almost trivial iff the canonical birational map $X_{\A^1}\dashrightarrow\cX$ is an isomorphism in codimension one. 
\end{itemize}
\end{prop} 

\begin{proof} Consider first the case where $\cX$ is normal and irreducible, and assume that $X_{\A^1}\dashrightarrow\cX$ is an isomorphism in codimension one. The strict transform $\cL'$ of $\cL$ (viewed as a $\Q$-Weil divisor class) on $X_{\A^1}$ coincides with $L_{\A^1}$ outside $X\times\{0\}$. The latter being irreducible, we thus have $\cL'=L_{\A^1}+c X\times\{0\}$. This shows that $\cL'$ is ($\Q$-Cartier and) relatively ample. Since the normal varieties $X_{\A^1}$ and $\cX$ are isomorphic outside a Zariski closed subset of codimension at least $2$, we further have $H^0(\cX,m\cL)\simeq H^0(X_{\A^1},m\cL')$ for all $m$ divisible enough, and we conclude by ampleness that $(\cX,\cL)\simeq(X_{\A^1},\cL')$ is trivial. 

We now treat the reduced case, as in (ii). Observe first that $X_{\A^1}$ is regular at each generic point of $X\times\{0\}$, because $X$ is regular in codimension zero, being reduced. As a result, $\tX_{\A^1}\to X_{\A^1}$ is an isomorphism in codimension one. 

Now assume that $X_{\A^1}\dashrightarrow\cX$ is an isomorphism in
codimension one. Then $\cX$ is isomorphic to $X_{\A^1}$ at each
generic point $\xi$ of $\cX_0$. By the previous observation, $\cX$ is
regular at $\xi$, so that $\tcX\to\cX$ is an isomorphism at each
generic point of $\tcX_0$. The same therefore holds for
$\tcX\dashrightarrow\tX_{\A^1}$, which means that $\tcX$ is almost
trivial. Applying the first part of the proof to each irreducible component of $\tcX$ shows that $\tcX=\tX_{\A^1}$. 

Assume conversely that $(\cX,\cL)$ is almost trivial, \ie $\tcX=\tX_{\A^1}$. Since $\tcX\to\A^1$ factors through $\cX\to\A^1$ and $\tX_{\A^1}\to X_{\A^1}$ is an isomorphism in codimension one, we see that the coordinate $t$ on $\A^1$ is a uniformizing parameter on $\cX$ at each generic point of $\cX_0$, and it follows easily that $\cX\dashrightarrow X_{\A^1}$ is an isomorphism in codimension one. 

Finally, (i) is a consequence of (ii). 
\end{proof}

At the level of one-parameter subgroups, almost triviality admits the following simple characterization, which completes~\cite[Proposition~3.5]{Oda4}. 
\begin{prop}\label{prop:triv1PS} 
  Assume that $(X,L)$ is a polarized normal variety, and pick $r\ge 1$ with $rL$ very ample. 
  Let $(s_i)$ be a basis of $H^0(X,rL)$, pick integers
  $$
  a_1=\dots=a_p<a_{p+1}\le\dots\le a_{N_r},
  $$ 
  and let $\rho\colon\G_m\to\GL(H^0(X,rL))$ be the $1$-parameter subgroup such that 
  $\rho(t)s_i=t^{a_i} s_i$. The test configuration $(\cX,\cL)$ defined by $\rho$ is then 
  almost trivial iff $\bigcap_{1\le i\le p}(s_i=0)=\emptyset$ in $X$. 
\end{prop}
This recovers the key observation of~\cite[\S3.1]{LX} that almost trivial, nontrivial test configurations always exist, and gives a simple explicit way to construct them. 

\begin{proof} The canonical rational map
$$
\phi\colon X\times\A^1\dashrightarrow\cX\hookrightarrow\P^{N_r-1}\times\A^1
$$
is given by
\begin{multline*}
  \phi(x,t)=(\rho(t)[s_i(x)],t)=([t^{a_i}s_i(x)],t)\\
  =([s_1(x):\dots:s_p(x):t^{a_{p+1}-a_1}s_{p+1}:\dots:t^{a_{N_r}-a_1}s_{N_r}(x)],t), 
\end{multline*}
where $a_j-a_1\ge 1$ for $j>p$. By (i) of Proposition~\ref{prop:almosttriv}, $(\cX,\cL)$ is almost trivial iff $\phi$ extends to a morphism $X\times\A^1\to\P^{N_r-1}\times\A^1$, and this is clearly the case iff $\bigcap_{1\le i\le p}(s_i=0)=\emptyset$. 
\end{proof}
%
%
\subsection{Test configurations and filtrations}\label{sec:filtrtc}
By the reverse Rees construction of~\S\ref{sec:reesfiltr}, every test configuration $(\cX,\cL)$ for $(X,L)$ induces a $\Z$-filtration of the graded algebra 
$$
R(X,rL):=\bigoplus_{m\in\N} H^0(X,mrL)
$$ 
for $r$ divisible enough. More precisely, for each $r$ such that $r\cL$ is a line bundle, we define a $\Z$-filtration on $R(X,rL)$ by letting $F^\la  H^0(X,mrL)$ be the (injective) image of the weight-$\la$ part $H^0(\cX,rm\cL)_\la$ of $H^0(\cX,mr\cL)$ under the restriction map
$$
H^0(\cX,mr\cL)\to H^0(\cX,mr\cL)_{t=1}=H^0(X,mrL).
$$
Alternatively, we have
\begin{equation}\label{equ:filtr}
F^\la  H^0(X,mrL)=\left\{s\in H^0(X,mrL)\mid  t^{-\la}\bar s\in H^0(\cX,mr\cL)\right\}
\end{equation}
where $\bar s\in H^0(\cX\setminus\cX_0,mr\cL)$ denotes the $\G_m$-invariant section defined by $s\in H^0(X,mrL)$. 

As a direct consequence of the projection formula, we get the following invariance property. 
\begin{lem}\label{lem:inv} Let $(\cX,\cL)$ be a test configuration, and let $(\cX',\cL')$ be a pull-back of $(\cX,\cL)$ such that the corresponding morphism $\mu\colon\cX'\to\cX$ satisfies $\mu_*\cO_{\cX'}=\cO_\cX$. Then $(\cX,\cL)$ and $(\cX',\cL')$ define the same filtrations on $R(X,rL)$. 
\end{lem}
Note that $\mu_*\cO_{\cX'}=\cO_\cX$ holds automatically when $\cX$ (and hence $X$) is normal, by Zariski's main theorem. 

For later use, we also record the following direct consequence of the $\G_m$-equivariant isomorphism (\ref{equ:reessp}). 

\begin{lem}\label{L201} Let $(\cX,\cL)$ be a test configuration, with projection 
  $\pi\colon\cX\to\A^1$. For each $m$ with $m\cL$ a line bundle, the multiplicities of the 
  $\G_m$-module $\pi_*\cO_\cX(m\cL)_0$ satisfy
  \begin{equation*}
    \dim\left(\pi_*\cO_\cX(m\cL)_0\right)_\la=\dim F^\la  H^0(X,mL)/F^{\la+1} H^0(X,mL)
  \end{equation*}
  for all $\la\in\Z$. In particular, the weights of $\pi_*\cO_\cX(m\cL)_0$ coincide with the successive 
  minima of $F^\bullet H^0(X,mL)$. 
\end{lem}

\smallskip
\begin{prop}\label{prop:filtrtest} Assume $L$ is ample. Then the above construction sets up a one-to-one correspondence between ample test configurations for $(X,L)$ and finitely generated $\Z$-filtrations of $R(X,rL)$ for $r$ divisible enough. 
\end{prop}
\begin{proof} When $(\cX,\cL)$ is an ample test configuration, the $\Z$-filtration it defines on $R(X,rL)$ is finitely generated in the sense of Definition~\ref{defi:fingen}, since 
$$
\bigoplus_{m\in\N}\left(\bigoplus_{\la\in\Z}t^{-\la} F^\la  H^0(X,mrL)\right)=R(\cX,r\cL)
$$
is finitely generated over $k[t]$. Conversely, let $F^\bullet$ be a finitely generated $\Z$-filtration of $R(X,rL)$ for some $r$. Replacing $r$ with a multiple, we may assume that the graded $k[t]$-algebra
$$
\bigoplus_{m\in\N}\left(\bigoplus_{\la\in\Z}t^{-\la} F^\la  H^0(X,mrL)\right)
$$
is generated in degree $m=1$, and taking the Proj over $\A^1$ defines an ample test configuration for $(X,rL)$, hence also one for $(X,L)$. Using \S\ref{sec:reesfiltr}, it is straightforward to see that the two constructions are inverse to each other. 
\end{proof}

Still assuming $L$ is ample, let $(\cX,\cL)$ be merely semiample. The $\Z$-filtration it defines on $R(X,rL)$ is still finitely generated, as
$$
\bigoplus_{m\in\N}\left(\bigoplus_{\la\in\Z}t^{-\la} F^\la  H^0(X,mrL)\right)=R(\cX,r\cL)
$$
is finitely generated over $k[t]$. 

\begin{defi}\label{defi:amplemodel} The \emph{ample model} of a semiample test configuration $(\cX,\cL)$ is defined as the unique ample test configuration $(\cX_\amp,\cL_\amp)$ corresponding to the finitely generated $\Z$-filtration defined by $(\cX,\cL)$ on $R(X,rL)$ for $r$ divisible enough. 
\end{defi}
Ample models admit the following alternative characterization. 

\begin{prop}\label{prop:amplemodel} 
  The ample model $(\cX_\amp,\cL_\amp)$ of a semiample test configuration $(\cX,\cL)$ 
  is the unique ample test configuration such that:
  \begin{itemize}
  \item[(i)] 
    $(\cX,\cL)$ is a pull-back of $(\cX_\amp,\cL_\amp)$;
  \item[(ii)] 
    the canonical morphism $\mu\colon\cX\to\cX_\amp$ satisfies $\mu_*\cO_{\cX}=\cO_{\cX_\amp}$.
  \end{itemize}
\end{prop}
Note that (ii) implies that $\cX_\amp$ is normal whenever $\cX$ (and hence $X$) is.
\begin{proof} 
  Choose $r\ge 1$ such that $r\cL$ is a globally generated line bundle. 
  By Proposition~\ref{prop:reesfiltr}, the vector bundle $\pi_*\cO_\cX(r\cL)$ is 
  $\G_m$-equivariantly trivial over $\A^1$, and we thus get an induced $\G_m$-equivariant 
  morphism $f\colon\cX\to\P^N_{\A^1}$ over $\A^1$  for some $N$ with the property that 
  $f^*\cO(1)=r\cL$. The Stein factorization of $f$ thus yields an ample test configuration 
  $(\cX',\cL')$ satisfying (i) and (ii). By Lemma~\ref{lem:inv}, these properties guarantee that 
  $(\cX,\cL)$ and $(\cX',\cL')$ induce the same $\Z$-filtration on $R(X,rL)$, 
  and hence $(\cX',\cL')=(\cX_\amp,\cL_\amp)$ by Proposition~\ref{prop:filtrtest}. 
\end{proof} 
%
%
\subsection{Flag ideals}\label{FlagIdeals}
In this final section, we discuss (a small variant of) the \emph{flag ideal} point of
view of~\cite{Oda1,Oda3}. We assume that $X$ is normal, and use the following terminology. 

\begin{defi}\label{defi:pull} 
  A \emph{determination} of a test configuration $\cX$ for $X$ is a normal test configuration $\cX'$ dominating both $\cX$ and $X_{\A^1}$.  
\end{defi}
Note that a determination always exists: just pick $\cX'$ to be the normalization of the graph of the canonical birational map $\cX\dashrightarrow X_{\A^1}$.

Similarly, a determination of a test configuration $(\cX,\cL)$ for $(X,L)$ is a normal test configuration $(\cX',\cL')$ such that $\cX'$ 
is a determination of $\cX$ and $\cL'$ is the pull-back of $\cL$ under the morphism
$\cX'\to\cX$ (\ie $(\cX',\cL')$ is a pull-back of $(\cX,\cL)$).
In this case, denoting by $\rho\colon\cX'\to X_{\A^1}$ the canonical morphism, we have $\cL'=\rho^*L_{\A^1}+D$ for a unique $\Q$-Cartier divisor $D$ supported on $\cX_0'$, by the normality of $\cX'$. 

\begin{defi}\label{defi:flag} Let $(\cX,\cL)$ be test configuration for $(X,L)$. For each $m$ such that $m\cL$ is a line bundle, we define the \emph{flag ideal} of $(\cX,m\cL)$ 
 as 
$$
\fa^{(m)}:=\rho_*\cO_{\cX'}(mD), 
$$ 
viewed as a $\G_m$-invariant, integrally closed fractional ideal of the normal variety $X_{\A^1}$. 
\end{defi} 
By Lemma~\ref{lem:flag} below, $\fa^{(m)}$ is indeed independent of the choice of a determination. In particular, $\fa^{(m)}$ is also the flag ideal of $(\cX',m\cL')$ for every normal pull-back $(\cX',\cL')$ of $(\cX,\cL)$. 

Since $\fa^{(m)}$ is a $\G_m$-invariant fractional ideal on $X_{\A^1}$ that is trivial 
outside the central fiber, it is of the form 
\begin{equation}\label{equ:flag}
\fa^{(m)}=\sum_{\la\in\Z} t^{-\la}\fa^{(m)}_\la
\end{equation}
where $\fa^{(m)}_\la\subset\cO_X$ is a non-increasing sequence of integrally closed ideals on $X$ with $\fa^{(m)}_\la=0$ for $\la\gg 0$ and $\fa^{(m)}_\la=\cO_X$ for $\la\ll 0$ (see Proposition~\ref{prop:filtrflag} below for the choice of sign). 

\begin{lem}\label{lem:flag} The flag ideal
  $\fa^{(m)}$ is independent of the choice of a determination $(\cX',\cL')$. 
\end{lem}
\begin{proof} Let $(\cX'',\cL'')$ be another determination of $(\cX,\cL)$ (and recall that $\cX'$ and $\cX''$ are normal, by definition). Since any two determinations of $(\cX,\cL)$ are dominated by a third one, we may assume that $\cX''$ dominates $\cX'$. Denoting by $\mu'\colon\cX''\to\cX'$ the corresponding morphism, the fractional ideal attached to $(\cX'',\cL'')$ is then given by 
$$
(\rho\circ\mu')_*\cO_{\cX''}(m\mu'^*D)
$$
By the projection formula we have 
$$
\mu'_*\cO_{\cX''}(m\mu'^*D)=\cO_{\cX'}(mD)\otimes\mu'_*\cO_{\cX''}, 
$$
and we get the desired result since $\mu'_*\cO_{\cX''}=\cO_{\cX'}$ by normality of $\cX'$. 
\end{proof}
\begin{prop}\label{prop:filtrflag}
  Let $(\cX,\cL)$ be a normal, semiample test configuration for $(X,L)$.
  For each $m$ with $m\cL$ a line bundle, let $F^\bullet H^0(X,mL)$ be the corresponding 
  $\Z$-filtration and $\fa^{(m)}$ the flag ideal of $(\cX,m\cL)$. 
  Then, for $m$ sufficiently divisible and $\la\in\Z$, 
  the $\cO_X$-module $\cO_X(mL)\otimes\fa^{(m)}_\la$ is globally generated 
  and
  \begin{equation*}
    F^\la  H^0(X,mL)=H^0\left(X,\cO_X(mL)\otimes\fa^{(m)}_\la\right)
  \end{equation*}
  In particular, the successive minima of $F^\bullet H^0(X,mL)$ 
  (see~\S\ref{sec:prelim}) are exactly the 
  $\la\in\Z$ with $\fa^{(m)}_\la\ne\fa^{(m)}_{\la+1}$, with the
  largest one being 
  $\la^{(m)}_{\max}=\max\left\{\la\in\Z\mid\fa^{(m)}_\la\ne 0\right\}$.
\end{prop}

\begin{proof} 
  Let $(\cX',\cL')$ be a determination of $(\cX,\cL)$, \ie a pull-back such that $\cX'$ is normal 
  and dominates $X_{\A^1}$. By normality of $\cX$, the morphism $\mu\colon\cX'\to\cX$ satisfies 
  $\mu_*\cO_{\cX'}=\cO_\cX$, and the projection formula therefore shows that $(\cX',\cL')$ and 
  $(\cX,\cL)$ define the same $\Z$-filtration of $R(X,rL)$ for $r$ divisible enough. 
  Since $\fa^{(m)}$ is also the flag ideal of $(\cX',m\cL')$, we may assume to begin with that 
  $\cX$ dominates $X_{\A^1}$. Denoting by $\rho\colon\cX\to X_{\A^1}$ the canonical morphism, 
  we then have $\cL=\rho^*L_{\A^1}+D$ and 
  \begin{equation*}
    \fa^{(m)}=\rho_*\cO_{\cX}(mD), 
  \end{equation*}
  and hence 
  \begin{equation*}
    \rho_*\cO_X(m\cL)=\cO_X(mL_{\A^1})\otimes\fa^{(m)}
  \end{equation*}
  by the projection formula. 
  As a consequence, $H^0(X_{\A^1},\cO_{X_{\A^1}}(mL_{\A^1})\otimes t^{-\la}\fa^{(m)}_\la)$ 
  is isomorphic to the weight-$\la$ part of $H^0(\cX,m\cL)$, and the first point follows. 
  
  For $m$ divisible enough, $m\cL$ is globally generated on $\cX$, and hence so is 
  $\rho_*\cO_\cX(m\cL)$ on $X_{\A^1}$. Decomposing into weight spaces thus shows that 
  $\cO_X(mL)\otimes\fa^{(m)}_\la$ is globally generated on $X$ for all $\la\in\Z$.  
  We therefore have $\fa^{(m)}_\la\ne \fa^{(m)}_{\la+1}$ iff $F^\la  H^0(X,mL)\ne F^{\la+1} H^0(X,mL)$, 
  hence the second point. 
\end{proof}
%
\section{Duistermaat-Heckman measures and Donaldson-Futaki invariants}\label{sec:DHDF}
%
In this section, $(X,L)$ is a polarized\footnote{As before we allow $L$ to be an (ample) $\Q$-line bundle on $X$.} 
scheme over $k$. Our goal is to provide
an elementary, self-contained treatment of Duistermaat-Heckman measures and Donaldson-Futaki
invariants. Most arguments are inspired by those in~\cite{Don3,RT,Wan,Oda2,LX}. 

%
\subsection{The case of a $\G_m$-action}\label{sec:DHDFaction}
First assume that $L$ is an ample line bundle (as opposed to a
$\Q$-line bundle) and that $(X,L)$ is given a $\G_m$-action. For each $d\in\N$, the principal $\G_m$-bundle $\A^{d+1}\setminus\{0\}\to\P^d$ induces a projective morphism $\pi_d:X_d\to\P^d$, locally trivial in the Zariski topology and with typical fiber $X$, as well as a relatively ample line bundle $L_d$ on $X_d$. For $d=1$, we recover the compactified product test configuration, cf.~Example~\ref{ex:comp}.  

Following~\cite[p.470]{Don3}, we use this construction to prove the following key result, which is often claimed to follow from `general theory' in the K-stability literature. Another proof relying on the equivariant Rieman-Roch theorem 
is provided in Appendix B. 

\begin{thm}\label{thm:equivRR} Let $(X,L)$ be a polarized scheme with a $\G_m$-action, and set $n=\dim X$. For each $d,m\in\N$, the finite sum
$$
\sum_{\la\in\Z}\frac{\la^d}{d!}\dim H^0(X,mL)_\la
$$
is a polynomial function of $m\gg 1$, of degree at most $n+d$. The coefficient of $m^{n+d}$ is further equal to $(L_d^{n+d})/(n+d)!$. 
\end{thm}
Here we write as usual $(L_d^{n+d})=c_1(L_d)^{n+d}\cdot[X_d]$, with $[X_d]\in\CH_{n+d}(X_d)$ the fundamental class. 

Granting this result for the moment, we get as a first consequence:
\begin{cor}\label{cor:total} Let $w_m\in\Z$ be the weight of the $\G_m$-action on the determinant line $\det H^0(X,mL)$, and $N_m:=h^0(X,mL)$. Then we have an asymptotic expansion 
\begin{equation}\label{equ:Futaki}
\frac{w_m}{mN_m}=F_0+m^{-1}F_1+m^{-2}F_2+\dots 
\end{equation}
with $F_i\in\Q$.  
\end{cor}
Indeed, $w_m=\sum_{\la\in\Z}\la\dim H^0(X,mL)_\la$ is a polynomial of degree at most $n+1$ by Theorem~\ref{thm:equivRR}, while $N_m$ is a polynomial of degree $n$ by Riemann-Roch. 

\begin{defi}\label{defi:DFaction} The \emph{Donaldson-Futaki invariant} $\DF(X,L)$ of the polarized $\G_m$-scheme $(X,L)$ is defined as 
$$
\DF(X,L)=-2F_1.
$$
\end{defi}
The factor $2$ in the definition is here just for convenience, while the sign is chosen so that K-semistability will later correspond to $\DF\ge 0$, cf.~Definition~\ref{defi:Kstab}.

As a second consequence of Theorem~\ref{thm:equivRR}, we will prove:
\begin{cor}\label{cor:DH} The rescaled weight measures (cf.~Definition~\ref{defi:weight}) 
$$
\mu_m:=(1/m)_*\mu_{H^0(X,mL)}
$$
have uniformly bounded support, and converge weakly to a probability measure $\DH_{(X,L)}$ on $\R$ as $m\to\infty$. 
Its moments are further given by
\begin{equation}\label{equ:dmoment}
\int_{\R} \la^d\,\DH_{(X,L)}(d\la)=\binom{n+d}{n}^{-1}\frac{(L_d^{n+d})}{(L^n)}
\end{equation}
for each $d\in\N$. 
\end{cor}

\begin{defi}\label{defi:DHaction} We call $\DH_{(X,L)}$ the \emph{Duistermaat-Heckman measure} of the polarized $\G_m$-scheme $(X,L)$. 
\end{defi}
For any $r\in\Z_{>0}$, the $\G_m$-action on $(X,L)$ induces an action on 
$(X,rL)$. It follows immediately from the definition that
$\DH(X,rL)=r_*\DH(X,L)$ and $\DF(X,rL)=\DF(X,L)$. 
This allows us to define the Duistermaat-Heckman measure and 
Donaldson-Futaki invariant for $\G_m$-actions on polarized
schemes $(X,L)$, where $L$ is an (ample) $\Q$-line bundle.
\begin{defi}\label{D101}
  For any polarized scheme $(X,L)$ with a $\G_m$-action,
  we define
  \begin{equation*}
    \DH(X,L):=(1/r)_*\DH(X,rL)
    \quad\text{and}\quad
    \DF(X,L):=\DF(X,rL)
  \end{equation*}
  for any sufficiently divisible $r\in\Z_{>0}$.
\end{defi}
\begin{proof}[Proof of Theorem~\ref{thm:equivRR}] Let $\pi_d:X_d\to\P^d$ be the fiber bundle defined above. The key observation is that the $\G_m$-decomposition $H^0(X,mL)=\bigoplus_{\la\in\Z} H^0(X,mL)_\la$ implies that
$$
(\pi_d)_*\cO_{X_d}(mL_d)=\bigoplus_{\la\in\Z} H^0(X,mL)_\la\otimes\cO_{\P^d}(\la).
$$
By relative ampleness of $L_d$, the higher direct images of $mL_d$ vanish for $m\gg 1$;  the Leray spectral sequence and the asymptotic Riemann-Roch theorem (cf.~\cite[\S1]{Kle}) therefore yield
$$
\sum_{\la\in\Z}\chi(\P^d,\cO_{\P^d}(\la))\dim H^0(X,mL)_\la=\chi(\P^d,(\pi_d)_*\cO_{X_d}(mL_d))
$$
$$
=\chi(X_d,mL_d)=\frac{(L_d^{n+d})}{(n+d)!}m^{n+d}+O(m^{n+d-1}). 
$$
Now $\chi(\P^d,\cO_{\P^d}(\la))=\frac{\la(\la-1)\cdots(\la-d+1)}{d!}=\frac{\la^d}{d!}+O(\la^{d-1})$, and we get the result by induction on~$d$. 
\end{proof}

\begin{proof}[Proof of Corollary~\ref{cor:DH}] Since $L$ is ample, $R(X,L)$ is finitely generated. It follows that the weights of $H^0(X,mL)$ grow at most linearly with $m$, which proves that $\mu_m$ has uniformly bounded support. Since $\mu_m$ is a probability measure, it therefore converges to a probability measure iff the moments $\int_\R\la^d\mu_m(d\la)$ converge for each $d\in\N$. We have, by definition,
$$
\int_\R\frac{\la^d}{d!}\mu_m(d\la)=\frac{1}{m^d N_m}\sum_{\la\in\Z}\frac{\la^d}{d!}\dim H^0(X,mL)_\la
$$
with $N_m=h^0(X,mL)$. Theorem~\ref{thm:equivRR} shows that
$$
\sum_{\la\in\Z}\frac{\la^d}{d!}\dim H^0(X,mL)_\la=\frac{(L_d^{n+d})}{(n+d)!}m^{n+d}+O(m^{n+d-1}),
$$
while 
$$
N_m=\frac{(L^n)}{n!}m^n+O(m^{n-1}),
$$
hence the result. 
\end{proof}

\begin{rmk}\label{rmk:DH} In order to explain the terminology,
  consider the case where $X$ is a smooth complex variety with an
  $S^1$-invariant hermitian metric on $L$ with positive curvature form
  $\om$. We then get a Hamiltonian function $H:X\to\R$ for the
  $S^1$-action on the symplectic manifold $(X,\om)$. The
  Duistermaat-Heckman measure as originally defined in~\cite{DH} is
  $H_*(\om^n)$, but this is known to coincide (up to normalization of
  the mass) with $\DH_{(X,L)}$ as defined above (see for
  instance~\cite[Theorem 9.1]{WN12} and~\cite[Proposition 4.1]{BWN}). 
  See also~\cite{Bern09,WN12,His12} for an analytic approach to 
  Duistermaat-Heckman measures via geodesic rays.
\end{rmk}

\begin{rmk}\label{rmk:Ok} When $X$ is a variety, the existence part of Corollary~\ref{cor:DH} is a rather special case of~\cite{Ok}, which also shows that $\DH_{(X,L)}$ can be written as a linear projection of the Lebesgue measure of some convex body. This implies in particular that $\DH_{(X,L)}$ is either absolutely continuous or a point mass. Its density is claimed to be piecewise polynomial on~\cite[p.1]{Ok}, but while this is a classical result of Duistermaat and Heckman when $X$ is a smooth complex variety as in Remark~\ref{rmk:DH}, we were not able to locate a proof in the literature when $X$ is singular. In particular, the proof of~\cite[Proposition 3.4]{BP} is incomplete. Piecewise polynomiality will be established in Theorem~\ref{thm:PP} below. 
\end{rmk}

We gather here the first few properties of Duistermaat-Heckman measures. 
\begin{prop}\label{prop:DHaction} Let $(X,L)$ be a polarized $\G_m$-scheme of dimension $n$, and set $V=(L^n)$. 
\begin{itemize}
\item[(i)] 
  Denote by $(X,L(\la))$ the result of twisting the action on $L$ by the character $t^\la$. 
  Then $\DH_{(X,L(\la))}=\la+\DH_{(X,L)}$. 
\item[(ii)] 
  If $X^\a$ are the irreducible components of $X$ (with their reduced scheme 
  structure), then
  \begin{equation*}
    \DH_{(X,L)}=\sum_\a c_\a\DH_{(X^\a,L|_{X^\a})},
  \end{equation*}
  where $c_\a=m_\a\frac{c_1(L)^n\cdot[X^\a]}{c_1(L)^n\cdot[X]}$, 
  with $m_\a$ the multiplicity of $X$ along $X^\a$. 
\end{itemize}
\end{prop}
Note that $c_\a>0$ iff $X^\a$ has dimension $n$, and that $\sum_\a c_\a=1$ since $[X]=\sum_\a m_\a [X^\a]$.

\begin{proof} Property (i) is straightforward. Since $X_d\to\P^d$ is locally trivial, its irreducible components are of the form $X^\a_d$, with multiplicity $m_\a$. It follows that $[X_d]\in\CH_n(X_d)=\bigoplus\Z[X^\a_d]$ decomposes as $[X_d]=\sum_\a m_\a[X^\a_d]$. Assertion (ii) is now a direct consequence of (\ref{equ:dmoment}). 
\end{proof}
%
%
\subsection{The case of a test configuration}\label{sec:DHDFtest}
We still denote by $(X,L)$ a polarized scheme (where $L$ is allowed to
be a $\Q$-line bundle), but now without any \emph{a priori} given $\G_m$-action. 
\begin{defi}\label{defi:DHsemi} Let $(\cX,\cL)$ be an ample test configuration for $(X,L)$. We define the \emph{Duistermaat-Heckman measure $\DH_{(\cX,\cL)}$ and the \emph{Donaldson-Futaki invariant} $\DF(\cX,\cL)$ of $(\cX,\cL)$} as those of the polarized $\G_m$-scheme $(\cX_0,\cL_0)$. 
\end{defi} 

\begin{defi}\label{defi:Kstab} A polarized scheme $(X,L)$ is \emph{K-semistable} if $\DF(\cX,\cL)\ge 0$ for all ample test configurations $(\cX,\cL)$. It is \emph{K-stable} if we further have $\DF(\cX,\cL)=0$ only when $(\cX,\cL)$ is almost trivial in the sense of Definition~\ref{defi:triv}.
\end{defi}
 \begin{prop}\label{prop:DHDFsemi} 
  Let $(\cX,\cL)$ be an ample test configuration for $(X,L)$, with $\G_m$-equivariant 
  projection $\pi\colon\cX\to\A^1$ and compactification $(\bar\cX,\bar\cL)$, 
  and set $V:=(L^n)$. 
  \begin{itemize}
  \item[(i)] 
    For each $c\in\Q$, we have $\DH_{(\cX,\cL+c\cX_0)}=\DH_{(\cX,\cL)}+c$. 
  \item[(ii)] 
    Let $X^\a$ be the top-dimensional irreducible components of $X$ 
    (with their reduced scheme structure), and $m_\a$ be the multiplicity of $X$ along $X^\a$. 
    Then 
    \begin{equation*}
      \DH_{(\cX,\cL)}=\sum_\a c_\a\DH_{(\cX^\a,\cL|_{\cX^\a})},
    \end{equation*}
    where $\cX^\a$ is the irreducible component of $\cX$ corresponding to
    $X^\a$, and where $c_\a=V^{-1} m_\a c_1(L)^n\cdot[X^\a]$. 
  \item[(iii)] 
    The barycenter of the Duistermaat-Heckman measure satisfies
    \begin{equation}\label{equ:DHbar}
      \int_{\R} \la\,\DH_{(\cX,\cL)}(d\la)=\frac{(\bar\cL^{n+1})}{(n+1)V}
    \end{equation}
  \item[(iv)] 
    If $\cX$ (and hence $X$) is normal, then 
    \begin{equation}\label{equ:DF}
      \DF(\cX,\cL)=\frac{(K_{\bar\cX/\P^1}\cdot\bar\cL^n)}{V}+\bar S\frac{(\bar\cL^{n+1})}{(n+1)V}
    \end{equation}
    with $\bar S:=nV^{-1}(-K_X\cdot L^{n-1})$.  
  \end{itemize}
\end{prop}
In (iv), $K_X$ and $K_{\bar\cX/\P^1}=K_{\bar\cX}-\pi^*K_{\P^1}$ are understood as Weil divisor classes on the normal schemes $X$ and $\bar\cX$, respectively. This intersection theoretic expression is originally due to~\cite{Wan,Oda2}, see also~\cite{LX}. 

\begin{rmk} When $X$ is smooth, $k=\C$ and $L$ is a line bundle, $\bar S$ is the mean value of the scalar curvature $S(\om)$ of any K\"ahler form $\om\in c_1(L)$ (hence the chosen notation). 
\end{rmk}

\begin{proof}[Proof of Proposition~\ref{prop:DHDFsemi}] 
  After passing to a multiple, we may assume that $L$ and $\cL$ are line bundles.
  By flatness of $\cX\to\A^1$, the decomposition $[X]=\sum_\a m_\a[X^\a]$ in 
  $\CH_n(\cX)$ implies $[\cX_0]=\sum_\a m_\a[\cX^\a_0]$, where $\cX^\a_0$ 
  denotes the (possibly reducible) central fiber of $\cX^\a$. We now get (ii) as a 
  consequence of Proposition~\ref{prop:DHaction}, which also implies (i). 

  We now turn to the proof of the last two points.  
  By relative ampleness, $\pi_*\cO_{\bar\cX}(m\bar\cL)$ is a vector bundle on $\P^1$ of rank 
  $N_m=h^0(mL)$ for $m\gg 1$, with fiber at $0$ isomorphic to $H^0(\cX_0,m\cL_0)$. 
  As a result, $w_m$ is the weight of $\det\pi_*\cO_{\bar\cX}(m\cL)_0$, and hence
  \begin{equation*}
    w_m
    =\deg\det\pi_*\cO_{\bar\cX}(m\bar\cL)
    =\deg\pi_*\cO_{\bar\cX}(m\bar\cL),
  \end{equation*}
  since $\pi_*\cO_{\bar\cX}(m\bar\cL)$ is $\G_m$-equivariantly trivial away from $0$ by construction 
  of the compactification. By the usual Riemann-Roch theorem on $\P^1$, we infer
  \begin{equation*}
    w_m= \chi(\P^1,\pi_*\cO_{\bar\cX}(m\bar\cL))-N_m. 
  \end{equation*}
  By relative ampleness again, the higher direct images of $m\cL$ vanish for $m$ divisible enough, 
  and the Leray spectral sequence and the asymptotic Riemann-Roch theorem give as in the 
  proof of Theorem~\ref{thm:equivRR} 
  \begin{equation*}
    w_m=\chi(\bar\cX,m\bar\cL)-N_m=\frac{m^{n+1}}{(n+1)!}(\bar\cL^{n+1})+O(m^n), 
  \end{equation*}
  which yields (iii) since $N_m=\frac{m^n}{n!}V+O(m^{n-1})$. 
  
  When $\cX$ (and hence $\bar\cX$) is normal, the two-term asymptotic 
  Riemann-Roch theorem on a normal variety (cf.~Theorem~\ref{thm:RR} in the appendix) yields
  \begin{equation*}
    N_m=V\frac{m^n}{n!}\left[1+\frac{\bar S}{2} m^{-1}+O(m^{-2})\right], 
  \end{equation*}
  and
  \begin{align*}
    w_m&=-N_m+\frac{(\bar\cL^{n+1})}{(n+1)!}m^{n+1}
    -\frac{(K_{\bar\cX}\cdot\bar\cL^n)}{2n!}m^{n}+O(m^{n-1})\\
    &=\frac{(\bar\cL^{n+1})}{(n+1)!}m^{n+1}
      -\frac{(K_{\bar\cX/\P^1}\cdot\bar\cL^n)}{2n!}m^{n}+O(m^{n-1}),
  \end{align*}
  using that $(\pi^*K_{\P^1}\cdot\bar\cL^n)=-2V$ since $\deg K_{\P^1}=-2$. 
  The formula for $\DF(\cX,\cL)$ in (iv) now follows from a straightforward computation.
\end{proof}
%
%
\subsection{Behavior under normalization}\label{sec:DHDFnormal}
We now study the behavior of Duistermaat-Heckman measures and Donaldson-Futaki invariants under normalization. 

Recall that the normalization of the polarized scheme $(X,L)$ is the normal polarized scheme $(\tX,\tL)$ obtained by setting $\tL=\nu^*L$ with $\nu\colon\tX\to X$ the normalization morphism. Similarly, the normalization of an ample test configuration $(\cX,\cL)$ for $(X,L)$ is the ample test configuration $(\tcX,\tcL)$ for $(\tX,\tL)$ obtained by $\tcL=\nu^*\cL$ with $\nu\colon\tcX\to\cX$ the normalization morphism. 

We first prove that Duistermaat-Heckman measures are invariant under normalization, 
in the reduced case. 
\begin{thm}\label{thm:DHinv} 
  If $X$ is reduced, then $\DH_{(\cX,\cL)}=\DH_{(\tcX,\tcL)}$ for every ample test configuration
  $(\cX,\cL)$ for $(X,L)$.
\end{thm}
\begin{proof} 
  By Proposition~\ref{prop:DHDFsemi}~(ii), after twisting the $\G_m$-action on $\cL$ 
  by $t^\la$ with $\la\gg 1$, we may assume $\mu:=\DH_{(\cX,\cL)}$ and 
  $\tmu:=\DH_{(\tcX,\tcL)}$ are supported in $\R_+$. For $m$ divisible enough, let 
  \begin{equation*}
    \mu_m:=(1/m)_*\mu_{\pi_*\cO_\cX(m\cL)_0}
  \end{equation*} 
  be the scaled weight measure of the $\G_m$-module
  $\pi_*\cO_\cX(m\cL)_0$. Thus $\mu_m$ converges weakly to $\mu$ by 
  Proposition~\ref{prop:DHDFsemi}~(i). By Lemma~\ref{lem:tail}, the tail distribution of 
  $\mu_m$ is given by
  \begin{equation*}
    \mu_m\{x\ge\la\}=\frac{1}{N_m}\dim F^{\lceil m\la\rceil} H^0(X,mL), 
  \end{equation*}
  where $F^\bullet H^0(X,mL)$ is the Rees filtration induced by $(\cX,\cL)$, 
  and $N_m=H^0(X,mL)=\dim H^0(\cX_0,m\cL_0)$ for $m\gg 1$, by flatness and Serre vanishing. 

  Denoting by $\tmu_m$ and $F^\bullet H^0(\tX,m\tL)$ the scaled weight measure and filtration 
  defined by $(\tcX,\tcL)$, we similarly have 
  \begin{equation*}
    \tmu_m\{x\ge\la\}=\frac{1}{\tN_m}\dim F^{\lceil m\la\rceil} H^0(\tX,m\tL).
  \end{equation*}
  Since $\cX$ is reduced by Proposition~\ref{prop:basic}, the canonical morphism 
  $\cO_\cX\to\nu_*\cO_{\tcX}$ is injective, and the projection formula yields a $\G_m$-equivariant 
  inclusion
  $H^0(\cX,m\cL)\hookrightarrow H^0(\tcX,m\tcL)$. For each $\la\in\Z$, we thus have 
  $H^0(\cX,m\cL)_\la\hookrightarrow H^0(\tcX,m\tcL)_\la$, 
  and hence $F^\la  H^0(X,mL)\hookrightarrow F^\la  H^0(\tX,m\tL)$, which implies
  \begin{equation*}
    \tmu_m\{x\ge\la\}\ge\frac{\tN_m}{N_m}\mu\{x\ge\la\}.
  \end{equation*}
  Since $X$ is reduced, $\nu_*\cO_{\tX}/\cO_X$ is supported on a nowhere dense Zariski 
  closed subset, and hence 
  \begin{equation*}
    \tN_m=h^0(\tX,m\tL)=h^0(X,\cO_X(mL)\otimes\nu_*\cO_{\tX})=N_m+O(m^{n-1}).
  \end{equation*}
  Since the weak convergence of probability measures $\mu_m\to\mu$ implies 
  (in fact, is equivalent to) the a.e.\ convergence of the tail distributions, we conclude
  \begin{equation}\label{e203}
    \tmu\{x\ge\la\}\ge\mu\{x\ge\la\}
  \end{equation}
  for a.e.\ $\la\in\R$. 
  
  By (iii) of Proposition~\ref{prop:DHDFsemi} and the projection formula, $\mu$ and $\tmu$ 
  have the same barycenter $\bar\lambda$, and hence 
  \begin{equation}\label{e204}
    \int_{\R_+}\mu\{x\ge\la\}d\la
    =\int_{\R_+}\la\,d\mu
    =\bar\lambda
    =\int_{\R_+}\la\,d\tmu
    =\int_{\R_+}\tmu\{x\ge\la\}d\la
  \end{equation}
  since $\mu$ and $\tmu$ are supported in $\R_+$. By (\ref{e203}), we thus have 
  $\tmu\{x\ge\la\}=\mu\{x\ge\la\}$ for a.e.\ $\la\in\R$, and hence $\tmu=\mu$ 
  (by taking for instance the distributional derivatives), which concludes the proof.
\end{proof}

Regarding Donaldson-Futaki invariants, we prove the following explicit version of~\cite[Proposition 5.1]{RT} and~\cite[Corollary 3.9]{ADVLN}.

\begin{prop}\label{prop:DFnorm} 
  Let $(\cX,\cL)$ be an ample test configuration
  for a polarized scheme $(X,L)$. Let $\cX'$ be another test configuration for $X$ dominating $\cX$, such that $\mu\colon\cX'\to\cX$ is finite, and set $\cL':=\mu^*\cL$. Then 
  \begin{equation*}
    \DF(\cX,\cL)=\DF(\cX',\cL')+2 V^{-1}\sum_E m_E\left(E\cdot\cL^n\right),
  \end{equation*}
  where $E$ ranges over the irreducible components of $\cX_0$ 
  contained in the singular locus of $\cX$ and $m_E\in\N^*$ is the length of the sheaf 
  $\cF:=\left(\mu_*\cO_{\cX'}\right)/\cO_{\cX}$ at the generic point of $E$. 
\end{prop}
When $X$ is normal, the result applies to the normalization of a test
configuration; hence

\begin{cor}\label{cor:Kstabnorm} If $X$ is normal, then $(X,L)$ is K-semistable iff $\DF(\cX,\cL)\ge 0$ for all \emph{normal} ample test configurations.
\end{cor}

\begin{proof}[Proof of Proposition~\ref{prop:DFnorm}] 
  Let $m$ be sufficiently divisible. 
  Denoting by $w_m$ and $w'_m$ the $\G_m$-weights of 
  $\det H^0(\cX_0,m\cL_0)$ and $\det H^0(\cX'_0,m\cL'_0)$,
  the proof of Proposition~\ref{prop:DHDFsemi} yields
  $$
  w'_m-w_m
  =\chi(\bar\cX',m\bar\cL')-\chi(\bar\cX,m\bar\cL). 
  $$
  Since $\mu$ is finite, we have $R^q\mu_*\cO_{\bar\cX'}=0$ for all $q\ge 1$, 
  and the Leray spectral sequence gives
  \begin{equation*}
    \chi(\bar\cX',m\bar\cL')
    =\chi\left(\bar\cX,\cO_\cX(m\bar\cL)\otimes\mu_*\cO_{\bar\cX'}\right).
  \end{equation*}
  By additivity of the Euler characteristic in exact sequences and~\cite[\S2]{Kle}, we infer
  \begin{align*}
    w'_m-w_m
    &=\chi\left(\bar\cX,\cO_{\bar\cX}(m\bar\cL)\otimes\cF\right)\\
    &=\frac{m^n}{n!}\sum_E m_E\left(E\cdot\cL^n\right)+O(m^{n-1}), 
  \end{align*}
  which yields the desired result in view of Definition~\ref{defi:DFaction}. 
\end{proof}
%
%
\subsection{The logarithmic case}\label{sec:DFlog}
Assume that $X$ is normal, let $B$ be a boundary on $X$ and write $K_{(X,B)}:=K_X+B$
(see~\S\ref{sec:bound}). 
Let $L$ be an ample $\Q$-line bundle on $X$.
We then introduce a log version of the `mean scalar curvature' $\bar S$ by setting 
\begin{equation*}
  \bar S_B:=nV^{-1}\left(-K_{(X,B)}\cdot L^{n-1}\right). 
\end{equation*}
If $\cX$ is a normal test configuration for $X$, denote by $\cB$ (resp.\ $\bar\cB$)
the $\Q$-Weil divisor on $\cX$ (resp.\ $\bar\cX$) obtained as the (component-wise) 
Zariski closure in $\cX$ (resp.\ $\bar\cX$) of the $\Q$-Weil divisor 
$B\times(\A^1\setminus\{0\})$ with respect to the open embedding 
of $X\times(\A^1\setminus\{0\})$ into $\cX$ (resp. $\bar\cX$).
We then set 
\begin{equation*}
  K_{(\cX,\cB)}:=K_\cX+\cB,
  \quad
  K_{(\bar\cX,\bar\cB)}:=K_{\bar\cX}+\bar\cB,
\end{equation*}
and
\begin{equation*}
  K_{(\cX,\cB)/\A^1}:=K_{(\cX,\cB)}-\pi^*K_{\A^1},
  \quad
  K_{(\bar\cX,\bar\cB)/\P^1}:=K_{(\bar\cX,\bar\cB)}-\pi^*K_{\P^1}.
\end{equation*}
Note that these $\Q$-Weil divisor (classes) may not be $\Q$-Cartier in general. 

The intersection theoretic formula for $\DF$ in Proposition~\ref{prop:DHDFsemi} suggests the following generalization for pairs (compare~\cite[Theorem 3.7]{OSu}, see also~\cite{Don4,LS}). 

\begin{defi}\label{defi:DFpairs} 
  Let $B$ be a boundary on $X$. For each normal test configuration $(\cX,\cL)$ for $(X,L)$, 
  we define the \emph{log Donaldson-Futaki invariant} of $(\cX,\cL)$ as 
  \begin{equation*}
    \DF_B(\cX,\cL):=V^{-1}(K^B_{\bar\cX/\P^1}\cdot\bar\cL^n)
    +\bar S_BV^{-1}\frac{(\bar\cL^{n+1})}{n+1}, 
  \end{equation*}
\end{defi}

In view of Corollary~\ref{cor:Kstabnorm}, we may then introduce the following notion:
\begin{defi}\label{defi:logKstab} 
  A polarized pair $((X,B);L)$ is \emph{K-semistable} if $\DF_B(\cX,\cL)\ge 0$ for all 
  normal ample test configurations. 
  It is \emph{K-stable} if we further have $\DF_B(\cX,\cL)=0$ only when $(\cX,\cL)$ is trivial. 
\end{defi}
Note that $((X,B);L)$ is K-semistable (resp.\ K-stable) iff 
$((X,B);rL)$ is K-semistable (resp.\ K-stable) for some (or,
equivalently, any) $r\in\Z_{>0}$.

\begin{rmk}\label{rmk:slc} Let $X$ be a \emph{deminormal} scheme, \ie
  reduced, of pure dimension $n$, $S_2$ and with at most normal
  crossing singularities in codimension one, and let $\nu\colon\tX\to X$ be
  the normalization. If $K_X$ is $\Q$-Cartier, then
  $\nu^*K_X=K_{\tX}+\tB$, where $\tB$ denotes the inverse image of the
  conductor, and is a reduced Weil divisor on $\tX$ by the
  deminormality assumption. By definition, $X$ has \emph{semi-log
    canonical singularities} (slc for short) if $(\tX,\tB)$ is
  lc. (See~\cite[\S5]{KollarBook} for details.)

Now let $L$ be an ample $\Q$-line bundle on $X$, and let $(\cX,\cL)$ be an ample test configuration for $(X,L)$, with normalization $(\tcX,\tcL)$. In~\cite[Proposition 3.8]{Oda2} and~\cite[\S5]{Oda3}, Odaka introduces the \emph{partial normalization} $\tcX\to\cX'\to\cX$ by requiring that 
$$
\cO_{\cX'}=\cO_{\tcX}\cap \cO_{X\times\left(\A^1\setminus\{0\}\right)}.
$$
We get this way an ample test configuration $(\cX',\cL')$ for $(X,L)$, with the extra property that $\tcX\to\cX'$ is an isomorphism over the generic points of $\cX'_0$, cf.~\cite[Lemma 3.9]{Oda2}. Arguing as in the proof of Proposition~\ref{prop:DFnorm}, we may then check that 
\begin{equation}\label{equ:DFslc}
\DF_{\tB}(\tcX,\tcL)=\DF(\cX',\cL')\le\DF(\cX,\cL).
\end{equation}
This shows that $(X,L)$ is K-semistable iff $\DF_{\tB}(\tcX,\tcL)\ge 0$ for the normalization $(\tcX,\tcL)$ of every ample test configuration $(\cX,\cL)$ for $(X,L)$. 
\end{rmk}
%
%
\section{Valuations and test configurations}\label{sec:valtest}
In what follows, $X$ denotes a normal variety of dimension $n$,
with function field $K=k(X)$. 
The function field of any test configuration for $X$ is then isomorphic to $K(t)$.
We shall relate valuations on $K$ and $K(t)$ from both an algebraic
and geometric point of view.
%
%
\subsection{Restriction and Gauss extension}
First consider a valuation $w$ on $K(t)$. 
We denote by $r(w)$  its restriction to $K$.
\begin{lem}\label{lem:div} 
  If $w$ is an Abhyankar valuation, then so is $r(w)$.
  If $w$ is divisorial, then $r(w)$ is either divisorial
  or trivial.
\end{lem}
\begin{proof}
  The first assertion follows from Abhyankar's inequality (\ref{equ:Abhy}).
  Indeed, if $w$ is Abhyankar, then $\trdeg(w)+\ratrk(w)=n+1$,
  so~\eqref{equ:Abhy} gives $\trdeg(r(w))+\ratrk(r(w))\ge n$.
  As the opposite inequality always holds, we must have
  $\trdeg(r(w))+\ratrk(r(w))=n$, \ie $r(w)$ is Abhyankar.
  
  We also have $\trdeg(r(w))\le\trdeg(w)$, so if $w$ is divisorial, then
  $\trdeg(r(w))=n$ or $\trdeg(r(w))=n-1$, corresponding to $r(w)$
  being trivial or divisorial, respectively.
\end{proof}

\medskip
The restriction map $r$ is far from injective, but we can 
construct a natural one-sided inverse by exploiting the 
$k^*$-action (or $\G_m$-action) on $K(t)=k(X_{\A^1})$
defined by $(a\cdot f)(t)=f(a^{-1}t)$ for $a\in k^*$ and $f\in K(t)$. 
In terms of the Laurent polynomial expansion
\begin{equation}\label{equ:Laurent}
f=\sum_{\la\in\Z}f_\la t^\la
\end{equation}
with $f_\la\in K$, the $k^*$-action on $K(t)$ reads
\begin{equation}\label{equ:action}
a\cdot f=\sum_{\la\in\Z}a^{-\la}f_\la t^\la.
\end{equation}

\begin{lem}\label{lem:gauss} A valuation $w$ on $K(t)$ is $k^*$-invariant iff 
\begin{equation}\label{equ:gauss}
w(f)=\min_{\la\in\Z}\left(r(w)(f_\la)+\la w(t)\right). 
\end{equation}
for all $f\in K(t)$ with Laurent polynomial expansion (\ref{equ:Laurent}).  In particular, $r(w)$ is trivial iff $w$ is the multiple of the $t$-adic valuation. 
\end{lem}
\begin{proof} 
  In view of (\ref{equ:action}), it is clear that (\ref{equ:gauss}) implies $k^*$-invariance. 
  Conversely let $w$ be a $k^*$-invariant valuation on $K(t)$. 
  The valuation property of $w$ shows that
  $$
  w(f)\ge\min_{\la\in\Z}\left(r(w)(f_\la)+\la w(t)\right)
  $$
Set $\Lambda:=\left\{\la\in\Z\mid f_\la\ne 0\right\}$ and pick distinct elements $a_\mu\in k^*$, $\mu\in\Lambda$ (recall that $k$ is algebraically closed, and hence infinite). The Vandermonde  matrix $(a_\mu^\la)_{\la,\mu\in\Lambda}$ is then invertible, and each term $f_\la t^\la$ with $\la\in\Lambda$ may thus be expressed as $k$-linear combination of $(a_\mu\cdot f)_{\mu\in\Lambda}$. Using the valuation property of $w$ again, we get for each $\la\in\La$ 
$$
r(w)(f_\la)+\la w(t)=w\left(f_\la t^\la\right)\ge\min_{\mu\in\Lambda} w(a_\mu\cdot f)=w(f), 
$$
where the right-hand equality holds by $k^*$-invariance of $w$. The result follows. 
\end{proof}

\begin{defi}\label{defi:Gauss} The \emph{Gauss extension} of a valuation $v$ on $K$ is the valuation $G(v)$ on $K(t)$ defined by
$$
G(v)(f)=\min_{\la\in\Z}\left(v(f_\la)+\la\right)
$$
for all $f$ with Laurent polynomial expansion (\ref{equ:Laurent}). 
\end{defi}
Note that $r(G(v))=v$ for all valuations $v$ on $K$, while a valuation $w$ on $K(t)$ satisfies $w=G(r(w))$ iff it is $k^*$-invariant and $w(t)=1$, by Lemma~\ref{lem:gauss}. 
Further, the Gauss extension of $v$ is the smallest extension $w$ with $w(t)=1$.
%
%
\subsection{Geometric interpretation}
We now relate the previous algebraic considerations to test configurations. For each test configuration $\cX$ for $X$, the canonical birational map $\cX\dashrightarrow X_{\A^1}$ yields an isomorphism $k(\cX)\simeq K(t)$. When $\cX$ is normal, every irreducible component $E$ of $\cX_0$ therefore defines a divisorial valuation $\ord_E$ on $K(t)$. 

\begin{defi}\label{defi:nontriv} Let $\cX$ be a normal test
  configuration for $X$. For each irreducible component $E$ of $\cX_0$, we set $v_E:=b_E^{-1}r(\ord_E)$ with $b_E=\ord_E(\cX_0)=\ord_E(t)$. We say that $E$ is \emph{nontrivial} if it is not the strict transform of $X\times\{0\}$.  
\end{defi}
Since $E$ is preserved under the $\G_m$-action on $\cX$, $\ord_E$ is $k^*$-invariant, and we infer from Lemma~\ref{lem:div} and Lemma~\ref{lem:gauss}:
\begin{lem}\label{lem:div2} For each irreducible component $E$ of $\cX_0$, we have $b_E^{-1}\ord_E=G(v_E)$, \ie
$$
b_E^{-1}\ord_E(f)=\min_\la\left(v_E(f_\la)+\la\right).
$$
in terms of the Laurent polynomial expansion (\ref{equ:Laurent}). Further, $E$ is nontrivial iff $v_E$ is nontrivial, and hence a divisorial valuation on $X$. 
\end{lem}

By construction, divisorial valuations on $X$ of the form $v_E$ have a value group $\Ga_v=v(K^*)$ contained in $\Q$. Thus they are of the form $v_E=c\,\ord_F$ with $c\in\Q_{>0}$ and $F$ a prime divisor on a normal variety $Y$ mapping birationally to $X$. Conversely, we prove:
\begin{thm}\label{thm:restrdiv} A divisorial valuation $v$ on $X$ is
  of the form $v=v_E$ for a non-trivial irreducible component $E$ of a normal test configuration iff $\Ga_v$ is contained in $\Q$. In this case, we may recover $b_E$ as the denominator of the generator of $\Ga_v$. 
\end{thm}

\begin{lem}\label{lem:divtest} A divisorial valuation $w$ on $K(t)$ satisfying $w(t)>0$ is $k^*$-invariant iff $w=c\,\ord_E$ with $c>0$ and $E$ an irreducible component of the central fiber $\cX_0$ of a normal test configuration $\cX$ of $X$. 
\end{lem}
\begin{proof} If $E$ is an irreducible component of $\cX_0$, then $\ord_E(t)>0$, and the $\G_m$-invariance of $E$ easily implies that $\ord_E$ is $k^*$-invariant. Conversely, let $w$ be a $k^*$-invariant divisorial valuation on $K(t)$ satisfying $w(t)>0$. The center $\xi$ on $X\times\A^1$ is then $\G_m$-invariant and contained in $X\times\{0\}$. If we let $\cY_1$ be the test configuration obtained by blowing-up the closure of $\xi$ in $X\times\A^1$, then the center $\xi_1$ of $w$ on $\cY_1$ is again $\G_m$-invariant by $k^*$-invariance of $w$, and the blow-up $\cY_2$ of the closure of $\xi_1$ is thus a test configuration. Continuing this way, we get a tower of test configurations 
$$
X\times\A^1\leftarrow\cY_1\leftarrow\cY_2\leftarrow\dots\leftarrow\cY_i\leftarrow\dots
$$
Since $w$ is divisorial, a result of Zariski (cf.~\cite[Lemma 2.45]{KM}) guarantees that the closure of the center $\xi_i$ of $w$ on $\cY_i$ has codimension $1$ for $i\gg 1$. We then have $w=c\,\ord_E$ with $E$ the closure of the center of $w$ on the normalization $\cX$ of $\cY_i$. 
\end{proof}

\begin{proof}[Proof of Theorem~\ref{thm:restrdiv}] Let $E$ be a non-trivial irreducible component of $\cX_0$ for a normal test configuration $\cX$ of $X$. Since the value group of $\ord_E$ on $k(\cX)=K(t)$ is $\Z$, the value group of $v_E$ on $k(X)=K$ is of the form $\frac{c}{b_E}\Z$ for some positive integer $c$. Lemma~\ref{lem:div2} yields $\Z=c\Z+b_E\Z$, so that $c$ and $b_E$ are coprime. 

Conversely, let $v$ be a divisorial valuation on $X$ with
$\Ga_v=\frac{c}{b}\Z$ for some coprime positive integers $b,c$. Then
$w:=b G(v)$ is a $k^*$-invariant divisorial valuation on $K(t)$ with
value group $c\Z+b\Z=\Z$. By Lemma~\ref{lem:divtest}, we may thus find
a normal test configuration $\cX$ for $X$ and a non-trivial
irreducible component $E$ of $\cX_0$ such that $\ord_E=w$. We then have $b_E=w(t)=b$, and hence $v=v_E$.  
\end{proof}

%
%
%
%
\subsection{Rees valuations and deformation to the normal cone}\label{sec:reesdef}
Our goal in this section is to relate the Rees valuations of a closed subscheme $Z\subset X$ to the valuations associated to the normalization of the deformation to the normal cone of $Z$,
see Example~\ref{ex:defnorm}. 

\begin{thm}\label{thm:reescone} Let $Z\subset X$ be a closed
  subscheme, $\cX$ the deformation to the normal cone of $Z$, and
  $\tcX$ its normalization, so that $\mu\colon\tcX\to X_{\A^1}$ is the
  normalized blow-up of $Z\times\{0\}$. Then the Rees valuations of
  $Z$ coincide with the valuations $v_E$, where $E$ runs over the
  non-trivial irreducible irreducible components of $\tcX_0$. 
\end{thm}
In other words, the Rees valuations of $Z$ are obtained by restricting to $k(X)\subset k(X)(t)$ those of $Z\times\{0\}$. 

If we denote by $E_0$ the strict transform of $X\times\{0\}$ in $\cX$, one can show that $\cX\setminus E_0$ is isomorphic to the Spec over $X$ of the \emph{extended Rees algebra} $\cO_X[ t^{-1}\fa, t]$, where $\fa$ is the ideal of $Z$, cf.~\cite[pp.87--88]{Ful}. We thus see that Theorem~\ref{thm:reescone} is equivalent to the well-known fact that the Rees valuations of $\fa$ coincide with the restrictions to $X$ of the Rees valuations of the principal ideal $(t)$ of the extended Rees algebra (see for instance~\cite[Exercise 10.5]{HS}). We nevertheless provide a proof for the benefit of the reader. 

\begin{lem}\label{lem:intinv} Let $\fb=\sum_{\la\in\N}\fb_\la t^\la$ be a $\G_m$-invariant ideal of $X\times\A^1$, and let 
$$
\overline{\fb}=\sum_{\la\in\N}(\overline{\fb})_\la t^\la
$$
be its integral closure. For each $\la$ we then have $\overline{\fb_\la}\subset(\overline{\fb})_\la$, with equality for $\la=0$. 
\end{lem}
\begin{proof} 
Each $f\in\overline{\fb_\la}$ satisfies a monic equation $f^d+\sum_{j=1}^d b_j f^{d-j}=0$ with $b_j\in\fb_\la^j$. Then 
$$
(t^\la f)^d+\sum_{j=1}^d(t^{\la j}b_j)(t^\la f)^{d-j}=0
$$ 
with $t^{\la j}b_j\in (t^\la\fb_\la)^j\subset\fb^j$. It follows that $t^\la f\in\overline{\fb}$, which proves the first assertion. 

Conversely, we may choose $l\gg 1$ such that the $\G_m$-invariant ideal $\fc:=\overline{\fb^l}$ satisfies $\overline{\fb}\cdot\fc=\fb\cdot\fc$ (\cf proof of Lemma~\ref{lem:normblow}). Write $\fc=\sum_{\la\ge\la_0}\fc_\la t^\la$ with $\fc_{\la_0}\ne 0$. Then $(\overline{\fb})_0\cdot\fc_{\la_0} t^{\la_0}$ is contained in the weight $\la_0$ part of $\fb\cdot\fc$, which is equal to $(\fb_0\cdot\fc_{\la_0}) t^{\la_0}$. We thus have $(\overline{\fb})_0\cdot\fc_{\la_0}\subset\fb_0\cdot\fc_{\la_0}$, and hence $(\overline{\fb})_0\subset\overline{\fb_0}$ by the determinant trick. 
\end{proof}

\begin{proof}[Proof of Theorem~\ref{thm:reescone}]
Let $\fa$ be the ideal defining $Z$.
By Theorem~\ref{thm:rees}, we are to check that:
\begin{itemize}
\item[(i)] $\overline{\fa^m}=\bigcap_E\left\{f\in\cO_X\mid v_E(f)\ge m\right\}$ for all $m\in\N$; 
\item[(ii)] no $E$ can be omitted in (i). 
\end{itemize}
Set $D:=\mu^{-1}(Z\times\{0\})$. Since $\ord_E$ is $k^*$-invariant, Lemma~\ref{lem:gauss} yields
$$
\ord_E(D)=\ord_E(\fa+(t))=\min\{r(\ord_E)(\fa),b_E\}.
$$
We claim that we have in fact $\ord_E(D)=b_E$. As recalled in Example~\ref{ex:defnorm}, the blow-up $\rho\colon\cX\to X\times\A^1$ along $Z\times\{0\}$ satisfies $\cX_0=\mu^{-1}(Z\times\{0\})+F$, with $F$ the strict transform of $X\times\{0\}$. Denoting by $\nu\colon\tcX\to\cX$ the normalization morphism, we infer $\tcX_0=D+\nu^*F$, and hence $b_E=\ord_E(\tcX_0)=\ord_E(D)$.

This shows in particular that the valuations $b_E^{-1}\ord_E$ are the Rees valuations of $\fa+(t)$. We also get that $v_E(\fa)=b_E^{-1}r(\ord_E)(\fa)\ge 1$, and hence $\overline{\fa^m}\subset\bigcap_E\left\{f\in\cO_X\mid v_E(f)\ge m\right\}$. Conversely, assume $f\in\cO_X$ satisfies $v_E(f)\ge m$ for all $E$. Since the $b_E^{-1}\ord_E$ are the Rees valuations of $\fa+(t)$, applying Theorem~\ref{thm:rees} on $X\times\A^1$ yields $f\in\overline{(\fa+(t))^m}$. Since $\fa^m$ is the weight $0$ part of $(\fa+(t))^m$, Lemma~\ref{lem:intinv} yields $f\in\overline{\fa^m}$, and we have thus established (i). 

Finally, let $S$ be any finite set of $k^*$-invariant valuations $w$ on $K(t)$ such that 
$$
\overline{\fa^m}=\bigcap_{w\in S}\left\{f\in\cO_X\mid r(w)(f)\ge m\right\}
$$
for all $m\in\N$. We claim that we then have 
$$
\overline{(\fa+(t))^m}=\bigcap_{w\in S}\left\{f\in\cO_\cX\mid w(f)\ge m\right\}
$$
for all $m\ge\N$. This will prove (ii), by the minimality of the set of Rees valuations of $\fa+(t)$. So assume that $f\in\cO_\cX$ satisfies $w(f)\ge m$ for all $w\in S$. In terms of the Laurent expansion (\ref{equ:Laurent}), we get $r(w)(f_\la)+\la\ge m$ for all $\la$, $w$, and hence $f_\la\in\overline{\fa^{m-\la}}$ by assumption. By Lemma~\ref{lem:intinv}, we conclude as desired that $f\in\overline{(\fa+(t))^m}$. 
\end{proof}

\begin{cor}\label{cor:rees} Let $(X,L)$ be a normal polarized variety and $Z\subset X$ a 
closed subscheme. Then there exists a normal, ample test configuration $(\cX,\cL)$ such that the Rees valuations of $Z$ are exactly the divisorial valuations $v_E$ on $X$ associated to the non-trivial irreducible components of $\cX_0$. 
\end{cor}
\begin{proof} Let $\mu\colon\cX\to X\times\A^1$ be the normalized blow-up of $Z\times\{0\}$, so that $\cX$ is the normalization of the deformation to the normal cone of $Z$. As recalled in Lemma~\ref{lem:normblow}, $D:=\mu^{-1}(Z)$ is a Cartier divisor with $-D$ ample. We may thus choose $0<c\ll 1$ such that $\cL:=\mu^*L_{\A^1}-c D$ is ample, and $(\cX,\cL)$ is then a normal, ample test configuration. The rest follows from Theorem~\ref{thm:reescone}. 
\end{proof}
%
%
%
%
\subsection{Log discrepancies and log canonical divisors}\label{sec:logdisc}
In this section we assume that $k$ has characteristic $0$. 
Let $B$ be a boundary on $X$. Recall the definition of $A_{(X,B)}$ from~\S\ref{sec:bound}.
\begin{prop}\label{prop:discr}
  For every irreducible component $E$ of $\cX_0$, the log discrepancies of $v_E$ 
  and $\ord_E$ (with respect to the pairs $(X,B)$ and $(X_{\A^1},B_{\A^1})$, respectively) 
  are related by
  \begin{align*}
    A_{(X,B)}(v_E)
    &=A_{\left(X_{\A^1},B_{\A^1}\right)}\left(b_E^{-1}\ord_E\right)-1\\
    &=A_{\left(X_{\A^1},B_{\A^1}+X\times\{0\}\right)}(b_E^{-1}\ord_E). 
  \end{align*}
\end{prop}
Recall that $A_{(X,B)}(v_\triv)$ is defined to be $0$, and that $b_E=\ord_E(\cX_0)=\ord_E(t)$. 
\begin{proof} 
  If $E$ is the strict transform of $X\times\{0\}$, then $A_{(X,B)}(v_E)=A_{(X,B)}(v_\triv)=0$, 
  while $A_{(X_{\A^1},B_{\A^1})}(\ord_E)=b_E=1$. 

  Assume now that $E$ is non-trivial. Since $v_E$ is a divisorial valuation on $X$, 
  we may find a proper birational morphism $\mu\colon X'\to X$ with $X'$ smooth 
  and a smooth irreducible divisor $F\subset X'$ such that $v_E=c\ord_F$ for some 
  rational $c>0$. By Lemma~\ref{lem:div2}, the divisorial valuation $\ord_E$ is monomial 
  on $X'_{\A^1}$ with respect to the snc divisor $X'\times\{0\}+F_{\A^1}$, with weights 
  $\ord_E(X'\times\{0\})=b_E$ and $\ord_E(F_{\A^1})=b_Ev_E(F)=b_E c$. It follows 
  (see~\eg~\cite[Prop.~5.1]{JM}) that 
  \begin{multline*}
    A_{(X_{\A^1},B_{\A^1})}(\ord_E)
    =b_E A_{(X_{\A^1},B_{\A^1})}\left(\ord_{X\times\{0\}}\right)
    +b_E c A_{(X_{\A^1},B_{\A^1})}\left(\ord_{F_{\A^1}}\right)\\
    =b_E+b_E c A_{(X,B)}\left(\ord_F\right)
    =b_E\left(1+A_{(X,B)}(v_E)\right),
  \end{multline*}
  which completes the proof.
\end{proof}

Now consider a normal test configuration $\cX$ for $X$, with compactification $\bar\cX$.
As in~\S\ref{sec:DFlog}, let $\cB$ (resp.\ $\bar\cB$) be the closure of 
$B\times(\A^1\setminus\{0\})$ in $\cX$ (resp. $\bar\cX$).
The log canonical divisors on $\A^1$  and $\P^1$ are defined as 
\begin{equation*}
  K^\lo_{\A^1}:=K_{\A^1}+[0]
  \quad\text{and}\quad
  K^\lo_{\P^1}:=K_{\P^1}+[0]+[\infty],
\end{equation*}
respectively. We now set 
\begin{align*}
  K^{\lo}_{(\cX,\cB)}
  :&=K_\cX+\cB+\cX_{0,\red},\\
  K^\lo_{(\bar\cX,\bar\cB)}
  :&=K_{\bar\cX}+\bar\cB+\bar\cX_{0,\red}+\bar\cX_{\infty,\red}\\
  &=K_{\bar\cX}+\bar\cB+\cX_{0,\red}+\bar\cX_\infty,
\end{align*}
and call these the \emph{log canonical divisors} of $(\cX,\cB)$ and
$(\bar\cX,\bar\cB)$, respectively. Similarly,
\begin{align*}
  K^\lo_{(\cX,\cB)/\A^1}
  :&=K^\lo_{(\cX,\cB)}-\pi^*K^\lo_{\A^1}\\
  &=K_{(\cX,\cB)/\A^1}-(\cX_0-\cX_{0,\red})
\end{align*}
and
\begin{align*}
  K^\lo_{(\bar\cX,\bar\cB)/\P^1}
  :&=K^\lo_{(\bar\cX,\bar\cB)}-\pi^*K^\lo_{\P^1}\\
  &=K_{(\bar\cX,\bar\cB)/\P^1}-(\cX_0-\cX_{0,\red})
\end{align*}
are the \emph{relative log canonical divisors}.
Again we emphasize that these $\Q$-Weil divisor classes
may not be $\Q$-Cartier in general. 

\smallskip
There are two main reasons for introducing the relative log canonical divisors. 
First, they connect well with the log discrepancy function on divisorial valuations on $X$.
Namely, consider normal test configurations $\cX$ and $\cX'$ for $X$, with $\cX'$
dominating $\cX$ via $\mu\colon\cX'\to\cX$. Suppose that $K^\lo_{(\cX,\cB)}$ is $\Q$-Cartier.
Then
\begin{align}\label{equ:Klog}
  K^\lo_{(\bar\cX',\bar\cB')/\P^1}-\mu^*K^\lo_{(\bar\cX,\bar\cB)/\P^1}
  =K^\lo_{(\cX',\cB')/\A^1}-\mu^*K^\lo_{(\cX,\cB)/\A^1}
  =\sum_{E'} A_{(\cX,\cB+\cX_{0,\red})}(\ord_{E'}){E'},
\end{align}
where $E'$ ranges over the irreducible components of $\cX'_0$.
Combining this with Proposition~\ref{prop:discr}, we infer:
\begin{cor}\label{cor:discr} 
  For any normal test configuration $\cX$ dominating $X_{\A^1}$ 
  via $\rho\colon\cX\to X_{\A^1}$, we have
  \begin{equation}\label{equ:Klogbis}
    K^\lo_{(\bar\cX,\bar\cB)/\P^1}-\rho^*K^\lo_{(X_{\P^1},B_{\P^1})/\P^1}
    =K^\lo_{(\cX,\cB)/\A^1}-\rho^*K^\lo_{(X_{\A^1},B_{\A^1})/\A^1}
    =\sum_E b_E A_{(X,B)}(v_E) E,
  \end{equation}
  with $E$ ranging over the irreducible components of $\cX_0$. 
\end{cor}

Second, the relative log canonical divisors behave well under base change.
Namely, let $(\cX_d,\cL_d)$ be the normalized base change of $(\cX,\cL)$, and
denote by $f_d\colon\P^1\to\P^1$ and $g_d\colon\bar\cX_d\to\bar\cX$ the
induced finite morphisms, both of which have degree $d$. The pull-back
formula for log canonical divisors (see~\eg~\cite[\S2.42]{KollarBook}) then yields
\begin{equation}\label{e404}
  K^\lo_{(\bar\cX_d,\bar\cB_d)}=g_d^*K^\lo_{(\bar\cX,\bar\cB)}
  \quad\text{and}\quad
  K^\lo_{\P^1}=f_d^*K^\lo_{\P^1},
\end{equation}
so that $K^\lo_{(\bar\cX_d,\bar\cB_d)/\P^1}=g_d^*K^\lo_{(\bar\cX,\bar\cB)/\P^1}$.
Note that while the (relative) log canonical divisors above may not be $\Q$-Cartier, 
we can pull them back under the finite morphism $g_d$, 
see~\cite[\S2.40]{KollarBook}.
%
%
%
%
\section{Duistermaat-Heckman measures and filtrations}\label{sec:DHfiltr}
In this section, we analyze in detail the limit measure of a
filtration, a concept closely related to Duistermaat-Heckman
measures. This allows us to establish Theorem A and Corollary~B.
%
%
\subsection{The limit measure of a filtration}\label{sec:limit}
Let $X$ be a variety of dimension $n$, $L$ an ample line bundle on
$X$, and set $R=R(X,L)$. 
Let us we review and complement the study in~\cite{BC} of a natural measure
on $\R$  associated to a general $\R$-filtration $F^\bullet R$ on $R$.

Recall that the \emph{volume} of a graded subalgebra $S\subset R$ is defined as 
\begin{equation}\label{equ:vol}
\vol(S):=\limsup_{m\to\infty}\frac{n!}{m^n}\dim S_m\in\R_{\ge0}.
\end{equation}
The following result is proved using Okounkov bodies~\cite{LM,KK} (see also the first author's appendix in~\cite{Sze2}). 
\begin{lem}\label{lem:vol} Let $S\subset R$ be a graded subalgebra containing an ample series, \ie
\begin{itemize}
\item[(i)] $S_m\ne 0$ for all $m\gg 1$; 
\item[(ii)] there exist $\Q$-divisors $A$ and $E$, ample and effective respectively, such that $L=A+E$ and $H^0(X,mA)\subset S_m\subset H^0(X,mL)$ for all $m$ divisible enough. 
\end{itemize}
Then $\vol(S)>0$, and the limsup in (\ref{equ:vol}) is a limit. For
each $m\gg 1$, let $\fa_m\subset\cO_X$ be the base ideal of $S_m$, \ie
the image of the evaluation map $S_m\otimes\cO_X(-mL)\to\cO_X$, and
let $\mu_m\colon X_m\to X$ be the normalized blow-up of $X$ along $\fa_m$, so that $\cO_{X_m}\cdot\fa_m=\cO_{X_m}(-F_m)$ with $F_m$ an effective Cartier divisor. Then we also have
$$
\vol(S)=\lim_{m\to\infty}\left(\mu_m^*L-\tfrac{1}{m}F_m\right)^n.
$$
\end{lem}

Now let $F^\bullet R$ be an $\R$-filtration of the graded ring $R$, as defined in \S\ref{sec:filtr}.  We denote by 
$$
\la^{(m)}_{\max}=\la^{(m)}_{1}\ge\dots\ge\la^{(m)}_{N_m}=\la^{(m)}_{\min}
$$ 
the successive minima of $F^\bullet H^0(X,mL)$. As $R$ is an integral domain, the sequence $(\la^{(m)}_{\max})_{m\in\N}$ is superadditive in the sense
that $\lambda^{(m+m')}_{\max}\ge\la^{(m)}_{\max}+\lambda^{(m')}_{\max}$, and
this implies that 
$$
\la_{\max}=\la_{\max}(F^{\bullet} R):=\lim_{m\to\infty}\frac{\la^{(m)}_{\max}}{m}=\sup_{m\ge 1}\frac{\la^{(m)}_{\max}}{m}\in(-\infty,+\infty].
$$
By definition, we have $\lambda_{\max}<+\infty$ iff there exists $C>0$ such that $F^\la  H^0(X,mL)=0$ for any $\lambda,m$ such that $\la\ge Cm$, and we then say that $F^\bullet R$ has \emph{linear growth}. 

For example, it follows from~\cite[Lemma~3.1]{PS07} that 
the filtration associated to a test configuration (see~\S\ref{sec:filtrtc}) 
has linear growth.
\begin{rmk}\label{R403}
  In contrast, there always exists $C>0$ such that $F^\la H^0(X,mL)=H^0(X,mL)$ 
  for any $\lambda,m$ such that $\la\le-Cm$. This is a simple consequence of the 
  finite generation of $R$, cf.~\cite[Lemma 1.5]{BC}. 
\end{rmk}

For each $\la\in\R$, we define a graded subalgebra of $R$ by setting
\begin{equation}\label{equ:Rla}
R^{(\la)}:=\bigoplus_{m\in\N}F^{m\la} H^0(X,mL). 
\end{equation}
The main result of~\cite{BC} may be summarized as follows. 
\begin{thm}\label{thm:BC} Let $F^\bullet R$ be a filtration with linear growth. 
\begin{itemize}
\item[(i)] For each $\la<\la_{\max}$, $R^{(\la)}$ contains an ample series. 
\item[(ii)] The function $\la\mapsto \vol(R^{(\la)})^{1/n}$ is concave on $(-\infty,\la_{\max})$, and vanishes on $(\la_{\max},+\infty)$. 
\item[(iii)] If we introduce, for each $m$, the probability measure
\begin{equation}\label{equ:num}
\nu_m:=\frac{1}{N_m}\sum_j\delta_{m^{-1}\la^{(m)}_{j}}=-\frac{d}{d\la}\frac{\dim F^{m\la} H^0(X,mL)}{N_m}
\end{equation}
 on $\R$, then $\nu_m$ has uniformly bounded support and 
converges weakly as $m\to\infty$ to the probability measure 
\begin{equation}\label{equ:limmeas}
\nu:=-\frac{d}{d\la}V^{-1}\vol(R^{(\la)}).
\end{equation}
\end{itemize}
\end{thm}

We call $\nu$ the \emph{limit measure} of the filtration $F^\bullet R$. The log concavity property of $\vol(R^{(\la)})$ immediately yields:
\begin{cor}\label{cor:supp} The support of the limit measure $\nu$ is given by
$\supp\nu=[\la_{\min},\la_{\max}]$ with 
$$
\la_{\min}:=\inf\left\{\la\in\R\mid\vol(R^{(\la)})<V\right\}.
$$
Further, $\nu$ is absolutely continuous with respect to the Lebesgue measure, except perhaps for a point mass at $\la_{\max}$. 
\end{cor}
More precisely, the mass of $\nu$ on $\{\la_{\max}\}$ is equal to $\lim_{\la\to(\la_{\max})_-}\vol(R^{(\la)})$. 
\begin{rmk} 
  While we trivially have
  $\la_{\min}\ge\limsup_{m\to\infty}m^{-1}\la^{(m)}_{\min}$, the
  inequality can be strict in general. It will, however, be an equality for the filtrations 
  considered in~\S\ref{sec:PP} and~\S\ref{sec:filtval}.
\end{rmk}
\begin{rmk}\label{R101}
  Let $F^\bullet R(X,L)$ be a filtration of linear growth and limit
  measure $\nu$. 
  For any $r\in\Z_{>0}$, we obtain a filtration $F^\bullet R(X,rL)$ by
  restriction. This filtration also has linear growth and its limit
  measure is given by $r_*\nu$.
\end{rmk}
%
%
\subsection{Limit measures and Duistermaat-Heckman measures}\label{S102}
Now suppose $L$ is an ample $\Q$-line bundle on $X$.
To simplify the terminology, we introduce
\begin{defi}
  We define the \emph{Duistermaat-Heckman} measure of 
  any semiample test configuration of $(X,L)$ as
  $\DH_{(\cX,\cL)}:=\DH_{(\cX_\amp,\cL_\amp)}$, where 
  $(\cX_\amp,\cL_\amp)$ is the ample model of $(\cX,\cL)$
  as in Proposition~\ref{prop:amplemodel}. 
\end{defi}
With this definition, Duistermaat-Heckman measures are invariant under normalization:
\begin{cor}\label{C301}
  If $(\cX,\cL)$ is a semiample test configuration for $(X,L)$, and
  $(\tcX,\tcL)$ is its normalization, then $\DH_{(\cX,\cL)}=\DH_{(\tcX,\tcL)}$.
\end{cor}
\begin{proof}
  Let $(\cX_{\amp},\cL_{\amp})$ be the ample model of $(\cX,\cL)$ and 
  $(\tcX_{\amp},\tcL_{\amp})$ its normalization. The composition 
  $\tcX\to\cX_\amp$ lifts to a map $\tcX\to\tcX_\amp$, under which
  $\tcL$ is the pullback of $\cL_\amp$.
  By the uniqueness statement of Proposition~\ref{prop:amplemodel},
  $(\tcX_\amp,\tcL_\amp)$ is the ample model of $(\tcX,\tcL)$.
  By definition, we thus have
  $\DH_{(\cX,\cL)}=\DH_{(\cX_{\amp},\cL_{\amp})}$
  and $\DH_{(\tcX,\tcL)}=\DH_{(\tcX_{\amp},\tcL_{\amp})}$,
  whereas Theorem~\ref{thm:DHinv} yields
  $\DH_{(\tcX_{\amp},\tcL_{\amp})}=\DH_{(\cX_{\amp},\cL_{\amp})}$,
  concluding the proof.
\end{proof}
We now relate Duistermaat-Heckman measures and limit measures.
Recall from~\S\ref{sec:filtrtc} that any test configuration for $(X,L)$
induces a filtration of $R(X,rL)$ for 
$r$ sufficiently divisible.
\begin{prop}\label{prop:DHlimit} 
  If $(\cX,\cL)$ is semiample, then, 
  for $r$ sufficiently divisible, the limit measure of the
  filtration on $R(X,rL)$ induced by $(\cX,\cL)$ is equal to $r_*\DH(\cX,\cL)$.
\end{prop}
\begin{proof}[Proof of Proposition~\ref{prop:DHlimit}]
  Using homogeneity (see Definition~\ref{D101} and Remark~\ref{R101}), 
  we may assume $r=1$.
  Further, $(\cX,\cL)$ and $(\cX_\amp,\cL_\amp)$ induce the same filtration
  on $(X,L)$, so we may assume $(\cX,\cL)$ is ample.

  By the projection formula and the ampleness of $\cL$, it then follows
  that the fiber at $0$ of the vector bundle $\pi_*\cO_\cX(m\cL)$ can be identified with 
  $H^0((\cX)_0,m(\cL)_0)$, and Lemma~\ref{L201} therefore shows that
  the weight measure of the latter $\G_m$-module is given by
  \begin{equation*}
    \mu_{H^0((\cX_\amp)_0,m(\cL_\amp)_0)}
    =\frac{1}{N_m}\sum_{\la\in\Z}\dim\left(F^\la H^0(X,mL)/F^{\la+1} H^0(X,mL)\right)\d_\la.
  \end{equation*}
  As a result, $\mu_m:=(1/m)_*\mu_{H^0((\cX)_0,m(\cL)_0)}$ satisfies
  \begin{equation*}
    \mu_m=-\frac{d}{d\la}\frac{\dim F^{m\la} H^0(X,mL)}{N_m},
  \end{equation*}
  and hence converges to the limit measure measure of $F^\bullet R$ by Theorem~\ref{thm:BC}.
\end{proof}
%
%
%
%
%
\subsection{Piecewise polynomiality in the normal case}\label{sec:PP}
\begin{thm}\label{thm:PP} 
  Let $X$ be an $n$-dimensional normal variety, $L$ an ample
  line bundle on $X$, and $F^\bullet R$ a finitely generated
  $\Z$-filtration of $R=R(X,L)$. 
  Then  $F^\bullet R$ has linear growth, and the density of its limit measure $\nu$  is a piecewise polynomial function, of degree at most $n-1$.
\end{thm} 
By the density of $\nu$ we mean the density of the absolutely
continuous part, see Corollary~\ref{cor:supp}.

Since a semiample test configuration of a polarized variety $(X,L)$
induces a finitely generated filtration of $R(X,rL)$ for $r$
sufficiently divisible, we get:
\begin{cor}\label{cor:PP} 
  Let $(X,L)$ be a polarized normal variety. Then the Duistermaat measure $\DH_{(\cX,\cL)}$ of any semiample test configuration $(\cX,\cL)$ for $(X,L)$ is the sum of a point mass and an absolutely continuous measure with piecewise polynomial density. 
\end{cor}
The general case where $(X,L)$ is an arbitrary polarized scheme will
be treated in Theorem~\ref{thm:PPscheme}, by reducing to the normal
case studied here. 

\begin{proof}[Proof of Theorem~\ref{thm:PP}] The following argument is inspired by the proof of~\cite[Proposition
  4.13]{ELMNP}.\footnote{While the base field in~\textit{loc.cit.}\ is
    $\C$, the results we use are valid over an arbitrary algebraically
    closed field.} 
  For $\tau=(m,\la)\in\N\times\Z$, let 
$\fa_\tau$ be the base ideal of $F^\la  H^0(X,mL)$, \ie the image of
the evaluation map $F^\la  H^0(X,mL)\otimes\cO_X(-mL)\to\cO_X$. Let
$\mu_\tau\colon X_\tau\to X$ be the normalized blow-up of $\fa_\tau$, which is also the normalized blow-up of its integral closure $\overline{\fa_\tau}$. Then
$$
\cO_{X_\tau}\cdot\fa_\tau=\cO_{X_\tau}\cdot\overline{\fa_\tau}=\cO_{X_\tau}(-F_\tau),
$$ 
with $F_\tau$ a Cartier divisor, and we set 
$$
V_\tau:=\left(\mu_\tau^*L-\tfrac{1}{m}F_\tau\right)^n. 
$$
Since $R^{(\la)}$ contains an ample series for $\la\in(-\infty,\la_{\max})$, Lemma~\ref{lem:vol} yields
$$
\vol(R^{(\la)})=\lim_{m\to\infty}V_{(m,\lceil m\la\rceil)}. 
$$
Now, we use the finite generation of $F^\bullet R$, which implies that the $\N\times\Z$-graded $\cO_X$-algebra 
$\bigoplus_{\tau\in\N\times\Z}\fa_\tau$ is finitely generated. 
By~\cite[Proposition 4.7]{ELMNP}, we may thus find a positive integer $d$ and finitely many vectors $e_i=(m_i,\la_i)\in\N\times\Z$, $1\le i\le r$, with the following properties:
\begin{itemize}
\item[(i)] $e_1=(0,-1)$, $e_r=(0,1)$, and the slopes $a_i:=\la_i/m_i$ are strictly increasing with $i$; 
\item[(ii)] Every $\tau\in\N\times\Z$ may be written as $\tau=p_i e_i+p_{i+1}e_{i+1}$ with $i$, $p_i,p_{i+1}\in\N$ uniquely determined, and the integral closures of $\fa_{d\tau}$ and $\fa_{de_i}^{p_i}\cdot\fa_{de_{i+1}}^{p_{i+1}}$ coincide. 
\end{itemize}
Choose a projective birational morphism $\mu\colon X'\to X$ with $X'$ normal and dominating the blow-up of each $\fa_{d e_i}$, so that there is a Cartier divisor $E_i$ with $\cO_{X'}\cdot\fa_{de_i}=\cO_{X'}(-E_i)$. For all $\tau=(m,\la)\in\N\times\Z$ written as in (ii) as $\tau=p_ie_i+p_{i+1}e_{i+1}$, we get 
$$
\cO_{X'}\cdot\fa_{d e_i}^{p_i}\cdot\fa_{d e_{i+1}}^{p_{i+1}}=\cO_{X'}(-(p_i E_i+p_{i+1}E_{i+1})),
$$
and the universal property of normalized blow-ups therefore shows that $\mu$ factors through the normalized blow-up of $\fa_{d e_i}^{p_i}\cdot\fa_{d e_{i+1}}^{p_{i+1}}$. By Lemma~\ref{lem:normblow}, the latter is also the normalized blow-up of 
$$
\overline{\fa_{d e_i}^{p_i}\cdot\fa_{d e_{i+1}}^{p_{i+1}}}=\overline{\fa_{d\tau}},
$$
so we infer that 
$$
\fa_{d\tau}\cdot\cO_{X'}=\cO_{X'}\left(-(p_i E_i+p_{i+1}E_{i+1})\right), 
$$
with $p_i E_i+p_{i+1}E_{i+1}$ the pull-back of $F_{d\tau}$. As a result, we get
$$
V_{d\tau}=\left(\mu^*L-\tfrac{1}{dm}\left(p_i E_i+p_{i+1}E_{i+1}\right)\right)^n. 
$$
Pick $\la\in(0,\la_{\max})$, so that $\la\in[a_i,a_{i+1})$ for some $i$. We infer from the previous discussion that
$$
\vol(R^{(\la)})=\lim_{m\to\infty} V_{(m,\lceil m\la\rceil)}=\left(\mu^*L-(f_i(\la) E_i+f_{i+1}(\la)E_{i+1})\right)^n
$$
for some affine functions $f_i,f_{i+1}$, and we conclude that $\vol(R^{(\la)})$ is a piecewise polynomial function of $\la\in(-\infty,\la_{\max})$, of degree at most $n$. The result follows by (\ref{equ:limmeas}).
\end{proof}

\begin{rmk}\label{rmk:finitetype} For a finitely generated $\Z$-filtration $F^\bullet R$, the graded subalgebra $R^{(\la)}$ is finitely generated for each $\la\in\Q$~\cite[Lemma 4.8]{ELMNP}. In particular, $\vol(R^{(\la)})\in\Q$ for all $\la\in\Q\cap(-\infty,\la_{\max})$. 
\end{rmk} 
%
%
\subsection{The filtration defined by a divisorial valuation}\label{sec:filtval}
Let $X$ be a normal projective variety and $L$ an ample line bundle on $X$.

Any valuation $v$ on $X$ defines a filtration $F_v^\bullet R$ by setting
$$
F_v^\la H^0(X,mL):=\left\{s\in H^0(X,mL)\mid v(s)\ge\la\right\}.
$$
As a special case of~\cite[Proposition 2.12]{BKMS}, $F_v^\bullet R$ has linear growth for any divisorial valuation $v$. The following result will be needed later on.

\begin{lem}\label{lem:suppdiv} Let $v$ be a divisorial valuation on $X$, and let $\nu$ be the limit measure of the corresponding filtration $F_v^\bullet R$. Then $\supp\nu=[0,\la_{\max}]$. In other words, we have
$$
\lim_{m\to\infty}\frac{\dim\left\{s\in H^0(X,mL)\mid v(s)\ge\la m\right\}}{N_m}<1
$$
for any $\la>0$. 
\end{lem}
The equivalence between the two statements follows from Corollary~\ref{cor:supp} above. 

\begin{proof} Let $Z\subset X$ be the closure of the center $c_X(v)$ of $v$ on $X$, and $w$ a Rees valuation of $Z$. Since the center of $w$ on $X$ belongs to $Z=\overline{c_X(v)}$, the general version of Izumi's theorem in~\cite{HS01} yields a constant $C>0$ such that $v(f)\le C w(f)$ for all $f\in\cO_{X,c_X(w)}$. 

Let $\mu\colon X'\to X$ be the normalized blow-up of $Z$ and set $E:=\mu^{-1}(Z)$. By definition, the Rees valuations of $Z$ are given up to scaling by vanishing order along the irreducible components of $E$. Given $\la>0$, we infer that 
$$
\left\{v\ge\la m\right\}\subset\mu_*\cO_{X'}(-m\d E)
$$
for all $0<\d\ll 1$ and all $m\ge 1$. It follows that
$$
\left\{s\in H^0(X,mL)\mid v(s)\ge\e\la m\right\}\hookrightarrow H^0\left(X',m\left(\mu^*L-\d E\right)\right), 
$$
so that $\vol(R^{(\la)})\le(\mu^*L-\d E)^n$. But since $-E$ is $\mu$-ample, $\mu^*L-\d E$ is ample on $X'$ for $0<\d\ll 1$, so that
$$
\frac{d}{d\la}(\mu^*L-\d E)^n=-n\left(E\cdot(\mu^*L-\d E)^{n-1}\right)<0. 
$$
It follows that 
$$
\vol(R^{(\la)})\le (\mu^*L-\d E)^n<(\mu^*L)^n=V
$$ 
for $0<\d\ll 1$, hence the result. 
\end{proof}

\begin{rmk}\label{rmk:div} At least in characteristic zero, the continuity of the volume function shows that $\vol(R^{(\la)})\to 0$ as $\la\to\la_{\max}$ from below, so that $\nu$ has no atom at $\la_{\max}$, and is thus absolutely continuous on $\R$ (cf.~\cite[Proposition 2.25]{BKMS}). 

On the other hand, $F^\bullet R$ is not finitely generated in general. Indeed, well-known examples of irrational volume show that $\vol(R^{(1)})$ can sometimes be irrational (compare Remark~\ref{rmk:finitetype}). 
\end{rmk}
\begin{rmk}
  The filtration defined by a valuation and its relation to $K$-stability has been recently studied 
  by Fujita~\cite{Fuj15a,Fuj15b,Fuj16}, Li~\cite{Li15,Li16} and Liu~\cite{Liu16}.
\end{rmk}
%
%
\subsection{The support of a Duistermaat-Heckman measure}\label{sec:support}
The following precise description of the support of a Duistermaat-Heckman measure is the key to the characterization of almost trivial ample test configurations to be given below. 

\begin{thm}\label{thm:supp} Let $(\cX,\cL)$ be a normal, semiample test configuration dominating $X_{\A^1}$, and write $\cL=\rho^*L_{\A^1}+D$ with $\rho\colon\cX\to X_{\A^1}$ the canonical morphism. Then the support $[\la_{\min},\la_{\max}]$ of its Duistermaat-Heckman measure satisfies
\begin{equation*}
  \la_{\min}=\min_Eb_E^{-1}\ord_E(D)
\quad\text{and}\quad
\la_{\max}=\max_Eb_E^{-1}\ord_E(D)=\ord_{E_0}(D),  
\end{equation*}
where $E$ runs over the irreducible components of $\cX_0$, $b_E:=\ord_E(\cX_0)=\ord_E(t)$, and $E_0$ is the strict transform of $X\times\{0\}$ (which has $b_{E_0}=1$). 
\end{thm}

\begin{lem}\label{lem:filtr} In the notation of Theorem~\ref{thm:supp}, the induced filtration of $R$ satisfies, for all $m$ divisible enough and all $\la\in\Z$,
$$
F^\la  H^0(X,mL)=\bigcap_E\left\{s\in H^0(X,mL)\mid v_E(s)+m\,b_E^{-1}\ord_E(D)\ge\la\right\},
$$
where $E$ runs over the irreducible components of $\cX_0$. 
\end{lem}
According to Lemma~\ref{lem:div2}, $v_E$ is a divisorial valuation on $X$ for $E\ne E_0$, while $v_{E_0}$ is the trivial valuation (so that $v_{E_0}(s)$ is either $0$ for $s\ne 0$, or $+\infty$ for $s=0$). 

\begin{proof} Pick any $m$ such that $m\cL$ is a line bundle. By (\ref{equ:filtr}), a section $s\in H^0(X,mL)$ is in $F_{(\cX,\cL)}^\la H^0(X,mL)$ iff $\bar s t^{-\la}\in H^0(\cX,m\cL)$, with $\bar s$ the $\G_m$-invariant rational section of $m\cL$ induced by $s$. By normality of $\cX$, this amounts in turn to 
$\ord_E\left(\bar s t^{-\la}\right)\ge 0$ for all $E$, \ie $\ord_E(\bar s)\ge\la b_E$ for all $E$. The result follows since $m\cL=\rho^*(mL_{\A^1})+m D$ implies that 
$$
\ord_E(\bar s)=r(\ord_E)(s)+m \ord_E(D)=b_Ev_E(s)+m\ord_E(D). 
$$
 \end{proof}

\begin{lem}\label{lem:minmax} In the notation of Theorem~\ref{thm:supp}, the filtration $F^\bullet H^0(X,mL)$ satisfies
\begin{equation*}
  \frac{\la^{(m)}_{\min}}{m}=\min_Eb_E^{-1}\ord_E(D)
  \quad\text{and}\quad
  \frac{\la^{(m)}_{\max}}{m}=\ord_{E_0}(D)=\max_Eb_E^{-1}\ord_E(D)
\end{equation*}
for all $m$ divisible enough. 
\end{lem}
\begin{proof} 
  Set $c:=\min_Eb_E^{-1}\ord_E(D)$, and pick $m$ divisible enough (so that $mc$ is in particular an integer). The condition $v_E(s)+m\ord_E(D)\ge m c b_E$ automatically holds for all $s\in H^0(X,mL)$, since $v_E(s)\ge 0$. By Lemma~\ref{lem:filtr}, we thus have $F^{mc} H^0(X,mL)=H^0(X,mL)$, and hence $m c\le\la^{(m)}_{\min}$. 

We may assume $mL$ is globally generated, so for every $E$ we may find a section 
$$
s\in H^0(X,mL)=F^{\la^{(m)}_{\min}} H^0(X,mL)
$$ 
that does not vanish at the center of $v_E$ on $X$, \ie $v_E(s)=0$. By Lemma~\ref{lem:filtr}, it follows that $m\ord_E(D)\ge\la^{(m)}_{\min}b_E$. Since this holds for every $E$, we conclude that $mc\ge\la^{(m)}_{\min}$. 

We next use that $m\cL=\rho^*(mL_{\A^1})+mD$ is globally generated. This implies in particular that $\cO_\cX(mD)$ is $\rho$-globally generated, which reads
$$
\cO_\cX(mD)=\rho_*\cO_\cX(mD)\cdot\cO_{\cX}
$$
as fractional ideals. But we trivially have $\rho_*\cO_\cX(mD)\subset\cO_{X_{\A^1}}(m\rho_*D)$, and we infer
$$
D\le\rho^*\rho_*D. 
$$
Now $\rho_*D=\ord_{E_0}(D) X\times\{0\}$, hence $\rho^*\rho_*D=\ord_{E_0}\cX_0'$, which yields $\ord_E(D)\le\ord_{E_0}(D)b_E$, hence $\ord_{E_0}(D)=\max_Eb_E^{-1}\ord_E(D)$.

Since $\rho_*\cO_\cX(mD)$ is the flag ideal $\fa^{(m)}$ of Definition~\ref{defi:flag}, 
we also see that
\begin{multline*}
  m\max_Eb_E^{-1}\ord_E(D)=\min\left\{\la\in\Z\mid mD\le\la\cX_0\right\}\\
  =\min\left\{\la\in\Z\mid t^{-\la}\in\fa^{(m)}\right\}
  =\max\left\{\la\in\Z\mid\fa^{(m)}_\la\ne 0\right\},
\end{multline*}
and we conclude thanks to Proposition~\ref{prop:filtrflag}. 
\end{proof}

\begin{proof}[Proof of Theorem~\ref{thm:supp}] In view of Corollary~\ref{cor:supp}, the description of the supremum of the support of $\nu=\DH_{(\cX,\cL)}$ follows directly from Lemma~\ref{lem:minmax}. 

We now turn to the infimum. The subtle point of the argument is that it is not a priori obvious that the stationary value 
$$
\frac{\la^{(m)}_{\min}}{m}=\min_Eb_E^{-1}\ord_E(D)
$$
given by Lemma~\ref{lem:minmax}, which is of course the infimum of the support of $\nu_m$ as in (\ref{equ:num}), should also be the infimum of the support of their weak limit $\nu=\lim_m\nu_m$. What is trivially true is the inequality
$$
\min_Eb_E^{-1}\ord_E(D)=\inf\supp\nu_m\le\inf\supp\nu. 
$$
Now pick $\la>\min_Eb_E^{-1}\ord_E(D)$. According to Corollary~\ref{cor:supp}, it remains to show that 
\begin{equation}\label{equ:volnegl}
\lim_{m\to\infty}\frac{\dim F^{m\la}H^0(X,mL)}{N_m}<1. 
\end{equation}
Note that $\e:=\la b_E-\ord_E(D)>0$ for at least one irreducible component $E$. By Lemma~\ref{lem:filtr}, it follows that
\begin{equation}\label{equ:finc}
F^{m\la}H^0(X,mL)\subset\left\{s\in H^0(X,mL)\mid v_E(s)\ge m\e\right\}. 
\end{equation}
By Lemma~\ref{lem:div2}, $v_E$ is either the trivial valuation or a divisorial valuation. In the former case, the right-hand side of (\ref{equ:finc}) consists of the zero section only, while in the latter case we get (\ref{equ:volnegl}) thanks to Lemma~\ref{lem:suppdiv}. 
\end{proof}
%
%
\subsection{Proof of Theorem A}\label{sec:thmA}
Now let $(X,L)$ be an arbitrary polarized scheme, and $(\cX,\cL)$ an
ample test configuration for $(X,L)$. Theorem~A and Corollary~B in the
introduction are consequences of the following two results. 
\begin{thm}\label{thm:PPscheme} The density of the absolutely
  continuous part of the Duister\-maat-Heckman measure
  $\DH_{(\cX,\cL)}$ is a piecewise polynomial function, while the
  singular part is a finite sum of point masses. 
\end{thm}

\begin{thm}\label{thm:normzero} The measure $\DH_{(\cX,\cL)}$ is a
  finite sum of point masses iff $(\cX,\cL)$ is almost trivial. In
  this case, $\DH_{(\cX,\cL)}$ is a Dirac mass when $X$ is irreducible. 
\end{thm}
Recall  that almost trivial but nontrivial test configurations always exist when $X$ is, say, a normal variety (cf.~\cite[\S3.1]{LX} and Proposition~\ref{prop:triv1PS}). 

\begin{proof}[Proof of Theorem~\ref{thm:PPscheme}] Since $\DH_{(\cX,\cL)}$ is defined as the Duistermaat-Heckman measure of some polarized $\G_m$-scheme by Definition~\ref{defi:DHsemi}, it is enough to show the result for the Duister\-maat-Heckman measure $\DH_{(X,L)}$ of a polarized $\G_m$-scheme $(X,L)$ (\ie the special case of Theorem~\ref{thm:PPscheme} where $(\cX,\cL)$ is a product test configuration). By (ii) in Proposition~\ref{prop:DHaction}, we may further assume that $X$ is a variety, \ie reduced and irreducible. By the invariance property of Theorem~\ref{thm:DHinv}, we are even reduced to the case where $X$ is a normal variety, which is treated in Corollary~\ref{cor:PP}. 
\end{proof}

\begin{proof}[Proof of Theorem~\ref{thm:normzero}] 
Using again (ii) in Proposition~\ref{prop:DHaction} and Theorem~\ref{thm:DHinv}, we may assume that $\cX$ (and hence $X$) is a normal variety. By Lemma~\ref{lem:triv}, our goal is then to show that the support of $\DH_{(\cX,\cL)}$ is reduced to a point iff $(\cX,\cL+c\cX_0)=(X_{\A^1},L_{\A^1})$ for some $c\in\Q$. 

In order to deduce this from Theorem~\ref{thm:supp}, let $(\cX',\cL')$ be a pull-back of $(\cX,\cL)$ with $\cX'$ normal and dominating $X_{\A^1}$. Since $(\cX,\cL)$ is the ample model of $(\cX',\cL')$, we have $\DH_{(\cX,\cL)}=\DH_{(\cX',\cL')}$. 
In the notation of Theorem~\ref{thm:supp}, this measure is a Dirac mass iff $b_E^{-1}\ord_E(D)=c$ is independent of $E$, \ie $D=c\cX'_0$. But this means that $(X_{\A^1},L_{\A^1})$ is the ample model of $(\cX',\cL'-c\cX'_0)$, \ie $(\cX,\cL-c\cX_0)=(X_{\A^1},L_{\A^1})$ by uniqueness of the ample model. 
\end{proof}
%
%
%
%
\section{Non-Archimedean metrics}\label{sec:NAmetrics}
From now on, $X$ will always denote a normal projective variety,
unless otherwise specified, and $L$ will be a $\Q$-line bundle on $X$.
%
%
\subsection{Test configurations as non-Archimedean metrics}
Motivated by Berkovich space considerations (see~\S\ref{S202} below)
we introduce the following notion.
\begin{defi}\label{defi:equiv} 
Two test configurations $(\cX_1,\cL_1)$, $(\cX_2,\cL_2)$ for $(X,L)$
are \emph{equivalent} if there exists a test configuration
$(\cX_3,\cL_3)$ that is a pull-back of both $(\cX_1,\cL_1)$ and
$(\cX_2,\cL_2)$. 
An equivalence class is called a \emph{non-Archimedean metric} on $L$, and is denoted by $\phi$. We denote by $\phi_\triv$ the equivalence class of $(X_{\A^1},L_{\A^1})$. 
\end{defi}
A non-Archimedean metric $\phi$ on the trivial line bundle $\cO_X$ can be viewed as a function $\phi:X^{\mathrm{div}}\to\Q$ on the set of divisorial valuations on $X$. Indeed, by Theorem~\ref{thm:restrdiv}, every divisorial valuation on $X$ is of
the form $v_E=b_E^{-1}r(\ord_E)$,
where $E$ is an irreducible component of $\cX_0$ for some normal
test configuration $\cX$ of $X$. We may assume $\phi$ is
represented by $\cO_X(D)$ for some $\Q$-Cartier divisor $D$ supported on
$\cX_0$, and then $\phi(v_E)=b_E^{-1}\ord_E(D)$. 
When $D=\cX_0$, we get the constant function $\phi\equiv1$.
In general, there exists $C>0$ such that $D\pm C\cX_0$ is effective;
hence $|\phi|\le C$, so $\phi$ is a  bounded function.
To the trivial metric on $\cO_X$ corresponds the zero function. 
%
%
\subsection{Operations on non-Archimedean metrics}\label{S101}
If $\phi_i$ is a non-Archimedean metric on a $\Q$-line bundle
$L_i$ on $X$ and $r_i\in\Q$ for $i=1,2$, then we get a naturally defined non-Archimedean metric 
$r_1\phi_1+r_2\phi_2$ on $L:=r_1L_1+r_2L_2$ as follows:
if $\phi_i$ is represented by a test configuration  
$(\cX,\cL_i)$ with the same $\cX$, then $\phi$
is represented by $(\cX,r_1\cL_1+r_2\cL_2)$. 

In particular, if $\phi,\phi'$ are non-Archimedean metrics on the same line bundle $L$, then $\phi-\phi'$ is a non-Archimedean metric on $\cO_X$, which will thus be viewed as a function on $X^{\mathrm{div}}$. If we choose a normal representative $(\cX,\cL)$ of $\phi$ that dominates $X_{\A^1}$ and write as before $\cL=\rho^*L_{\A^1}+D$ with $D$ a $\Q$-Cartier divisor supported on $\cX_0$, then 
\begin{equation}\label{equ:divfunc}
(\phi-\phi_\triv)(v_E)=b_E^{-1}\ord_E(D)
\end{equation}
for each component $E$ of $\cX_0$. 

If $f\colon Y\to X$ is a surjective
morphism, with $Y$ normal and projective, then any non-Archimedean metric
$\phi$ on $L$ induces a non-Archimedean metric $f^*\phi$ on $f^*L$.
Indeed, suppose $\phi$ is represented by a test configuration
$(\cX,\cL)$. We can find a test configuration $\cY$ of $Y$ and
a projective $\G_m$-equivariant morphism $\cY\to\cX$ compatible with
$f$ via the identifications $\cX_1\simeq X$, $\cY_1\simeq Y$.
Define $f^*\phi$ as the metric represented by the pullback
of $\cL$ to $\cY$. The pullback of the trivial metric on $L$
is the trivial metric on $f^*L$.
%
%
\subsection{Translation and scaling}
The operations above give rise to two natural actions on the space of non-Archimedean 
metrics on a fixed line bundle.

First, if $\phi$ is a non-Archimedean metric on a line bundle $L$, 
then so is $\phi+c$, for any $c\in\Q$. Thus we obtain a
\emph{translation action} by $(\Q,+)$ on the set of non-Archimedean metrics.

Second, we have a \emph{scaling action} by the semigroup $\N^*$ which 
to a non-Archimedean metric $\phi$ on $L$ associates a new 
non-Archimedean metric $\phi_d$ on $L$ for every $d\in\N^*$.
If $\phi$ is represented by a test configuration $(\cX,\cL)$, then
$\phi_d$ is represented by the base change of $(\cX,\cL)$ under
the base change $t\to t^d$.
This scaling action is quite useful and a particular
feature of working over a trivially valued ground field.
Note that $\phi_\triv$ is fixed by the scaling action. 

Viewing, as above, a metric $\phi$ on the trivial line bundle $\cO_X$ as a function on 
divisorial valuations, we have 
\begin{equation}\label{e405}
  \phi_d(dv)=d\phi(v)
\end{equation}
for any divisorial valuation $v$ on $X$. 
%
%
\subsection{Positivity}
Next we introduce positivity notions for metrics.
\begin{defi} 
  Assume $L$ is ample. 
  Then a non-Archimedean metric $\phi$ on $L$ is called 
  \emph{positive} if some representative $(\cX,\cL)$ of $\phi$ is semiample. 

  We denote by $\cH^{\NA}(L)$ the set of all non-Archimedean
  positive metrics on $L$, \ie the quotient of the set of semiample test configurations by the above equivalence relation. 
\end{defi}
We sometimes write $\cH^{\NA}$ when no confusion is possible. 
The notation mimics $\cH=\cH(L)$ for the space of smooth, positively curved
Hermitian metrics on $L$ when working over~$\C$.
\begin{lem}\label{lem:amp} 
  When $L$ is ample, every metric $\phi\in\cH^{\NA}(L)$ 
  is represented by a unique normal, ample test
  configuration $(\cX,\cL)$. 
  Every normal representative of $\phi$ is a pull-back of $(\cX,\cL)$. 
\end{lem}
\begin{proof} We first prove uniqueness. Let $(\cX_i,\cL_i)$, $i=1,2$,
  be equivalent normal ample test configurations, so that there exists $(\cX_3,\cL_3)$ as in Definition~\ref{defi:equiv}. For $i=1,2$, the birational morphism $\mu_i\colon\cX_3\to\cX_i$ satisfies $(\mu_i)_*\cO_{\cX_3}=\cO_{\cX_i}$, by normality of $\cX_i$. It follows that $(\cX_1,\cL_1)$ and $(\cX_2,\cL_2)$ are both ample models of $(\cX_3,\cL_3)$, and hence $(\cX_1,\cL_1)=(\cX_2,\cL_2)$ by the uniqueness part of Proposition~\ref{prop:amplemodel}. 

Now pick a normal representative $(\cX,\cL)$ of $\phi$. By Proposition~\ref{prop:amplemodel}, its ample model $(\cX_\amp,\cL_\amp)$ is a normal, ample representative, and $(\cX,\cL)$ is a pull-back of $(\cX_\amp,\cL_\amp)$. This proves the existence part, as well as the final assertion. 
\end{proof}
It is sometimes convenient to work with a weaker positivity notion.
\begin{defi}
  Assume $L$ is nef. Then a non-Archimedean metric $\phi$ on $L$ is
  \emph{semipositive} if some (or, equivalently, any) 
  representative $(\cX,\cL)$ of $\phi$ is relatively nef with respect
  to $\cX\to\A^1$.
\end{defi}
In this case, $\bar\cL$ is relatively nef for $\bar\cX\to\P^1$,
where $(\bar\cX,\bar\cL)$ is the compactification of $(\cX,\cL)$.

\smallskip
When $L$ is nef (resp.\ ample), the translation and scaling
actions preserve the subset of semipositive (resp.\ positive)
non-Archimedean metrics on $L$.
Positivity of metrics is also preserved under pull-back, as follows. If $f:Y\to X$ is  surjective morphism with $Y$ normal, projective, then $f^*L$ is nef for any nef line bundle $L$ on $X$, and $f^*\phi$ is semipositive for any semipositive metric $\phi$ on $L$. If $L$ and $f^*L$ are further ample (which implies that $f$ is finite), then $f^*\phi$ is positive for any positive metric $\phi$ on $L$.\footnote{This seemingly inconsistent property is explained by the fact that the (analytification of the) ramification locus of $f$ does not meet the Berkovich skeleton where $\phi$ is determined.}
%
%
\subsection{Duistermaat-Heckman measures and $L^p$-norms}\label{sec:DHmetric}
In this section, $L$ is ample.
\begin{defi}\label{defi:DHmetric} Let $\phi\in\cH^\NA(L)$ be a
  positive non-Archimedean metric on $L$.
\begin{itemize}
\item[(i)] The \emph{Duistermaat-Heckman measure} of $\phi$ is defined by setting $\DH_\phi:=\DH_{(\cX,\cL)}$ for any semiample representative of $\phi$.
\item[(ii)]  The \emph{$L^p$-norm} of $\phi$ is defined as the $L^p(\nu)$-norm of $\la-\bar\la$, with $\bar\la:=\int_\R\la\,d\nu$ the barycenter of $\nu=\DH_\phi$. 
\end{itemize}
\end{defi}

This is indeed well-defined, thanks to the following result. 
\begin{lem}\label{lem:DHmetric} For any two equivalent semiample test configurations $(\cX_1,\cL_1)$, $(\cX_2,\cL_2)$, we have $\DH_{(\cX_1,\cL_1)}=\DH_{(\cX_2,\cL_2)}$. 
\end{lem}
\begin{proof} 
  By Corollary~\ref{C301} we may also assume that 
  $\cX_i$ is normal for $i=1,2$. 
  Since any two normal test configurations is dominated by a third, we 
  may also assume that $\cX_1$ dominates $\cX_2$. 
  In this case, $(\cX_1,\cL_1)$ and $(\cX_2,\cL_2)$ have the same
  ample model, and hence the same Duistermaat-Heckman measure.
\end{proof}
In view of (\ref{equ:divfunc}), Theorem~\ref{thm:supp} can be reformulated as follows.

\begin{thm}\label{thm:suppNA} If $\phi$ is a positive metric on $L$, then 
$$
\sup_{X^{\mathrm{div}}}(\phi-\phi_\triv)=(\phi-\phi_\triv)(v_\triv)=\sup\supp\DH_\phi.
$$
\end{thm}

The key property of $L^p$-norms is to characterize triviality, as follows. 
\begin{thm}\label{thm:trivmetric} 
  Let $\phi\in\cH^{\NA}(L)$ be a positive non-Archimedean metric on $L$.
  Then the following conditions are equivalent:
\begin{itemize}
\item[(i)] the Duistermaat-Heckman measure $\DH_\phi$ is a Dirac mass; 
\item[(ii)] for some (or, equivalently, any) $p\in[1,\infty]$, $\|\phi\|_p=0$;
\item[(iii)] $\phi=\phi_\triv+c$ for some $c\in\Q$.  
 \end{itemize}
\end{thm}

\begin{lem}\label{lem:DHtrans} Let $\phi\in\cH^\NA(L)$ be a positive non-Archimedean metric on $L$. For $c\in\Q$ and $d\in\N^*$, we have
\begin{itemize}
\item[(i)] $\DH_{\phi+c}$ and $\DH_{\phi_d}$ are the pushforwards of $\DH_\phi$ by
    $\la\mapsto\la+c$ and $\la\mapsto d\la$, respectively.
\item[(ii)] $\|\phi+c\|_p=\|\phi\|_p$ and $\|\phi_d\|_p=d\|\phi\|_p$. 
\end{itemize}
\end{lem}
    
\begin{proof} The first property in (i) follows from
  Proposition~\ref{prop:DHDFsemi}~(i). Let $(\cX,\cL)$ be the unique
  normal, ample representative of $\phi$, and denote by $(\cX',\cL')$
  the base change of $(\cX,\cL)$ by $t\mapsto t^d$. 
  Then $(\cX'_0,\cL'_0)\simeq(\cX_0,\cL_0)$, but with the $\G_m$-action composed with $t\mapsto t^d$. As a result, the $\G_m$-weights of $H^0(\cX'_0,m\cL'_0)$ are obtained by multiplying those of $H^0(\cX_0,m\cL_0)$ by $d$, and the second property in (i) follows. Part (ii) is a formal consequence of (i). 
 \end{proof}

\begin{proof}[Proof of Theorem~\ref{thm:trivmetric}] The equivalence between (i) and (ii) is immediate, and (iii)$\Longrightarrow$(ii) follows from Lemma~\ref{lem:DHtrans}. Conversely, assume that $\DH_\phi$ is a Dirac mass. By Theorem~\ref{thm:normzero}, the unique normal ample representative $(\cX,\cL)$ of $\phi$ is (almost) trivial. By Lemma~\ref{lem:triv}, this means that $(\cX,\cL+c\cX_0)=(X_{\A^1},L_{\A^1})$, \ie $\phi+c=\phi_\triv$.  
\end{proof}

\begin{rmk}\label{R304}
 For each ample representative $(\cX,\cL)$ of $\phi$, we have, by definition,
$$
\|\phi\|_p^p=\lim_{m\to\infty}\frac{1}{N_m}\sum_{\la\in\Z}|m^{-1}\la-\bar\la|^p\dim H^0(\cX_0,m\cL_0)_\la
$$
with 
$$
\bar\la=\lim_{m\to\infty}\frac{1}{mN_m}\sum_{\la\in\Z}\la\dim H^0(\cX_0,m\cL_0)_\la.
$$
This shows that the present definition generalizes the $L^p$-norm of an ample test configuration introduced in~\cite{Don2} for $p$ an even integer. 
\end{rmk}
%
%
\subsection{Intersection numbers}\label{S402}
Various operations on test configurations descend to non-Archimedean
metrics. As a first example, we discuss intersection numbers.

Every finite set of test configurations $\cX_i$ for $X$ is
dominated by some test configuration $\cX$. Given finitely many
non-Archimedean metrics $\phi_i$ on $\Q$-line bundles $L_i$, 
we may thus find representatives
$(\cX_i,\cL_i)$ for $\phi_i$ with $\cX_i=\cX$ independent of $i$. 

\begin{defi}\label{defi:int} 
  Let $\phi_i$ be a non-Archimedean metric 
  on $L_i$ for $0\le i\le n$.
  We define the \emph{intersection number} of the $\phi_i$ as
  \begin{equation}\label{e201}
    (\phi_0\cdot\ldots\cdot\phi_n):=(\bar\cL_0\cdot\ldots\cdot\bar\cL_n),
  \end{equation}
  where $(\cX_i,\cL_i)$ is any representative of $\phi_i$ 
  with $\cX_i=\cX$ independent of $i$, and where $(\bar\cX,\bar\cL_i)$
  is the compactification of $(\cX,\cL_i)$.
\end{defi}
By the projection formula, the right hand side of~\eqref{e201} is
independent of the choice of representatives. Note that the
intersection number $(\phi_0\cdot\ldots\cdot\phi_n)$ may be negative
even when the $L_i$ are ample and the $\phi_i$ are positive, since in
this case $\bar\cL_i$ is only \emph{relatively} semiample with respect to $\bar\cX\to\P^1$.  
\begin{rmk}\label{R401}
  When $L_0=\cO_X$, we can compute the intersection number
  in~\eqref{e201} without passing to the compactification. Indeed, if
  we write $\cL_0=\cO_\cX(D)$ and $D=\sum_Er_EE$, then 
  \begin{equation*}
    (\phi_0\cdot\ldots\cdot\phi_n)
    =\sum_Er_E(\cL_1|_E\cdot\ldots\cdot\cL_n|_E).
  \end{equation*}
  If $\phi_0\equiv 1$, that is, $D=\cX_0$, then 
  $(\phi_0\cdot\ldots\cdot\phi_n)=(L_1\cdot\ldots\cdot L_n)$ by 
  flatness of $\bar\cX\to\P^1$.
\end{rmk}
The intersection paring 
$(\phi_0,\dots,\phi_n)\mapsto(\phi_0\cdot\ldots\cdot\phi_n)$ is 
$\Q$-multilinear in its arguments in the sense of~\S\ref{S101}.
By the projection formula, it is invariant under pullbacks: if $Y$
is a projective normal variety of dimension $n$ and $f\colon Y\to X$
is a surjective morphism of degree $d$, then 
$(f^*\phi_0\cdot\ldots\cdot f^*\phi_n)=d(\phi_0\cdot\ldots\cdot\phi_n)$.
\begin{lem}\label{lem:int} 
 For non-Archimedean metrics $\phi_0,\dots,\phi_n$ on $\Q$-line
 bundles $L_0,\dots,L_n$ we have 
 \begin{equation*}
   ((\phi_0+c)\cdot\phi_1\cdot\ldots\cdot\phi_n)
   =(\phi_0\cdot\ldots\cdot\phi_n)
   +c(L_1\cdot\ldots\cdot L_n)
 \end{equation*}
 and
 \begin{equation*}
   ((\phi_0)_d\cdot\ldots\cdot(\phi_n)_d)=d(\phi_0\cdot\ldots\cdot\phi_n)
 \end{equation*}
 for all $d\in\N^*$ and $c\in\Q$. 
\end{lem}
\begin{proof} 
  The first equality is a consequence of the the discussion above, and the second
  formula follows from the projection formula.
\end{proof}
The following inequality is crucial. See~\cite{YZ16} for far-reaching generalizations.
\begin{lem}\label{lem:monotone} 
  Let $L_2,\dots,L_n$ be nef $\Q$-line bundles on $X$, 
  $\phi$ a non-Archimedean metric on $\cO_X$, and 
  $\phi_i$ a semipositive non-Archimedean metric on $L_i$ for $2\le i\le n$.
  Then 
  \begin{equation}\label{e302}
    \left(\phi\cdot\phi\cdot\phi_2\cdot\ldots\cdot\phi_n\right)\le 0.
  \end{equation}
\end{lem}
\begin{proof} Choose normal representatives $(\cX,\cL)$, $(\cX,\cL_i)$
  for $\phi$, with the same test configuration $\cX$ for $X$. We have 
  $\cL=\cO_\cX(D)$ for a $\Q$-Cartier divisor $D$ supported on
  $\cX_0$. Then~\eqref{e302} amounts to 
  $\left(D\cdot D\cdot\bar\cL_2\cdot\ldots\cdot\bar\cL_n\right)\le 0$,
  which follows from a standard Hodge Index Theorem argument;
  see~\eg~\cite[Lemma~1]{LX}.
 \end{proof}
%
%
\subsection{The non-Archimedean Monge-Amp\`ere measure}\label{S401}
Let $L$ be a big and nef $\Q$-line bundle on $X$, and set $V:=(L^n)$.
Then any $n$-tuple $(\phi_1,\dots,\phi_n)$ of non-Archimedean metrics on $L$
induces a signed finite atomic \emph{mixed Monge-Amp\`ere measure} on
$\Xdiv$ as follows. Pick representatives $(\cX,\cL_i)$ of $\phi_i$, $1\le i\le n$, with
the same test configuration $\cX$ for $X$ and set 
\begin{equation*}
  \MA^\NA(\phi_1,\dots,\phi_n)=V^{-1}\sum_Eb_E(\cL_1|_E\cdot\ldots\cdot\cL_n|_E)\delta_{v_E},
\end{equation*}
where $E$ ranges over irreducible components of $\cX_0=\sum_Eb_EE$,
and $v_E=r(b_E^{-1}\ord_E)\in\Xdiv$.
Note that 
\begin{equation*}
  \int_{\Xdiv}\MA^\NA(\phi_1,\dots,\phi_n)
  =V^{-1}(\cX_0\cdot\cL_1\cdot\ldots\cdot\cL_n)
  =V^{-1}(\cX_1\cdot\cL_1\cdot\ldots\cdot\cL_n)
  =V^{-1}(L^n)=1,
\end{equation*}
where the second equality follows from the flatness of $\cX\to\A^1$.
When the $\phi_i$ are semipositive, the mixed Monge-Amp\`ere measure is
therefore a probability measure.

As in the complex case, we also write $\MA^\NA(\phi)$ for $\MA^\NA(\phi,\dots,\phi)$.
Note that $\MA^\NA(\phi+c)=\MA^\NA(\phi)$ for any $c\in\Q$.
%
%
\subsection{Berkovich space interpretation}\label{S202}
Let us now briefly explain the term ``non-Archime\-dean
metric''. See~\cite{siminag,simons,trivval} for more details.

Equip the base field $k$ with the trivial absolute value $|\cdot|_0$,
\ie $|a|_0=1$ for $a\in k^*$. Also equip the field $K:=k\lau{t}$ 
of Laurent series with the non-Archimedean norm
in which $|t|=e^{-1}$ and $|a|=1$ for $a\in k^*$.

The Berkovich analytification $\Xan$ is a compact Hausdorff
space equipped with a structure sheaf~\cite{BerkBook}. It contains the set of valuations $v:k(X)^*\to\R$ on the function field of $X$ as a dense subset. Similarly, any line bundle $L$ on $X$ has an analytification
$\Lan$. The valued field extension $K/k$ further
gives rise to analytifications $X_K^{\mathrm{an}}$ and
$L_K^{\mathrm{an}}$, together with a natural morphism $X_K^\an\to\Xan$ under which $\Lan$ pulls pack to $L_K^\an$. 
The Gauss extension in~\S\ref{sec:valtest} gives a section $\Xan\to X_K^{\mathrm{an}}$, whose image exactly consists of
the $k^*$-invariant points.

After the base change $k[t]\to k\cro{t}$, any test configuration
$(\cX,\cL)$ defines a model of $(X_K,L_K)$ over the valuation ring
$k\cro{t}$ of $K=k\lau{t}$. When $\cX$ is normal, this further induces a continuous metric
on $L_K^{\mathrm{an}}$, \ie a function on the total space
satisfying certain natural conditions. Using the Gauss extension, we 
obtain a metric also on $\Lan$. 

Replacing a normal test configuration $(\cX,\cL)$ by a pullback does not change
the induced metric on $\Lan$, and one may in fact show that two normal test configurations induce the same metric iff they are equivalent in the Definition~\ref{defi:equiv}. This justifies the name non-Archimedean
metric for an equivalence class of test configurations. 
Further, in the analysis of~\cite{siminag,nama}, 
positive metrics play the role of K\"ahler potentials in complex geometry.

However, we abuse terminology a little since there are natural
metrics on $\Lan$ that do not come from test configurations. 
For example, any filtration on $R(X,L)$ defines a metric on $\Lan$. 
Metrics arising from test configurations can be viewed as analogues of
smooth metrics on a holomorphic line bundle. 
For some purposes it is important to work with a more flexible notion of
metrics, but we shall not do so here.
%
%
%
%
\section{Non-Archimedean functionals}\label{S201}
The aim of this section is to introduce non-Archimedean analogues
of several classical functionals in K\"ahler geometry; as indicated in
the introduction, the analogy will be turned into a precise connection 
in~\cite{BHJ2}.

Throughout this section, $X$ is a normal projective variety and $L$ a
$\Q$-line bundle on $X$. We shall assume that $L$ is big and nef, 
so that $V:=(L^n)>0$. The most important case is of course 
when $L$ is ample.
\begin{defi} 
  Let $V$ be a set of non-Archimedean metrics on $L$
  that is closed under translation and scaling. 
  Then a functional $F\colon V\to\R$ is 
  \emph{homogeneous} if $F(\phi_d)=d F(\phi)$ for $\phi\in V$ and
  $d\in\N^*$, and \emph{translation invariant} if 
  $F(\phi+c)=F(\phi)$ for $\phi\in V$ and $c\in\Q$. 
\end{defi}
For example, Lemma~\ref{lem:DHtrans} shows that when $L$ is ample,
the $L^p$-norm is a homogeneous and translation invariant functional 
on $\cH^{\NA}(L)$. 
%
%
\subsection{The non-Archimedean Monge-Amp\`ere energy}\label{sec:E}
\begin{defi}\label{defi:E} 
  The \emph{non-Archimedean Monge-Amp\`ere energy functional} is defined by
  \begin{equation*}
    E^{\NA}(\phi):=\frac{\left(\phi^{n+1}\right)}{(n+1)V}
  \end{equation*}
  for any non-Archimedean metric $\phi$ on $L$.
\end{defi}
Here $(\phi^{n+1})$ denotes the intersection number defined in~\S\ref{S402}.
Note that $E^\NA(\phi_\triv)=0$ since $(\phi_\triv^{n+1})=(L_{\P^1}^{n+1})=0$.
Lemma~\ref{lem:int} and Proposition~\ref{prop:DHDFsemi} imply:
\begin{lem}\label{lem:ENA} 
  The functional $E^{\NA}$ is homogeneous and satisfies 
  \begin{equation}\label{equ:transE}
    E^{\NA}(\phi+c)=E^{\NA}(\phi)+c
  \end{equation}
  for any non-Archimedean metric $\phi$ on $L$ and any $c\in\Q$. We further have 
  $$
  E^{\NA}(\phi)=\int_\R\la\,\DH_\phi(d\la).
  $$
  when $L$ is ample and $\phi\in\cH^{\NA}(L)$ is positive.
\end{lem}

\begin{lem}\label{lem:EMA} 
  For every non-Archimedean metric $\phi$ on $L$ we have 
  \begin{equation*}
    E^{\NA}(\phi)
    =\frac{1}{(n+1)V}\sum_{j=0}^n\left((\phi-\phi_\triv)\cdot\phi^j\cdot\phi_\triv^{n-j}\right).
  \end{equation*}
  Further, when $\phi$ is semipositive, we have, for $j=0,\dots,n-1$,
  \begin{equation}\label{e202}
    \left((\phi-\phi_\triv)\cdot\phi^j\cdot\phi_\triv^{n-j}\right)
    \ge\left((\phi-\phi_\triv)\cdot\phi^{j+1}\cdot\phi_\triv^{n-j-1}\right).
  \end{equation}
\end{lem}

\begin{proof} Since $(\phi_\triv^{n+1})=0$, we get
$$
(n+1)V E^{\NA}(\phi)=(\phi^{n+1})-(\phi_\triv^{n+1})=\sum_{j=0}^n\left((\phi-\phi_\triv)\cdot\phi^j\cdot\phi_\triv^{n-j}\right).
$$
The inequality~\eqref{e202}  is now a consequence of Lemma~\ref{lem:monotone}. 
\end{proof}
\begin{rmk}\label{R402}
  In view of Remark~\ref{R401} we can write the energy functional as
  \begin{equation*}
    E^\NA(\phi)=\frac{1}{n+1}\sum_{j=0}^n\frac1V\int_{\Xdiv}(\phi-\phi_\triv)\mu_j,
  \end{equation*}
  where $\mu_j=\MA^\NA(\phi,\dots,\phi,\phi_\triv,\dots,\phi_\triv)$ is a
  mixed Monge-Amp\`ere measure with $j$ copies of $\phi$. 
  Note that this formula is identical to its counterpart in K\"ahler geometry.
\end{rmk}
%
%
\subsection{The non-Archimedean $I$ and $J$-functionals}\label{sec:J}
\begin{defi}\label{defi:J} The \emph{non-Archimedean $I$ and $J$-functionals} 
are defined by
$$
I^{\NA}(\phi):=V^{-1}\left(\phi\cdot\phi_\triv^n\right)-V^{-1}\left((\phi-\phi_\triv)\cdot\phi^n\right)
$$
and
$$
J^{\NA}(\phi):=V^{-1}(\phi\cdot\phi_\triv^n)-E^{\NA}(\phi)
$$
for any non-Archimedean metric $\phi$ on $L$.
\end{defi} 
\begin{lem}\label{lem:sup} When $L$ is ample, we have 
$$
V^{-1}(\phi\cdot\phi_\triv^n)=(\phi-\phi_\triv)(v_\triv)=\sup_{X^{\mathrm{div}}}(\phi-\phi_\triv)=\sup\supp\DH_\phi
$$
for every positive metric $\phi\in\cH^{\NA}(L)$.
\end{lem}
\begin{proof} Choose a normal, semiample test configuration $(\cX,\cL)$ representing $\phi$ and such that $\cX$ dominates $X_{\A^1}$. Denote by $\rho\colon\cX\to X_{\A^1}$ the canonical morphism, so that $\cL=\rho^*L_{\A^1}+D$ for a unique $\Q$-Cartier divisor $D$ supported on $\cX_0$. Then
$$
(\phi\cdot\phi_\triv^n)=\left((\phi-\phi_\triv)\cdot\phi_\triv^n\right)=\left(D\cdot\rho^*L_{\A^1}^n\right)=(\rho_*D\cdot L_{\A^1}^n)=V\ord_{E_0}(D),
$$
with $E_0$ the strict transform of $X\times\{0\}$ on $\cX$. Theorem~\ref{thm:suppNA} yields the desired conclusion. 
\end{proof}
\begin{prop}\label{prop:J} The non-Archimedean functionals $I^{\NA}$ and $J^{\NA}$ 
are translation invariant and homogeneous. 
On the space of semipositive metrics, they are nonnegative and satisfy
\begin{equation}\label{e303}
  \tfrac 1 nJ^{\NA}\le I^{\NA}-J^{\NA}\le n J^{\NA}.
\end{equation}
When $L$ is ample, we further have
$$
J^{\NA}(\phi)=\sup\supp\DH_{\phi}-\int_\R\la\,\DH_{\phi}(d\la)
$$
for all positive metrics $\phi\in\cH^{\NA}(L)$. 
\end{prop}
\begin{proof} Translation invariance and homogeneity 
  follow directly from Lemma~\ref{lem:int}. 
  Now assume $\phi$ is semipositive.
  Then~\eqref{e202} shows that $I^{\NA}(\phi)\ge0$, 
  $J^{\NA}(\phi)\ge0$, and
  \begin{multline*}
    V^{-1}\left((\phi-\phi_\triv)\cdot\phi_\triv^n\right)
    +nV^{-1}\left((\phi-\phi_\triv)\cdot\phi^n\right)\\
    \le (n+1)E^{\NA}(\phi)
    \le n V^{-1}\left((\phi-\phi_\triv)\cdot\phi_\triv^n\right)
    +V^{-1}\left((\phi-\phi_\triv)\cdot\phi^n\right).
  \end{multline*}
  This implies
  \begin{multline*}
    n\left(I^{\NA}(\phi)-J^{\NA}(\phi)\right)
    =n\left(E^{\NA}(\phi)-V^{-1}\left((\phi-\phi_\triv)\cdot\phi^n\right)\right)\\
    \ge V^{-1}\left((\phi-\phi_\triv)\cdot\phi_\triv^n\right)-E^{\NA}(\phi)=J^{\NA}(\phi),
  \end{multline*}
  and similarly for the second inequality in~\eqref{e303}.
  
  The final assertion is a consequence of Lemma~\ref{lem:ENA} and Lemma~\ref{lem:sup}. 
\end{proof}

The above result shows that the functionals $J^\NA$ and $I^\NA$ are equivalent on the
space of semipositive metrics, in the following sense:
$$
\frac{n+1}{n} J^\NA\le I^\NA\le(n+1) J^\NA.
$$
We next show that they are also equivalent to the $L^1$-norm
$\|\cdot\|_1$ on positive metrics.
\begin{thm}\label{thm:J} 
  Assume $L$ is ample. Then, for every positive metric $\phi\in\cH^{\NA}(L)$, we have 
$$
c_n J^{\NA}(\phi)\le\|\phi\|_1\le 2 J^\NA(\phi)
$$
with $c_n:=2n^n/(n+1)^{n+1}$. In particular, $J^{\NA}(\phi)=0$ iff $\phi=\phi_\triv+c$ for some $c\in\Q$. 
\end{thm}

\begin{proof} The final assertion follows from Theorem~\ref{thm:trivmetric}. By translation invariance, we may assume, after replacing $\phi$ with $\phi+c$, that $\nu:=\DH_\phi$ has barycenter $\bar\la=0$. By Proposition~\ref{prop:J} and Definition~\ref{defi:DHmetric}, we then have 
$$
J^\NA(\phi)=\la_{\max}=\sup\supp\nu
$$
and $\|\phi\|_1=\int_\R|\la|d\nu$. By Theorem~\ref{thm:BC},
$f(\la)=\nu\{x\ge\la\}^{1/n}$ is further concave on
$(-\infty,\la_{\max})$. Theorem~\ref{thm:J} is now a consequence of Lemma~\ref{lem:logconc} below. 
\end{proof}
\begin{lem}\label{lem:logconc} 
  Let $\nu$ be a probability measure on $\R$ with compact support and such that 
  $\int_\R\la\,d\nu=0$. Assume also that $f(\la):=\nu\{x\ge\la\}^{1/n}$ is concave 
  on $(-\infty,\la_{\max})$, with $\la_{\max}=\max\supp\nu$. Then 
  \begin{equation}\label{equ:J1}
    c_n\la_{\max}
    \le\int|\la|d\nu
    \le2\la_{\max},
  \end{equation}
  with $c_n$ as above. 
\end{lem}
\begin{proof} 
  Since $\int_\R\la\,d\nu=0$, we have 
$$
\int_\R|\la|d\nu=2\int_0^{\la_{\max}}\la\,d\nu,
$$
giving the right-hand inequality in (\ref{equ:J1}). 
Our goal is to show that
$$
\int_0^{\la_{\max}}\la\,d\nu\ge\frac{n^n}{(n+1)^{n+1}}\la_{\max}.
$$
After scaling, we may and do assume for simplicity that $\la_{\max}=1$. Since $\nu$ is the distributional derivative of $-f(\la)^n$, it is easy to check that 
$$
\int_0^1\la\,d\nu=\int_0^1f(\la)^nd\la=\int_{-\infty}^0(1-f(\la)^n)d\la.
$$
Set $a:=f'(0_+)<0$ and $b:=f(0)\in(0,1)$. By concavity of $f$ on $(-\infty,1)$, we have $f(\la)\le a\la+b$ on $(-\infty,1)$ and 
$$
f(\la)\ge b(1-\la)+f(1_+)\ge b(1-\la)
$$ 
on $(0,1)$. This last inequality yields
\begin{equation}\label{equ:fb}
\int_0^1\la\,d\nu=\int_0^1{f(\la)}^n d\la\ge b^n\int_0^1(1-\la)^n d\la=\frac{b^n}{n+1}. 
\end{equation}
The first one shows that
\begin{equation*}
  \int_0^1\left(a\la+b\right)^n d\la\ge\int_0^1 f(\la)^n d\la
  =\int_{-\infty}^0\left(1-f(\la)^n\right)d\la\ge\int_{\la_0}^0\left(1-(a\la+b)^n\right) d\la,
\end{equation*}
with $\la_0<0$ defined by $a\la_0+b=1$. Computing the integrals, we infer
$$
\frac{1}{a(n+1)}\left((a+b)^{n+1}-b^{n+1}\right)\ge-\la_0+\frac{1}{a(n+1)}\left(1-b^{n+1}\right), 
$$
\ie
$$
(a+b)^{n+1}-b^{n+1}\le-a \la_0(n+1)+1-b^{n+1}.
$$
Since $-a\la_0=b-1$ and $a+b\ge f(1_-)\ge 0$, this shows that $0\le(b-1)(n+1)+1$, \ie $b\ge\frac{n}{n+1}$. Plugging this into (\ref{equ:fb}) yields the desired result. 
\end{proof}
\begin{rmk}
  The inequalities in Theorem~7.9 can be viewed as non-Archimedean analogues 
  of~\cite[(61)]{Dar15} and~\cite[Proposition~5.5]{DR15}.
\end{rmk}
\begin{rmk}\label{rmk:derJ} In our notation, the expression for the \emph{minimum norm} $\|(\cX,\cL)\|_m$ given in~\cite[Remark 3.11]{Der1} reads $\|(\cX,\cL)\|_m=\tfrac{1}{n+1}(\phi^{n+1})-\left((\phi-\phi_\triv)\cdot\phi^n\right)$, \ie 
$$
V^{-1}\|(\cX,\cL)\|_m=I^\NA(\phi)-J^{\NA}(\phi), 
$$
where $\phi\in\cH^{\NA}(L)$ denotes the 
metric induced by $(\cX,\cL)$. It therefore follows from our results
that the minimum norm is equivalent to the $L^1$-norm on positive metrics.
\end{rmk}
%
%
\subsection{The non-Archimedean Mabuchi functional}\label{sec:MNA}
From now on we assume that the base field $k$ has characteristic
$0$. We still assume that $X$ is a normal projective variety and $L$ a
nef and big $\Q$-line bundle on $X$.
Fix a boundary $B$ on $X$. 
Recall the notation introduced in~\S\ref{sec:logdisc} for the relative 
canonical and log canonical divisors.

When $L$ is ample, we can rewrite the
definition of the Donaldson-Futaki invariant with respect to
$((X,B);L)$ of a normal test configuration $(\cX,\cL)$ (see Definition~\ref{defi:DFpairs}) as
\begin{equation}\label{equ:DFB}
  \DF_B(\cX,\cL)
  =V^{-1}(K_{(\bar\cX,\cB)/\P^1}\cdot\bar\cL^n)+\bar S_B E^{\NA}(\cX,\cL). 
\end{equation}

This formula also makes sense when $L$ is not ample. 
Since canonical divisor classes are compatible under push-forward, the
projection formula shows that $\DF_B$ is invariant under pull-back,
hence descends to a functional, also denoted $\DF_B$,
on non-Archimedean metrics on $L$. 
While it is straightforward to see that $\DF_B$ is translation invariant, it is, however, \emph{not} homogeneous, and we therefore introduce an `error term' to recover this property. 
\begin{defi}\label{defi:M} 
  The \emph{non-Archimedean Mabuchi functional} with respect to $((X,B);L)$ is 
  \begin{align}
    M_B^{\NA}(\phi)
    :&=\DF_B(\phi)
       +V^{-1}\left((\cX_{0,\red}-\cX_0)\cdot\cL^n\right)\label{e403}\\
     &=V^{-1}\left(K^{\log}_{(\bar\cX,\bar\cB)/\P^1}\cdot\bar\cL^n\right)
    +\bar S_B E^{\NA}(\cX,\cL),\label{e308}
  \end{align}
  for any normal test configuration $(\cX,\cL)$ representing $\phi$.
\end{defi}
\begin{prop}\label{prop:M} 
  The non-Archimedean Mabuchi functional
  $M_B^{\NA}$ is translation invariant and homogeneous. 
\end{prop}
\begin{proof}
  Translation invariance is straightforward to verify. 
  As for homogeneity, it is enough to prove it for 
  \begin{equation*}
    (\cX,\cL)\mapsto\left(K^\lo_{(\bar\cX,\cB)/\P^1}\cdot\bar\cL^n\right).
  \end{equation*}
  As in~\cite[\S3]{LX}, this, in turn, is a consequence of the
  pull-back formula for log canonical divisors. More precisely, let
  $(\cX_d,\cL_d)$ be the normalized base change of $(\cX,\cL)$, and
  denote by $f_d\colon\P^1\to\P^1$ and $g_d\colon\bar\cX_d\to\bar\cX$ the
  induced finite morphisms, both of which have degree $d$. By~\eqref{e404} we have
  $K^\lo_{(\bar\cX_d,\bar\cB_d)/\P^1}=g_d^*K^\lo_{(\bar\cX,\bar\cB)/\P^1}$.
  Hence we get 
  \begin{equation}\label{e401}
    \left(K^\lo_{(\bar\cX_d,\bar\cB_d)/\P^1}\cdot\bar\cL_d^n\right)
    =d\left(K^\lo_{(\bar\cX,\bar\cB)/\P^1}\cdot\bar\cL^n\right)
  \end{equation}
  by the projection formula. 
\end{proof}
\begin{prop}\label{P201}
  We have $M_B^{\NA}(\phi)\le\DF_B(\phi)$ when $\phi$ is semipositive.
  Further, equality holds if $\phi$ is represented by a normal 
  test configuration $(\cX,\cL)$ with $\cX_0$ reduced.
\end{prop}
Indeed, the `error term' in~\eqref{e403} is nonpositive since $\bar\cL$ is relatively semiample.
While equality does not always hold in Proposition~\ref{P201}, we have 
the following useful result.
\begin{prop}\label{prop:weaksemi} 
  For every non-Archimedean metric $\phi$ on $L$ there exists 
  $d_0=d_0(\phi)\in\Z_{>0}$ such that 
  $\DF_B(\phi_d)=M_B(\phi_d)=dM_B^{\NA}(\phi)$ 
  for all $d$ divisible by $d_0$.
\end{prop}
\begin{proof} 
  Let $(\cX,\cL)$ be any normal representative of $\phi$. 
  Then a normal representative $(\cX_d,\cL_d)$ of $\phi_d$ is given by
  the normalization of the base change of $(\cX,\cL)$ by $t\mapsto t^d$.
  It is well-known (see~\eg~\cite[\href{stacks.math.columbia.edu/tag/09IJ}{Tag 09IJ}]{stacks})
  that the central fiber of $\cX_d$ is reduced for $d$ sufficiently divisible.
  Then $\DF_B(\phi_d)=M_B(\phi_d)$,
  whereas $M_B(\phi_d)=dM_B(\phi)$ by Proposition~\ref{prop:M}.
\end{proof}
%
%
\subsection{Entropy and Ricci energy}\label{S301}
Next we define non-Archimedean analogues of the entropy and Ricci
energy functionals, and prove that the Chen-Tian formula holds. 
\begin{defi} 
  We define the \emph{non-Archimedean entropy} $H_B^{\NA}(\phi)$ 
  of a non-Archimedean metric $\phi$ on $L$ by 
  \begin{equation*}
    \int_{\Xdiv}A_{(X,B)}(v)\MA^{\NA}(\phi),
  \end{equation*}
  where $\MA^\NA(\phi)$ is the non-Archimedean Monge-Amp\`ere measure
  of $\phi$, defined in~\S\ref{S401}.
\end{defi}
Concretely, pick a normal test configuration $(\cX,\cL)$ for $(X,L)$ representing $\phi$,
and write $\cX_0=\sum_Eb_EE$ and $v_E=b_E^{-1}r(\ord_E)$. Then
\begin{equation}\label{e402}
  H_B^{\NA}(\phi):=V^{-1}\sum_E A_{(X,B)}(v_E)b_E(E\cdot\cL^n).
\end{equation}
Note that $H_B^{\NA}(\phi)\ge0$ whenever $(X,B)$ is lc and $\phi$ is
semipositive. Indeed, in this case we have $A_{(X,B)}(v_E)\ge0$
and $(E\cdot\cL^n)\ge0$ for all $E$. 
See~\S\ref{S304} for much more precise results.

As an immediate consequence of~\eqref{e402} and Corollary~\ref{cor:discr}, we have
\begin{cor}\label{C402}
  If $(\cX,\cL)$ is a normal representative of $\phi$, with $\cX$
  dominating $X_{\A^1}$, then 
  \begin{equation}\label{e305}
    H^\NA_B(\phi)
    =V^{-1}\left(K^\lo_{(\bar\cX,\bar\cB)/\P^1}\cdot\bar\cL^n\right)
    -V^{-1}\left(\rho^*K^\lo_{(X_{\P^1},B_{\P^1})/\P^1}\cdot\bar\cL^n\right),
  \end{equation}
  where $\rho\colon\bar\cX\to X_{\P^1}$ is the canonical morphism.
\end{cor}
\begin{cor}\label{C401}
  The non-Archimedean entropy functional
  $H_B^{\NA}$ is translation invariant and homogeneous. 
\end{cor}
\begin{proof}
  Translation invariance is clear from the definition, since 
  $\MA^\NA(\phi+c)=\MA(\phi)$, and homogeneity follows from~\eqref{e401} 
  and~\eqref{e305}.
\end{proof}
\begin{defi} 
  The \emph{non-Archimedean Ricci energy} $R_B^{\NA}(\phi)$ 
  of a non-Archimedean metric $\phi$ on $L$ is
  \begin{equation*}
    R^{\NA}_B(\phi):=V^{-1}\left(\psi_\triv\cdot\phi^n\right),
  \end{equation*}
  with $\psi_\triv$ the trivial non-Archimedean metric on $K_{(X,B)}$. 
\end{defi}
More concretely, if $(\cX,\cL)$ is a normal representative of $\phi$, with 
$\cX$ dominating $X_{\A^1}$, then 
\begin{equation}\label{e307}
  R^{\NA}_B(\phi)
  =V^{-1}\left(\rho^*K^\lo_{(X_{\P^1},B_{\P^1})/\P^1}\cdot\bar\cL^n\right) 
  =V^{-1}\left(p^*K_{(X,B)}\cdot\bar\cL^n\right),
\end{equation}
with $p\colon\bar\cX\to X$ the composition of $\rho\colon\bar\cX\to X_{\P^1}$ 
with $X_{\P^1}\to X$.
\begin{prop}\label{P401}
  The non-Archimedean Ricci energy functional $R_B^{\NA}$ is homogenous and 
  satisfies $R_B^\NA(\phi+c)=R_B^\NA(\phi)-\bar{S}_B c$ for any $c\in\Q$.
\end{prop}
\begin{proof}
  Homogeneity follows from~\eqref{e401} and~\eqref{e307}. The formula for
  $R_B^\NA(\phi+c)$ also follows from~\eqref{e307}. 
  Indeed, set $\bar\cM:=\rho^*K^\lo_{(X_{\P^1},B_{\P^1})/\P^1}$. 
  Then 
  \begin{multline*}
    R^\NA_B(\phi+c)-R^\NA(\phi)
    =V^{-1}(\bar\cM\cdot(\bar\cL+c\cX_0)^n)
    -V^{-1}(\bar\cM\cdot\bar\cL^n)\\
    =cnV^{-1}(\bar\cM\cdot\bar\cL^{n-1}\cdot\cX_0)
    =cnV^{-1}(K_X\cdot L^n)
    =-\bar{S}_Bc.
  \end{multline*}
  by flatness of $\bar\cX\to\P^1$, 
  since $\bar\cM|_{\cX_1}\simeq K_{(X,B)}$, $\bar\cL|_{\cX_1}\simeq L$, 
  and $\cX_0\cdot\cX_0=0$.
\end{proof}

As an immediate consequence of~\eqref{e308},~\eqref{e305} 
and~\eqref{e307} we get
\begin{prop}\label{prop:DFChen} 
  The following version of the Chen-Tian formula holds:
  \begin{equation*}
    M_B^{\NA}=H_B^{\NA}+R^{\NA}_B+\bar S_B E^{\NA}.
  \end{equation*}
\end{prop}
\begin{rmk} In the terminology of~\cite{Oda1}, $H_B^{\NA}(\phi)+V^{-1}\left((\cX_0-\cX_{0,\red})\cdot\cL^n\right)$ coincides (up to a multiplicative constant) with the `discrepancy term' of the Donaldson-Futaki invariant,  while $\bar S_B\,E^{\NA}(\phi)+R^{\NA}_B(\phi)$ corresponds to the `canonical divisor part'.
\end{rmk}
%
%
\subsection{Functoriality}\label{S305}
Consider a birational morphism
$\mu\colon X'\to X$, with $X'$ a normal projective variety. 
Set $L':=\mu^*L$ and define a boundary $B'$ on $X'$ by 
$K_{(X',B')}=\mu^*K_{(X,B)}$ and $\mu_*B'=B$. 
For any non-Archimedean metric $\phi$ on
$L$, let $\phi'=\mu^*\phi$ be the pullback, see~\S\ref{S101}.
Note that $L'$ is big and nef. By the projection formula, we have 
$V':=((L')^n)=V$ and 
$\bar S_{B'}:=n(V')^{-1}\left(-K_{(X',B')}\cdot (L')^{n-1}\right)=\bar S_B$.

Let us say that a functional $F=F_{X,B}$ on non-Archimedean metrics 
is \emph{pull-back invariant} if $F_{X',B'}(\phi')=F_{X,B}(\phi)$
for every non-Archimedean metric $\phi$ on $L$.
\begin{prop}
  The functionals $E^\NA$, $I^\NA$, $J^\NA$, $\DF_B^\NA$, $M_B^\NA$,
  $H_B^\NA$ and $R_B^\NA$ are all pullback invariant.
\end{prop}
\begin{proof}
  Let $(\cX,\cL)$ be a normal representative of $\phi$ such that 
  $\cX$ dominates $X_{\A^1}$. Pick a normal test configuration $\cX'$
  that dominates $X'_{\A^1}$ and such that unique $\G_m$-equivariant
  birational map $\cX'\to\cX$ extending $\mu$ is a morphism. Then $\phi'$
  is represented by $(\cX',\cL')$, where $\cL'$ is the pullback 
  of $\cL$. The pullback-invariance of the all the functionals now
  follows from the projection formula for the induced map 
  $\bar\cX'\to\bar\cX$.
\end{proof}
Recall from~\S\ref{S101} that if $\phi$ is a non-Archimedean metric on 
$L$, then $r\phi$ is a non-Archimedean metric on $rL$ for any
$r\in\Q_{>0}$. One directly verifies that the functionals $E^{\NA}$,
$I^{\NA}$ and $J^\NA$ are homogeneous of degree 1 in the sense that 
$E^\NA(r\phi)=rE^\NA(\phi)$ etc, whereas the functionals 
$\DF_B$, $M_B$, $H_B$ and $R_B$ are homogeneous of degree 0,
that is, $\DF_B(r\phi)=\DF_B(\phi)$ etc.
%
%
\subsection{The log K\"ahler-Einstein case}\label{S302}
In the \emph{log K\"ahler-Einstein case}, \ie when $K_{(X,B)}$ 
is proportional to $L$, the formula for $M_B^{\NA}$ takes the following alternative form. 
\begin{lem}\label{lem:MKE} Assume that $K_{(X,B)}\equiv\la L$ for some $\la\in\Q$. Then
$$
M_B^{\NA}=H_B^{\NA}+\la\left(I^{\NA}-J^{\NA}\right).
$$
\end{lem}
\begin{proof} 
  Let $\psi_\triv$ and $\phi_\triv$ be the trivial non-Archimedean
  metrics on $K_{(X,B)}$ and $L$, respectively. Since 
  $K_{(X_{\P^1},B_{\P^1})}\equiv\la L_{\P^1}$ we get
  \begin{equation*}
    R^{\NA}_B(\phi)=V^{-1}(\psi_\triv\cdot\phi^n)=\la V^{-1}\left(\phi_\triv\cdot\phi^n\right). 
  \end{equation*}
  Further, $\bar S_B=-n\la$, so we infer
  \begin{multline*}
    R^{\NA}_B(\phi)+\bar S_B E^{\NA}(\phi)
    =\la V^{-1}\left[\left(\phi_\triv\cdot\phi^n\right)
      -\frac{n}{n+1}(\phi^{n+1}) \right]\\
    =\la V^{-1}\left[\frac{1}{n+1}(\phi^{n+1})
      -\left((\phi-\phi_\triv)\cdot\phi^n\right)\right]\\
    =\la\left[E^{\NA}(\phi)-V^{-1}\left((\phi-\phi_\triv)\cdot\phi^n\right)\right]
    =\la\left(I^{\NA}(\phi)-J^{\NA}(\phi)\right),
  \end{multline*}
  which completes the proof in view of the Chen-Tian formula.
\end{proof}
%
%
\subsection{The non-Archimedean Ding functional}\label{sec:Ding}
In this section, $(X,B)$ denotes a \emph{weak log Fano pair}, \ie $X$ is a normal, projective variety 
and $B$ is a $\Q$-Weil divisor such that $(X,B)$ is subklt with $L:=-K_{(X,B)}$ big and nef. 
For example, $X$ could be smooth, with $-K_X$ ample (and $B=0$).

The following non-Archimedean version of the Ding functional
first appeared  in~\cite{Berm16}.\footnote{This appears 
in~\cite[Proposition~3.8]{Berm16}. See also Proposition~\ref{P403} below.}
It plays a crucial role in the variational approach 
to the Yau-Tian-Donaldson conjecture in~\cite{BBJ15}; see also~\cite{Fuj15b,Fuj16}.
The usual Ding functional was introduced in~\cite{Din88}.
\begin{defi} 
  The \emph{non-Archimedean Ding functional} is defined by
  \begin{equation*}
    D^{\NA}_B:=L^{\NA}_B-E^{\NA},
  \end{equation*}
  with
  \begin{equation*}
    L^\NA_B(\phi):=\inf_v(A_{(X,B)}(v)+(\phi-\phi_{\triv})(v)),
  \end{equation*}
  the infinimum taken over all valuations $v$ on $X$ that are divisorial or trivial.
\end{defi} 
Recall that $\phi-\phi_\triv$ is a non-Archimedean 
metric on $\cO_X$, which we identify with a bounded function on divisorial
valuations.
\begin{prop}\label{P402}
  The non-Archimedean Ding functional $D_B^{\NA}$ is translation invariant, 
  homogenous, and pullback invariant.
\end{prop}
\begin{proof}
  By the corresponding properties of the functional $E^\NA$, it suffices to prove that 
  $L_B^\NA$ is homogenous, pullback invariant, and satisfies 
  $L_B^\NA(\phi+c)=L_B^\NA(\phi)+c$ for $c\in\Q$. 

  The latter equality is clear from the definition, and the homogeneity of $D_B^\NA$
  follows from~\eqref{e405} applied to the metric $\phi-\phi_\triv$ on $\cO_X$,
  together with the fact that $A_{(X,B)}(tv)=tA_{(X,B)}(v)$ for $t\in\Q_+$.
  Functoriality is also clear. Indeed, with notation as in~\S\ref{S305}, 
  and with the identification of divisorial (or trivial) valuations on $X$
  and $X'$, we have, by construction, $\phi'-\phi'_\triv=\phi-\phi_\triv$ and 
  $A_{(X,B)}=A_{(X',B')}$. Thus $L^\NA_{B'}(\phi')=L^\NA_B(\phi)$.
\end{proof}
\begin{prop}\label{P404}
  For every non-Archimedean metric $\phi$ on $L$, we have
  $D^{\NA}_B(\phi)\le J^{\NA}(\phi)$.
\end{prop}
\begin{proof}
  The trivial valuation $v_\triv$ on $X$ satisfies
  $A_{(X,B)}(v_\triv)=0$ and
  $E^\NA(\phi)+J^\NA(\phi)=(\phi-\phi_\triv)(v_\triv)$.
  Hence
  \begin{equation*}
    L^\NA_B(\phi)
    \le A_{(X,B)}(v_\triv)+(\phi-\phi_\triv)(v_\triv)
    =E^\NA(\phi)+J^\NA(\phi),
  \end{equation*}
  which yields $D^\NA_B(\phi)\le J^\NA(\phi)$. 
\end{proof}
In the definition of the Ding functional, we take the infimum over all 
divisorial valuations on $X$. As the next result shows, this is neither practical nor necessary.
\begin{prop}\label{P403}
  Let $\phi$ be a non-Archimedean metric on $L=-K_{(X,B)}$ determined on a normal 
  test configuration $(\cX,\cL)$ for $(X,L)$, such that 
  $(\cX,\cB+\cX_{0,\red})$ is a sublc pair.
  Write
  \begin{equation*}
    \cL+K_{(\cX,\cB)/\A^1}^\lo=\cO_\cX(D),
  \end{equation*}
  for a $\Q$-Cartier divisor $D$ on $\cX$ supported on $\cX_0$.
  Then 
  \begin{align*}
    L^{\NA}_B(\phi)
    &=\lct_{(\cX,\cB+\cX_{0,\red}-D)}(\cX_0)\\
    &=\min_E\left(A_{(X,B)}(v_E)+(\phi-\phi_{\triv})(v_E)\right),
  \end{align*}
  where $E$ ranges over the irreducible components of $\cX_0$.
\end{prop}
Note that the assumption that $(\cX,\cB+\cX_{0,\red})$ be sublc is satisfied when 
$(\cX,\cB+\cX_0)$ is log smooth (even when $\cX_0$ is not necessarily reduced).

Proposition~\ref{P403} shows in particular that the definition of $D^\NA_B$ given above 
is compatible with~\cite{Berm16,Fuj15b}. By~\cite[Proposition 3.8]{Berm16}, the non-Archimedean Ding 
functional is thus the limit of the usual Ding functional in the sense of~\eqref{e301}; 
hence the name.
\begin{lem}\label{L401}
  Let $w$ be a divisorial valuation $w$ on $\cX$ centered on $\cX_0$ and normalized by 
  $w(\cX_0)=1$, and let $v=r(w)$ be the associated divisorial (or trivial) valuation on $X$.
  Then 
  \begin{equation}\label{e407}
    A_{(\cX,\cB+\cX_{0,\red})}(w)+w(D)
    =A_{(X,B)}(v)+(\phi-\phi_\triv)(v).
  \end{equation}
  In particular, $\ord_E(D)=b_E(A_{(X,B)}(v_E)+(\phi-\phi_\triv)(v_E))$ for
  every irreducible component $E$ of $\cX_0$.
\end{lem}
\begin{proof}
  Pick any normal test configuration $\cX'$ for $X$ dominating both 
  $\cX$ and $X_{\A^1}$ via $\mu\colon\cX'\to\cX$ and $\rho\colon\cX'\to X_{\A^1}$, 
  respectively, such that $w=b_{E'}^{-1}\ord_{E'}$ for an irreducible component 
  $E'$ of $\cX'_0$. 
  By~\eqref{equ:Klog} and~\eqref{equ:Klogbis} we have 
  \begin{equation*}
    K^\lo_{(\cX',\cB')/\A^1}-\mu^*K^\lo_{(\cX,\cB)/\A^1}
    =\sum_{E'}A_{(\cX,\cB+\cX_{0,\red})}(\ord_{E'})E',
  \end{equation*}
  and 
  \begin{equation*}
    K^\lo_{(\cX',\cB')/\A^1}-\rho^*K^\lo_{(X_{\A^1},B_{\A^1})/\A^1}
    =\sum_{E'} b_{E'}A_{(X,B)}(v_{E'})E',
  \end{equation*}
  respectively. We also have
  \begin{equation*}
    \mu^*\cL
    =-\rho^*K^\lo_{(X_{\A^1},B_{\A^1})/\A^1}
    +\sum_{E'}b_{E'}(\phi-\phi_\triv)(v_{E'})E'
  \end{equation*}
  Putting this together, and using $D=\cL+K_{(\cX,\cB)/\A^1}^\lo$, we get
  \begin{equation*}
    \mu^*D+\sum_{E'}A_{(\cX,\cB+\cX_{0,\red})}(\ord_{E'})E'
    =\sum_{E'}b_{E'}(A_{(X,B)}(\nu_{E'})+(\phi-\phi_\triv)(v_{E'}))E',
  \end{equation*}
  and taking the coefficient along $E'$ yields~\eqref{e407}.
  Finally, the last assertion follows since
  $A_{(\cX,\cB+\cX_{0,\red})}(w)=0$ when $w=b_E^{-1}\ord_E$ for any irreducible
  component $E$ of $\cX_0$.
\end{proof}
\begin{proof}[Proof of Proposition~\ref{P403}]
  Recall that $\lct_{(\cX,\cB+\cX_{0,\red}-D)}(\cX_0)$ is the supremum of $c\in\R$
  such that 
  \begin{equation*}
    0\le A_{(\cX,\cB+\cX_{0,\red}-D+c\cX_0)}(w)
    =A_{(\cX,\cB+\cX_{0,\red})}(w)+w(D)-cw(\cX_0)
  \end{equation*}
  for all divisorial valuations $w$ on $\cX$.
  Here it suffices to consider $w$ centered on $\cX_0$. Indeed, otherwise
  $w(D)=w(\cX_0)=0$ and 
  $A_{(\cX,\cB+\cX_{0,\red})}(w)=A_{(X_{\A^1},B_{\A^1})}(w)\ge 0$, 
  since $(X_{\A^1},B_{\A^1})$ is sublc.
  If $w$ is centered on $\cX_0$, then we may after scaling assume that $w(\cX_0)=1$.
  In this case,~\eqref{e407} applies, and shows that 
  $\lct_{(\cX,\cB+\cX_{0,\red}-D)}(\cX_0)=L^\NA(\phi)$.
  
  It remains to prove that $L^\NA(\phi)\le\ell:=\min_E(A_{(X,B)}(v_E)+(\phi-\phi_{\triv})(v_E))$,
  where $E$ ranges over irreducible components of $\cX_0$.
  The inequality $L^\NA(\phi)\le\ell$ is obvious. 
  For the reverse inequality, note that Lemma~\ref{L401} implies $D\ge\ell\cX_0$.
  We now use the assumption that 
  $(\cX,\cB+\cX_{0,\red})$ is sublc. Consider $w$ and $v$ as above.
  On the one hand, $(\cX,\cB+\cX_{0,\red})$ being sublc implies $A_{(\cX,\cB+\cX_{0,\red})}(w)\ge0$.
  On the other hand, we have $w(D)\ge\ell$ since $D\ge\ell\cX_0$.
  Thus~\eqref{e407} yields $A_{(X,B)}(v)+(\phi-\phi_\triv)(v)\ge\ell$.
  Since this is true for all divisorial or trivial valuations on $X$, we get 
  $L^\NA(\phi)\ge\ell$, which completes the proof.
\end{proof}
%
%
\subsection{Ding vs Mabuchi}
We continue to assume that $(X,B)$ is a weak log Fano pair.
By Lemma~\ref{lem:MKE},  the non-Archimedean Mabuchi functional is given by 
\begin{equation}\label{e109}
  M^\NA_B=H^\NA_B-(I^\NA-J^\NA). 
\end{equation}
For any normal test configuration $(\cX,\cL)$ representing 
a non-Archimedean metric $\phi$ on $L$, we can write this as 
\begin{align}
  M^\NA_B(\phi)
  &=V^{-1}((K^\lo_{(\cX,\cB)/\A^1}+\cL)\cdot\cL^n)
    -E^\NA(\phi)\label{e408}\\
  &=\sum_E c_E(A_{(X,B)}(v_E)+(\phi-\phi_\triv)(v_E))
    -E^\NA(\phi),\label{e409}
\end{align}
where $E$ ranges over the irreducible components of $\cX_0$ and 
$c_E:=V^{-1} b_E(\cL^n\cdot E)$. Note that $\sum_E c_E=1$ and that $c_E\ge0$
if $\phi$ is semipositive.
\begin{defi} 
  A non-Archimedean metric $\phi$ on $L=-K_{(X,B)}$ is \emph{anticanonical} if 
  it is represented by a normal test configuration $(\cX,\cL)$ for $(X,-K_{(X,B)})$
  such that 
  $(\cX,\cB+\cX_{0,\red})$ is sublc and such that 
  $\cL=-K^\lo_{(\cX,\cB)/\A^1}+c\cX_0$ for some $c\in\Q$. 
\end{defi}
Note that if $\phi$ is anticanonical, then so is $\phi+c$ for any $c\in\Q$.

\begin{prop}\label{P102} 
  For every semipositive non-Archimedean metric $\phi$ on $L$, we have
  \begin{equation*}
    D_B^\NA(\phi)\le M^\NA_B(\phi),
  \end{equation*}
  with equality if $\phi$ is anticanonical. 
\end{prop}
\begin{rmk}
  In K\"ahler geometry, the inequality $D_B(\phi)\le M_B(\phi)$ is well-known,
  and equality holds iff $\phi$ is a K\"ahler-Einstein metric, 
  see~\eg~\cite[Lemma~4.4]{BBEGZ}.
  Proposition~\ref{P102} therefore suggests that 
  semipositive anticanonical non-Archimedean metrics on $-K_{(X,B)}$
  play the role of (weak) non-Archimedean K\"ahler-Einstein metrics.
\end{rmk}
\begin{proof} 
  Consider the expression~\eqref{e409} for $M^\NA(\phi)$. 
  Since $\phi$ is semipositive, we have $c_E\ge 0$ and $\sum_Ec_E=1$.
  This implies 
  \begin{multline*}
    M^\NA_B(\phi)
    \ge\min_E(A_{(X,B)}(v_E)+(\phi-\phi_\triv)(v_E))-E^\NA(\phi)\\
    \ge\inf_v(A_{(X,B)}(v)+(\phi-\phi_\triv)(v))-E^\NA(\phi)
    =D^\NA_B(\phi).
  \end{multline*}
  Now suppose $\phi$ is anticanonical and let $(\cX,\cL)$ be a test configuration for $(X,L)$
  such that $(\cX,\cB+\cX_{0,\red})$ is sublc and such that 
  $\cL=-K^\lo_{(\cX,\cB)/\A^1}+c\cX_0$ for some $c\in\Q$.
  
  On the one hand,~\eqref{e408} gives $M^\NA_B(\phi)=c-E^\NA(\phi)$.
  On the other hand, Lemma~\ref{L401} yields $A(v_E)+(\phi-\phi_\triv)(v_E)=c$ for all 
  irreducible components $E$ of $\cX_0$. Thus Proposition~\ref{P403}
  implies that $D_B^\NA(\phi)=c-E^\NA(\phi)$,
  which completes the proof.
\end{proof}
%
%
%
%
\section{Uniform K-stability}\label{sec:Kstab}
We continue working with a pair $(X,B)$, where $X$ is a normal
projective variety over a algebraically closed field $k$ of
characteristic zero.
In this section, we further assume that the $\Q$-line bundle $L$ is ample.
%
%
\subsection{Uniform K-stability} 
In the present language, Definition~\ref{defi:logKstab} says that
$((X,B);L)$ is K-semistable iff $\DF_B(\phi)\ge 0$ for
all positive metrics $\phi\in\cH^\NA(L)$, while K-stability further requires that $\DF_B(\phi)=0$ only when $\phi=\phi_\triv+c$ for some $c\in\Q$. In line with the point of view of~\cite{Sze2}, we introduce:
\begin{defi}\label{defi:unifKstab} The polarized pair $((X,B);L)$ is
  \emph{$L^p$-uniformly K-stable} if $\DF_B\ge\d\|\cdot\|_p$ on
  $\cH^\NA(L)$ for some uniform constant $\d>0$. 
  For $p=1$, we simply speak of \emph{uniform K-stability}. 
\end{defi} 
Since $\|\cdot\|_p\ge\|\cdot\|_1$, 
$L^p$-uniform K-stability implies ($L^1$-)uniform K-stability for any $p\ge1$. 
Note also that uniform K-stability implies (as it should!)
K-stability, thanks to Theorem~\ref{thm:trivmetric}.

\begin{prop}\label{prop:coer} The polarized pair $((X,B);L)$ is K-semistable iff $M_B^\NA\ge 0$ on $\cH^\NA(L)$. It is $L^p$-uniformly K-stable iff $M_B^\NA\ge\d\|\cdot\|_p$ on $\cH^\NA(L)$ for some $\d>0$. For $p=1$, this is also equivalent to $M_B^\NA\ge\d J^\NA$ for some $\d>0$. 
\end{prop} 
\begin{proof} We prove the second point, the first one being similar (and easier). The if part is clear, since $M_B^{\NA}\le\DF_B$. For the reverse implication, let $\phi\in\cH^{\NA}(L)$.
  By Proposition~\ref{P201} we can pick $d\ge 1$ such that
  $M_B^{\NA}(\phi_d)=\DF_B(\phi_d)$. 
  By assumption, $\DF_B(\phi_d)\ge\d \|\phi_d\|_p$,
  and we conclude by homogeneity of $M_B^{\NA}$ and $\|\cdot\|_p$. 
  
 The final assertion is now 
 a consequence of the equivalence between $J^\NA$ and $\|\cdot\|_1$ proved in Theorem~\ref{thm:J}.  
\end{proof}
\begin{rmk} By Remark~\ref{rmk:derJ}, our notion of uniform
  K-stability is also equivalent to
  uniform K-stability with respect to the minimum norm in the 
  sense of~\cite{Der1}. 
\end{rmk}
\begin{rmk}\label{R303}
  It is clear that, for any $r\in\Q_{>0}$ and
  $p\ge 1$, $((X,B);L)$ is K-semistable (resp.\ $L^p$-uniformly
  K-stable)
  iff $((X,B);rL)$ is K-semistable (resp.\ $L^p$-uniformly K-stable).
\end{rmk}
The next result confirms G.~Sz\'ekelyhidi's expectation that
$p=\tfrac{n}{n-1}$ is a threshold value for $L^p$-uniform K-stability, 
cf.~\cite[\S3.1.1]{Sze1}. 
\begin{prop}\label{prop:thresh} 
  A polarized pair $((X,B);L)$ cannot be $L^p$-uniformly K-stable
  unless $p\le\frac{n}{n-1}$. More precisely, any polarized pair $((X,B);L)$ 
  admits a sequence $\phi_\e\in\cH^{\NA}(L)$,  paramet\-rized by $0<\e\ll 1$ rational, 
  such that $M_B^{\NA}(\phi_\e)\sim\e^n$, $\|\phi_\e\|_p\sim\e^{1+\frac{n}{p}}$ for each $p\ge 1$. 
\end{prop}
\begin{proof} 
  We shall construct $\phi_\e$ as a small perturbation of the trivial metric.
  By Remark~\ref{R303} we may assume that $L$ is an actual line bundle.
  Let $x\in X\smallsetminus\supp B$ be a regular closed point, and $\rho\colon\cX\to X_{\A^1}$ 
  be the blow-up of $(x,0)$ (\ie the deformation to the normal cone), with exceptional 
  divisor $E$. For each rational $\e>0$ small enough, $\cL_\e:=\rho^*L_{\A^1}-\e E$ is 
  relatively ample, and hence defines a normal, ample test configuration $(\cX,\cL_\e)$ for $(X,L)$,
  with associated non-Archimedean metric $\phi_\e\in\cH^{\NA}(L)$.

  Lemma~\ref{lem:filtr} gives the following description of the filtration $F^\bullet_\e R$ 
  attached to $(\cX,\cL_\e)$:
  \begin{equation*}
    F_\e^{m\la} H^0(X,mL)=\left\{s\in H^0(X,mL)\mid v_E(s)\ge m(\la+\e)\right\}
  \end{equation*}
  for $\la\le 0$, and $F_\e^{m\la}H^0(X,mL)=0$ for $\la>0$.
  If we denote by $F$ the exceptional divisor of the blow-up $X'\to X$ at $x$, 
  then $v_E=\ord_F$, and the Duistermaat-Heckman measure $\DH_\e$ is thus given by
  \begin{equation*}
    \DH_\e\{x\ge\la\}=V^{-1}\left(\rho^*L-(\la+\e)F\right)^n=1-V^{-1}(\la+\e)^n
  \end{equation*}
  for $\la\in(-\e,0)$, $\DH_\e\{x\ge\la\}=1$ for $\la\le-\e$, and $\DH_\e\{x\ge\la\}=0$ 
  for $\la>0$.
  Hence
  \begin{equation*}
    \DH_\e=nV^{-1}{\bf 1}_{[-\e,0]}(\la+\e)^{n-1}d\la+(1-V^{-1}\e^n)\d_0.
  \end{equation*}
  We see from this that $\la_{\mathrm{max}}=0$,
  \begin{equation*}
    J^{\NA}(\phi_\e)
    =-E^{\NA}(\phi_\e)
    =-\int_\R\la\,\DH_\e(d\la)
    =-\frac{n}{V}\int_{-\e}^0\la(\la+\e)^{n-1}\,d\la
    =O(\e^{n+1}),
  \end{equation*}
  and
  \begin{multline*}
    \|\phi_\e\|_p^p=\int_\R\left|\la-E^{\NA}(\phi_\e)\right|^p\,\DH_\e(d\la)\\
    =nV^{-1}\int_{-\e}^0\left|\la+O(\e^{n+1})\right|^p(\la+\e)^{n-1}d\la
    +(1-V^{-1}\e^n)O(\e^{p(n+1)})\\
    =\e^{p+n}\left[nV^{-1}\int_0^1\left|t+O(\e^n)\right|^p(1-t)^{n-1}dt
      +O(\e^{n(p-1)})+o(1)\right]\\
    =\e^{p+n}(c+o(1))
  \end{multline*}
  for some $c>0$. 
  The estimate for $M_B^{\NA}(\phi_\e)$ is a special case of 
  Proposition~\ref{prop:neartriv} below, but let us give a direct proof.
  By~\eqref{e308} it suffices to prove that 
  $(K^{\log}_{(\bar\cX,\bar\cB)/\P^1}\cdot\bar\cL_\e^n)\sim\e^n$.
  Here $K^{\log}_{(\bar\cX,\bar\cB)/\P^1}=\rho^*K^\lo_{(X_{\P^1},B_{\P^1})}+(n+1)E$.
  Since $\rho_*E^j=0$ for $0\le j\le n$ and $((-E)^{n+1})=-1$, the projection formula
  yields $(K^{\log}_{(\bar\cX,\bar\cB)/\P^1}\cdot\bar\cL_\e^n)=(n+1)\e^n$.
\end{proof}
%
%
\subsection{Uniform Ding stability}
Now consider the log Fano case, that is, $(X,B)$ is klt and
$L:=-K_{(X,B)}$ is ample.
We can then consider stability with respect to the 
non-Archimedean Ding functional $D_B^\NA$ on $\cH^\NA$ defined 
in~\S\ref{sec:Ding}. 

Namely, following~\cite{BBJ15} (see also Fujita~\cite{Fuj15b,Fuj16})
we say that $(X,B)$ is 
\emph{Ding semistable} if $D_B^\NA\ge0$, and \emph{uniformly Ding stable} if
$D_B^\NA\ge\d J^\NA$ for some $\d>0$.

A proof of the following result in the case when $X$ is smooth and $B=0$ appears in~\cite{BBJ15}. 
The general case is treated in~\cite{Fuj16}.
\begin{thm}\label{T101}
  Let $X$ be a normal projective variety and 
  $B$ an effective boundary on $X$ such that $(X,B)$ is klt
  and $L:=-K_{(X,B)}$ is ample. Then, for any $\d\in[0,1]$, we have
  $M^\NA_B\ge\delta J^\NA$ on $\cH^\NA$ iff 
  $D^\NA_B\ge\delta J^\NA$ on $\cH^\NA$. 
  In particular, $((X,B);L)$ is K-semistable 
  (resp.\ uniformly K-stable) 
  iff $(X,B)$ is Ding-semistable 
  (resp.\ uniformly Ding-stable).
\end{thm}
%
%
%
%
\section{Uniform K-stability and singularities of pairs}\label{sec:kstabsing}
In this section, the base field $k$ is assumed to have characteristic
$0$. We still assume $X$ is a normal variety, unless otherwise stated.
%
%
\subsection{Odaka-type results for pairs}\label{S304}
Let $B$ be an effective boundary on $X$. Recall that the pair $(X,B)$ is lc (log canonical) if $A_{(X,B)}(v)\ge 0$ for all divisorial valuations $v$ on $X$, while $(X,B)$ is klt if $A_{(X,B)}(v)>0$ for all such $v$. 
\begin{thm}\label{thm:lc} Let $(X,L)$ be a normal polarized variety, and $B$ an effective boundary on $X$. Then 
$$
(X,B)\text{ lc}\Longleftrightarrow H_B^{\NA}\ge 0\text{ on }\cH^{\NA}(L)
$$ 
and 
$$
((X,B);L)\text{ K-semistable}\Longrightarrow(X,B)\text{ lc}. 
$$
 \end{thm}
The proof of this result, given in~\S\ref{S203},
follows rather closely the line of argument of~\cite{Oda3}. The second implication is also observed in~\cite[Theorem 6.1]{OSu}. The general result of~\cite{Oda3}, dealing with the non-normal case, is discussed in \S\ref{sec:slc}. 

\begin{thm}\label{thm:klt} Let $(X,L)$ be a normal polarized variety and $B$ an effective boundary on $X$. Then the following assertions are equivalent:
\begin{itemize}
\item[(i)] $(X,B)$ is klt; 
\item[(ii)] there exists $\d>0$ such that $H_B^{\NA}\ge\d J^{\NA}$ on $\cH^{\NA}(L)$; 
\item[(iii)] $H_B^{\NA}(\phi)>0$ for every $\phi\in\cH^{\NA}(L)$ that is not a translate of $\phi_\triv$. 
\end{itemize}
\end{thm}
We prove this in~\S\ref{S204}.
The proof of (iii)$\Longrightarrow$(i) is similar to that of~\cite[Theorem 1.3]{Oda3} (which deals with the Fano case), while that of (i)$\Longrightarrow$(ii) relies on an Izumi-type estimate (Theorem~\ref{thm:izumi}). As we shall see, (ii) holds with $\d$ equal to 
the global log canonical threshold of $((X,B);L)$ (cf.~Proposition~\ref{prop:lct} below). 

\smallskip
The above results have the following consequences in the `log
K\"ahler-Einstein case', \ie when $K_{(X,B)}\equiv\la L$ for some
$\la\in\Q$. After scaling $L$, we may assume $\la=0$ or $\la=\pm1$.

First, we have a uniform version of~\cite[Theorem 4.1, (i)]{OSu}. Closely related results were independently obtained in~\cite[\S3.4]{Der1}. 
\begin{cor}\label{cor:canpol} Let $X$ be a normal projective variety and 
  $B$ an effective boundary on $X$ such that $L:=K_{(X,B)}$ is ample. Then 
  the following assertions are equivalent:
\begin{itemize}
\item[(i)] $(X,B)$ is lc;
\item[(ii)] $((X,B);L)$ is uniformly K-stable, with $M_B^{\NA}\ge\frac1{n}J^{\NA}$ on $\cH^{\NA}(L)$; 
\item[(iii)] $((X,B);L)$ is K-semistable.
\end{itemize}
\end{cor}

Next, in the log Calabi-Yau case we get a uniform version of~\cite[Theorem 4.1, (ii)]{OSu}:
\begin{cor}\label{cor:CY} Let $(X,L)$ be normal polarized variety, $B$ an effective boundary on $X$, and assume that $K_{(X,B)}\equiv 0$. Then
$((X,B);L)$ is K-semistable iff $(X,B)$ is lc. Further, the following assertions are equivalent:
\begin{itemize}
\item[(i)] $(X,B)$ is klt;
\item[(ii)] $((X,B);L)$ is uniformly K-stable; 
 \item[(iii)] $((X,B);L)$ is K-stable.  
\end{itemize} 
\end{cor}
\begin{rmk} By~\cite[Corollary~3.3]{Oda1}, there exist polarized K-stable Calabi-Yau orbifolds (which have log terminal singularities) $(X,L)$ that are not asymptotically Chow (or, equivalently, Hilbert) semistable. 
In view of Corollary~\ref{cor:CY}, it follows that uniform K-stability does not 
imply asymptotic Chow stability in general. 
\end{rmk}

Finally, in the log Fano case we obtain:
\begin{cor}\label{cor:Fano}
  Let $X$ be a normal projective variety and 
  $B$ an effective boundary on $X$ such that $L:=-K_{(X,B)}$ is ample. 
  If $((X,B);L)$ is K-semistable, then $H_B^{\NA}\ge\frac1nJ^{\NA}$ on
  $\cH^{\NA}(L)$; in particular, $(X,B)$ is klt.
\end{cor}
A partial result in the reverse direction can be found in Proposition~\ref{prop:alpha}. 
See also~\cite[Theorem 6.1]{OSu} and~\cite[Theorem 3.39]{Der1} for closely related results.
Corollaries~\ref{cor:canpol},~\ref{cor:CY} and~\ref{cor:Fano} are proved in~\S\ref{S205}.
%
%
%
\subsection{Lc and klt blow-ups}
The following result, due to Y.~Odaka and C.~Xu, deals with lc blow-ups. 
The proof is based on an ingenious application of the MMP.
\begin{thm}\label{thm:OX}\cite[Theorem 1.1]{OX} Let $B$ be an
  effective boundary on $X$ with coefficients at most $1$. Then there
  exists a unique projective birational morphism $\mu\colon X'\to X$ such that the strict transform $B'$ of $B$ on $Y$ satisfies:
\begin{itemize}
\item[(i)] the exceptional locus of $\mu$ is a (reduced) divisor $E$; 
\item[(ii)] $(X',E+B')$ is lc and $K_{X'}+E+B'$ is $\mu$-ample. 
\end{itemize}
\end{thm}

\begin{cor}\label{cor:lc} Let $B$ be an effective boundary on $X$, and assume that $(X,B)$ is not lc. Then there exists a closed subscheme $Z\subset X$ whose Rees valuations $v$ all satisfy $A_{(X,B)}(v)<0$. 
\end{cor}
\begin{proof} If $B$ has an irreducible component $F$ with coefficient
  $>1$, then $A_{(X,B)}(\ord_F)<0$, and $Z:=F$ has the desired property, since $\ord_F$ is its unique Rees valuation (cf.~Example~\ref{ex:rees}). 

If not, Theorem~\ref{thm:OX} applies. Denoting by
$A_i:=A_{(X,B)}(\ord_{E_i})$ the log discrepancies of the irreducible component $E_i$ of $E$, we have 
\begin{equation}\label{equ:disc}
K_{X'}+E+B'=\pi^*K_{(X,B)}+\sum_i A_i E_i, 
\end{equation}
which proves that $\sum_i A_i E_i$ is $\mu$-ample, and hence $A_i<0$ by the negativity lemma (or Lemma~\ref{lem:ample}). Proposition~\ref{prop:reesexc} now yields the desired subscheme. 
\end{proof}

We next prove an analogous result for klt pairs, using 
a well-known and easy consequence of the MMP as in~\cite{BCHM}. 

\begin{prop}\label{prop:klt} Let $B$ be an effective boundary, and assume that $(X,B)$ is not klt. Then there exists a closed subscheme $Z\subset X$ whose Rees valuations $v$ all satisfy $A_{(X,B)}(v)\le 0$. 
\end{prop}
\begin{proof} If $B$ has an irreducible component $F$ with coefficient at least $1$, then $A_{(X,B)}(\ord_F)\le 0$, and we may again take $Z=F$. 

Assume now that $B$ has coefficients $<1$. Let $\pi\colon X'\to X$
be a log resolution of $(X,B)$. This means $X'$ is smooth, the
exceptional locus $E$ of $\pi$ is a (reduced) divisor, and $E+B'$ has
snc support, with $B'$ the strict transform of $B$. If we denote by
$A_i:=A_{(X,B)}(\ord_{E_i})$ the log discrepancies of the irreducible
component $E_i$ of $E$, then~\eqref{equ:disc} holds, and hence 
\begin{equation}\label{equ:knum}
  K_{X'}+(1-\e)E+B'=\pi^*(K_X+B)+\sum_i(A_i-\e)E_i
\end{equation}
for any $0<\e<1$. If we pick $\e$ smaller than $\min_{A_i>0}A_i$, then
the $\Q$-divisor $D:=\sum_i(A_i-\e)F_i$ is $\pi$-big (since the generic fiber of $\pi$ is a point), and $\pi$-numerically equivalent to the log canonical divisor of the klt pair $(X',(1-\e)E+B')$ by (\ref{equ:knum}). 

Picking any $m_0\ge 1$ such that $m_0D$ is a Cartier divisor,~\cite[Theorem 1.2]{BCHM} shows that the $\cO_X$-algebra of relative sections 
$$
R(X'/X,m_0D):=\bigoplus_{m\in\N}\mu_*\cO_{X'}(mm_0D)
$$ 
is finitely generated. Its relative $\Proj$ over $X$ yields a
projective birational morphism $\mu\colon Y\to X$, with $Y$ normal,
such that the induced birational map $\phi\colon X'\dashrightarrow Y$ is surjective in codimension one (\ie $\phi^{-1}$ does not contract any divisor) and $\phi_*D=\sum_i(A_i-\e)\phi_*E_i$ is $\mu$-ample. 

Since $D$ is $\mu$-exceptional and $\phi$ is surjective in codimension
$1$, $\phi_*D$ is also $\mu$-exceptional. By Lemma~\ref{lem:ample},
$-\phi_*D$ is effective and its support coincides the exceptional
locus of $\mu$. Hence that the $\mu$-exceptional prime divisors
are exactly the strict transforms of those $E_i$'s with $A_i-\e<0$,
\ie $A_i\le 0$ by the definition of $\e$. As before, we conclude 
using Proposition~\ref{prop:reesexc}. 
\end{proof}
%
%
%
\subsection{Proof of Theorem~\ref{thm:lc}}\label{S203}
If $(X,B)$ is lc, then it is clear from the definition of the
non-Archi\-medean entropy functional that $H^{\NA}_B\ge0$ on $\cH^{\NA}(L)$. 

Now assume that $(X,B)$ is not lc. By Corollary~\ref{cor:lc}, there exists a closed subscheme $Z\subset X$ whose Rees valuations $v$ all satisfy $A_{(X,B)}(v)<0$. 
Corollary~\ref{cor:rees} then yields a normal, ample
test configuration $(\cX,\cL)$ of $(X,L)$ such that 
\begin{equation*}
  \left\{v_E\mid E\ \text{a nontrivial irreducible component of }\cX_0\right\}
\end{equation*}
coincides with the (nonempty) set of Rees valuations of $Z$. Thus 
$A_X(v_E)<0$ for all nontrivial irreducible components $E$ of $\cX_0$. 

Denote by $\phi\in\cH^{\NA}$ the non-Archimedean metric defined
by $(\cX,\cL)$. We directly get $H_B^{\NA}(\phi)<0$, so
$H_B^\NA\not\ge0$ on $\cH^\NA$. 
Further, Proposition~\ref{prop:neartriv} implies that the positive metric
$\phi_\e:=\e\phi+(1-\e)\phi_\triv$ satisfies $M_B^{\NA}(\phi_\e)<0$
for $0<\e\ll 1$. Hence $((X,B);L)$ cannot be K-semistable. 
This completes the proof.

\begin{defi}\label{defi:c_E} 
  Let $(\cX,\cL)$ be a normal, semiample test configuration for
  $(X,L)$ representing a positive metric $\phi\in\cH^{\NA}(L)$. 
  For each irreducible component $E$ of
  $\cX_0$,  let $Z_E\subset X$ be the closure of the center of
  $v_E$ on $X$, and set $r_E:=\codim_X Z_E$.   
  Then the canonical birational map $\cX\dashrightarrow X_{\A^1}$
  maps $E$ onto $Z_E\times\{0\}$.
  Let $F_E$ be the generic fiber of the induced map $E\dashrightarrow Z_E$,  
  and define the \emph{local degree} $\deg_E(\phi)$ of $\phi$ at $E$ as
  $$
  \deg_E(\phi):=(F_E\cdot\cL^{r_E}).  
  $$
\end{defi} 
Since $\cL$ is semiample on $E\subset\cX_0$, we have $\deg_E(\phi)\ge 0$, and $\deg_E(\phi)>0$ iff $E$ is not contracted on the ample model of $(\cX,\cL)$. The significance of these invariants is illustrated by the following estimate, whose proof is straightforward. 
\begin{lem}\label{lem:estim} 
With the above notation, assume that $\cX$ dominates $X_{\A^1}$ via 
$\rho\colon\cX\to X_{\A^1}$.
Given $0\le j\le n$ and line bundles $M_1,\dots,M_{n-j}$ on $X$, we have,
for $0<\e\ll1$ rational:
\begin{multline*}
  \left(E\cdot\left(\rho^*L_{\A^1}+\e D\right)^j
    \cdot\rho^*\left(M_{1,\A^1}\cdot\ldots\cdot M_{n-j,\A^1}\right)\right)\\
  =\begin{cases}
    \e^{r_E}\left[\deg_E(\phi)\binom{j}{r_E}\left(Z_E\cdot L^{j-r_E}\cdot M_1\cdot\ldots\cdot M_{n-j}\right)\right]+O(\e^{r_E+1})
    &\text{for $j\ge r_E$}\\
     0&\text{for $j<r_E$}.
  \end{cases}
\end{multline*}
\end{lem}

\begin{prop}\label{prop:neartriv} 
  Pick $\phi\in\cH^{\NA}(L)$ that is not a translate of $\phi_\triv$, 
  and let $(\cX,\cL)$ be its unique normal ample representative. 
  Set $r:=\min_E r_E$, with $r_E=\codim_X Z_E$ and $E$ running 
  over all non-trivial irreducible components of the ample model $(\cX,\cL)$ 
  of $\phi$ (and hence $r\ge 1$). 

  Let further $B$ be a boundary on $X$. 
  Then $\phi_\e:=\e\phi+(1-\e)\phi_\triv$ satisfies 
  \begin{equation*}
    J^{\NA}(\phi_\e)=O(\e^{r+1}),
    \quad
    R_B^{\NA}(\phi_\e)=O(\e^{r+1}),
  \end{equation*}  
  and 
  \begin{align*}
    M_B^{\NA}(\phi_\e)
    &=H_B^{\NA}(\phi_\e)+O(\e^{r+1})\\
    &=\e^r\left[V^{-1}\sum_{r_E=r}\deg_E(\phi)b_E
      \left(Z_E\cdot L^{n-r}\right)A_{(X,B)}(v_E)\right]+O(\e^{r+1}).
  \end{align*}
\end{prop}
\begin{proof} 
  Let $(\cX',\cL')$ be a normal test configuration dominating 
  $(\cX,\cL)$ and $(X_{\A^1},L_{\A^1})$. 
  Write $\cL'=\rho^*L_{\A^1}+D$, where $\rho\colon\cX'\to X_{\A^1}$ is
  the morphism. Note that $(\cX',\cL'_\e)$, with 
  $\cL'_\e=\rho^*L_{\A^1}+\e D$, is a representative of $\phi_\e$. 
  By translation invariance of $J^{\NA}$ and $M^{\NA}$, 
  we may assume $(\phi\cdot\phi_\triv^n)=0$, \ie $\ord_{E_0}(D)=0$ 
  for the strict transform $E_0$ of $X\times\{0\}$ to $\cX'$, by Theorem~\ref{thm:supp}. 
  Then $(\phi_\e\cdot\phi_\triv^n)=0$, and hence $J^{\NA}(\phi_\e)=-E^{\NA}(\phi_\e)$. 
  Lemma~\ref{lem:EMA} yields
  \begin{equation*}
    (n+1)VE^{\NA}(\phi_\e)=\sum_{j=0}^n\left(\e D\cdot\left(\rho^*L_{\A^1}
        +\e D\right)^j\cdot\rho^*L_{\A^1}^{n-j}\right). 
  \end{equation*}
  Since we have normalized $D$ by $\ord_{E_0}(D)=0$, Lemma~\ref{lem:estim}
  implies $E^{\NA}(\phi_\e)=O(\e^{r+1})$,
  and hence $J^{\NA}(\phi_\e)=O(\e^{r+1})$.

  Similarly,
  \begin{multline*}
    V R^{\NA}_B(\phi_\e)
    =\left(\rho^*K^\lo_{(X_{\P^1},B_{\P^1})/\P^1}\cdot(\bar\cL'_\e)^n\right)\\
    =\left(\rho^*K^\lo_{(X_{\P^1},B_{\P^1})/\P^1}\cdot(\bar\cL'_\e)^n\right)
    -\left(\rho^*K^\lo_{(X_{\P^1},B_{\P^1})/\P^1}\cdot\rho^*L_{\P^1}^n\right)\\
    =\sum_{j=0}^{n-1}\left(\e D\cdot\left(\rho^*L_{\A^1}
        +\e D\right)^j\cdot\rho^*L_{\A^1}^{n-j-1}\cdot\rho^*K^\lo_{(X_{\P^1},B_{\P^1})/\P^1}\right)
    =O(\e^{r+1}). 
  \end{multline*}
  The expression for $M^{\NA}_B$ now follows from the
  Chen-Tian formula (see Proposition~\ref{prop:DFChen}) 
  and Lemma~\ref{lem:estim} applied to 
  \begin{equation*}
    H_B^{\NA}(\phi_\e)
    =V^{-1}\sum_EA_{(X,B)}(v_E)b_E\left(E\cdot\left(\rho^*L_{\A^1}+\e D\right)^n\right)
  \end{equation*}
  where $E$ runs over the non-trivial irreducible components of $\cX'_0$. 
\end{proof}
%
%
%
%
\subsection{The non-normal case}\label{sec:slc}
In this section we briefly sketch the proof of the following general result, due to Odaka~\cite[Theorem 1.2]{Oda3}. 

\begin{thm}\label{thm:scl} Let $X$ be deminormal scheme with $K_X$ $\Q$-Cartier. Let $L$ be an ample line bundle on $X$, and assume that $(X,L)$ is K-semistable. Then $X$ is slc. 
\end{thm}

Recall from Remark~\ref{rmk:slc} that $(X,L)$ is K-semistable iff
$\DF_{\tB}(\tcX,\tcL)\ge 0$ for all ample test configurations
$(\cX,\cL)$ for $(X,L)$. Here $(\tX,\tL)$ and $(\tcX,\tcL)$ denote the
normalizations of $(X,L)$ and $(\cX,\cL)$, respectively, and $\tB$ is the conductor, viewed as a reduced Weil divisor on $\tX$. On the other hand, $X$ is slc iff $(\tX,\tB)$ is lc, by definition. 

Assuming that $X$ is not slc, \ie $(\tX,\tB)$ not lc, our goal is thus to produce an ample test configuration $(\cX,\cL)$ for $(X,L)$ such that $\DF_{\tB}(\tcX,\tcL)<0$. By Theorem~\ref{thm:OX}, the non-lc pair $(\tX,\tB)$ admits an lc blow-up $\mu\colon\tX'\to\tX$. As explained on~\cite[p.332]{OX}, Koll\'ar's gluing theorem implies that $\tX'$ is the normalization of a reduced scheme $X'$, with a morphism $X'\to X$. As a consequence, we can find a closed subscheme $Z\subset X$ whose inverse image $\widetilde{Z}\subset\tX$ is such that $A_{(\tX,\tB)}(v)<0$ for each Rees valuation $v$ of $\widetilde{Z}$. 

Let $\mu\colon\cX\to X\times\A^1$ be the deformation to the normal cone
of $Z$, with exceptional divisor $E$, and set $\cL_\e=\mu^*L_{\A^1}-\e
E$ with $0<\e\ll 1$. Since the normalization $\tcX$ of $\cX$ is also
the normalization of the deformation to the normal cone of
$\widetilde{Z}$, we have $A_{(\tX,\tB)}(v_E)<0$ for each irreducible component $E$ of $\tcX_0$, and Proposition~\ref{prop:neartriv} gives, as desired, $\DF_{\tB}(\tcX,\tcL_\e)<0$ for $0<\e\ll 1$. 
%
%
%
%
\subsection{The global log canonical threshold and proof of Theorem~\ref{thm:klt}}\label{S204} 
Recall from~\S\ref{sec:bound} the definition of the 
log canonical threshold of an effective $\Q$-Cartier divisor $D$ 
with respect to a subklt pair $(X,B)$:
\begin{equation*}
  \lct_{(X,B)}(D)=\inf_v\frac{A_{(X,B)}(v)}{v(D)}.
\end{equation*}
Similarly, given an ideal $\fa$ and $c\in\Q_+$, we set
\begin{equation*}
  \lct_{(X,B)}(\fa^c):=\inf_v\frac{A_{(X,B)}(v)}{v(\fa^c)},
\end{equation*}
with $v(\fa^c):=c v(\fa)$.

The main ingredient in the proof of (i)$\Longrightarrow$(ii) of
Theorem~\ref{thm:klt} is the following result.
\begin{thm}\label{thm:izumi} If $((X,B);L)$ is a polarized subklt pair, then
  \begin{equation}\label{equ:infima}
    \inf_D\lct_{(X,B)}(D)=\inf_{\fa,c}\lct_{(X,B)}(\fa^c),
\end{equation}
where the left-hand infimum is taken over all effective $\Q$-Cartier
divisors $D$ on $X$ that are $\Q$-linearly equivalent to $L$, and the right-hand one is over all non-zero ideals $\fa\subset\cO_X$ and all $c\in\Q_+$ such that $L\otimes\fa^c$ is nef. 
Further, these two infima are strictly positive.
\end{thm}
Here we say that $L\otimes\fa^c$ is nef if $\mu^*L-cE$ is nef on the
normalized blow-up $\mu\colon X'\to X$ of $\fa$, with $E$ the effective Cartier divisor such that $\fa\cdot\cO_{X'}=\cO_{X'}(-E)$. 

\begin{defi}\label{defi:glct} The \emph{global log canonical threshold} $\lct((X,B);L)$ of a polarized subklt pair $((X,B);L)$ is the common value of the two infima in Theorem~\ref{thm:izumi}.
\end{defi}

\begin{proof}[Proof of Theorem~\ref{thm:izumi}] Let us first prove
  that the two infima coincide. Let $D$ be an effective $\Q$-Cartier
  divisor $\Q$-linearly equivalent to $L$. Pick $m\ge 1$ such that
  $mD$ is Cartier, and set $\fa:=\cO_X(-mD)$ and $c:=1/m$. Then
  $v(\fa^c)=v(D)$ for all $v$, and $L\otimes\fa^c$ is nef since
  $L-cmD$ is even numerically trivial. 
  Hence $\inf\lct_{(X,B)}(D)\le\inf\lct_{(X,B)}(\fa^c)$. 

  Conversely, assume that $L\otimes\fa^c$ is nef. Let $\mu\colon X'\to X$ be the normalized blow-up of $X$ along $\fa$ and $E$ the effective Cartier divisor on $X'$ such that $\cO_{X'}(-E)=\fa\cdot\cO_{X'}$, so that $\mu^*L-c E$ is nef. Since $-E$ is $\mu$-ample, we can find $0<c'\ll 1$ such that $\mu^*L-c' E$ is ample. Setting $c_\e:=(1-\e)c+\e c'$, we then have $\mu^*L-c_\e E$ is ample for all $0<\e<1$. 

Let also $B'$ the unique $\Q$-Weil divisor on $X'$ such that
$\mu^*K_{(X,B)}=K_{(X',B')}$ and $\mu_*B'=B$, so that $(X',B')$ is a pair with $A_{(X,B)}=A_{(X',B')}$.  

If we choose a log resolution $\pi\colon X''\to X'$ of $(X',B'+E)$ and let $F=\sum_i F_i$ be the sum of all $\pi$-exceptional primes and of the strict transform of $B'_\red+E_\red$, then 
$$
\lct_{(X,B)}(\fa^{c_\e})=\lct_{(X',B')}(c_\e E)=\min_i\frac{A_{(X',B')}(\ord_{F_i})}{\ord_{F_i}(c_\e D)}
$$
Given $0<\e<1$, pick $m\gg 1$ such that 
\begin{itemize}
\item[(i)] $mc_\e\in\N$;
\item[(ii)] $m(\mu^*L-c_\e E)$ is very ample;
\item[(iii)] $m\ge\lct_{(X,B)}(\fa^{c_\e})$. 
\end{itemize}
Let $H\in |m(\mu^*L-c_\e E)|$ be a general element, and set
$D:=\mu_*(c_\e E+m^{-1}H)$, so that $D$ is $\Q$-Cartier, 
$\Q$-linearly equivalent to $L$,
and $\mu^*D=c_\e E+m^{-1}H$. 

By Bertini's theorem, $\pi$ is also a log resolution of $(X',B'+E+H)$, and hence
\begin{equation*}
  \lct_{(X,B)}(D)
  =\lct_{(X',B')}(c_\e E+m^{-1}H)
  =\min\left\{\frac{A_{(X',B')}(v)}{v(c_\e E+m^{-1}H)}\ \bigg|\
    v=v_i\ \text{or $v=\ord_H$}\right\}.
\end{equation*}
But $H$, being general, does not contain the center of $\ord_{D_i}$ on $X'$ and is not contained in $\supp E$, \ie $\ord_{D_i}(H)=0$ and $\ord_H(E)=0$, and (iii) above shows that
\begin{equation*}
  \lct_{(X,B)}(D)=\min\left\{\lct_{(X,B)}(\fa^{c_\e}),m\right\}=\lct_{(X,B)}(\fa^{c_\e}).
\end{equation*}

Since we have $\lct_{(X,B)}(\fa^{c_\e})=\frac{c}{c_\e}\lct_{(X,B)}(\fa^c)$ with $c_\e/c$ arbitrarily close to $1$, we conclude that the two infima in (\ref{equ:infima}) are indeed equal. 

\medskip
We next show that the left-hand infimum in (\ref{equ:infima}) is strictly positive, 
in two steps.

\smallskip
\noindent{\bf Step 1}. We first treat the case where $X$ is smooth and $B=0$. By Skoda's theorem (see for instance~\cite[Proposition 5.10]{JM}), we then have
$$
v(D)\le\ord_p(D) A_X(v)
$$
for every effective $\Q$-Cartier divisor $D$ on $X$, every divisorial valuation $v$, and every closed point $p$ in the closure of the center of $v$ on $X$. It is thus enough to show that $\ord_p(D)$ is uniformly bounded when $D\sim_\Q L$. 

Let $\mu\colon X'\to X$ be the blow-up at $p$, with exceptional divisor $E$. Since $L$ is ample, there exists $\e>0$ independent of $p$ such that $L_\e:=\mu^*L-\e E$ is ample, by Seshadri's theorem. 

Since $D$ is effective, we have $\mu^*D\ge\ord_p(D)E$, and hence
$$
(L^n)=\left(\mu^*L\cdot L_\e^{n-1}\right)\ge\ord_p(D)(E\cdot L_\e^{n-1})=\e^{n-1}\ord_p(D), 
$$
which yields the desired bound on $\ord_p(D)$. 

\noindent{\bf Step 2}. Suppose now that $(X,B)$ is a subklt
pair. Pick a log resolution $\mu\colon X'\to X$, and let $B'$ be the unique $\Q$-divisor such that $\mu^*K_{(X,B)}=K_{(X',B')}$ and $\mu_*B'=B$, so that
$$
A_{(X,B)}(v)=A_{(X',B')}(v)=A_{X'}(v)-v(B')
$$
for all divisorial valuations $v$. Since $(X,B)$ is subklt, $B'$ has
coefficients less than $1$, so there exists $0<\e\ll 1$ such that $B'\le(1-\e)B'_\red$. Since $B'_\red$ is a reduced snc divisor, the pair $(X',B'_\red)$ is lc, and hence $v(B')\le A_{X'}(v)$ for all divisorial valuations $v$. It follows that $v(B)\le(1-\e)A_{X'}(v)$, \ie
$$
\e A_{X'}(v)\le A_{(X,B)}(v)
$$
for all $v$. Pick any very ample effective divisor $H$ on $X'$ such
that $L':=\mu^*L+H$ is ample. For each effective $\Q$-Cartier divisor
$D\sim_\Q L$, $D':=\mu^*D+H$ is an effective $\Q$-Cartier divisor on
$X'$ with $D'\sim_\Q L'$. By Step 1, we conclude that 
$$
v(D)\le v(D')\le C A_{X'}(v)\le C\e^{-1}A_{(X,B)}(v),
$$
which completes the proof.
\end{proof}

\begin{prop}\label{prop:lct} For each polarized subklt pair $((X,B);L)$, we have 
$$
H^{\NA}\ge\d I^{\NA}\ge\frac{\d}{n}J^{\NA}
$$ 
on $\cH^{\NA}$ with $\d:=\lct((X,B);L)>0$. 
\end{prop}
\begin{proof} Pick $\phi\in\cH^{\NA}$, and let $(\cX,\cL)$ be a normal
  representative such that $\cX$ dominates $X_{\A^1}$ via 
  $\rho\colon\cX\to X_{\A^1}$, and write $\cL=\rho^*L_{\A^1}+D$. 

Choose $m\ge 1$ such that $m\cL$ is a globally generated line bundle, and let 
$$
\rho_*\cO_\cX(mD)=\fa^{(m)}=\sum_{\la\in\Z}\fa^{(m)}_\la t^{-\la}
$$
be the corresponding flag ideal. By Proposition~\ref{prop:filtrflag},
$\cO_X(mL)\otimes\fa^{(m)}_\la$ is globally generated on $X$ for all $\la\in\Z$. In particular, $L\otimes(\fa^{(m)}_\la)^{1/m}$ is nef, and hence
$$
v(\fa^{(m)}_\la)\le m\d^{-1}A_{(X,B)}(v)
$$
whenever $\fa^{(m)}_\la$ is non-zero. 

Now let $E$ be a non-trivial irreducible component of $\cX_0$. By Lemma~\ref{lem:div2}, we have 
$$
\ord_E(\fa^{(m)})=\min_\la\left(v_E(\fa^{(m)}_\la)-\la b_E\right) 
$$
with $b_E=\ord_E(\cX_0)$, and hence 
$$
\ord_E(\fa^{(m)})\le m \d^{-1} A_{(X,B)}(v_E)-b_E\max\left\{\la\in\Z\mid\fa^{(m)}_\la\ne 0\right\}
$$
By Proposition~\ref{prop:filtrflag}, we have
$$
\max\left\{\la\in\Z\mid\fa^{(m)}_\la\ne 0\right\}=\la_{\max}^{(m)}, 
$$
which is bounded above by 
$$
m\la_{\max}=m(\phi\cdot\phi_\triv^n),
$$ 
by Lemma~\ref{lem:sup}. We have thus proved that
\begin{equation}\label{equ:efa}
m^{-1}\ord_E(\fa^{(m)})\le \d^{-1} A_{(X,B)}(v_E)-b_E V^{-1}(\phi\cdot\phi_\triv^n).  
\end{equation}
But since $mD$ is $\rho$-globally generated, we have $\cO_\cX(mD)=\cO_{\cX}\cdot\fa^{(m)}$, and hence 
$$
m^{-1}\ord_E(\fa^{(m)})=-\ord_E(D).
$$
Using (\ref{equ:efa}) and $\sum_E b_E(E\cdot\cL^n)=(\cX_0\cdot\cL^n)=V$, we infer
$$
-V^{-1}((\phi-\phi_\triv)\cdot\phi^n)=-V^{-1}\left(D\cdot\cL^n\right)\le \d^{-1} H^{\NA}(\phi)-V^{-1}(\phi\cdot\phi_\triv^n)
$$
and the result follows by the definition of $I^{\NA}$ and by Proposition~\ref{prop:J}. 
\end{proof}

\begin{proof}[Proof of Theorem~\ref{thm:klt}] 
  The implication (i)$\Longrightarrow$(ii) follows from
  Proposition~\ref{prop:lct}, and (ii)$\Longrightarrow$(iii) is
  trivial. Now assume that (iii) holds. If $(X,B)$ is not klt,
  Proposition~\ref{prop:klt} yields a closed subscheme $Z\subset X$
  with $A_{(X,B)}(v)\le 0$ for all Rees valuations $v$ of $Z$. By
  Corollary~\ref{cor:rees}, we can thus find a normal, ample test
  configuration $(\cX,\cL)$ such that $A_{(X,B)}(v_E)\le 0$ for each
  non-trivial irreducible component $E$ of $\cX_0$. The corresponding
  non-Archimedean metric $\phi\in\cH^{\NA}$  therefore satisfies $H^{\NA}_B(\phi)\le 0$, which contradicts (iii). 
\end{proof}
%
%
%
%
\subsection{The K\"ahler-Einstein case}\label{S205}
\begin{proof}[Proof of Corollary~\ref{cor:canpol}] 
The implication (iii)$\Longrightarrow$(i) follows from Theorem~\ref{thm:lc}, and (ii)$\Longrightarrow$(iii) is trivial.  Now assume (i), so that $H_B^{\NA}\ge 0$ on $\cH^{\NA}$ by Theorem~\ref{thm:lc}. By Lemma~\ref{lem:MKE}, we have $M_B^{\NA}=H_B^{\NA}+\left(I^{\NA}-J^{\NA}\right)$, while $I^{\NA}-J^{\NA}\ge\frac{1}{n}J^{\NA}$ by Proposition~\ref{prop:J}. We thus get $M_B^{\NA}\ge\frac1n J^{\NA}$, which proves (iii).  
\end{proof}

\begin{proof}[Proof of Corollary~\ref{cor:CY}] If $K_{(X,B)}$ is numerically trivial, then Lemma~\ref{lem:MKE} gives $M_B^{\NA}=H_B^{\NA}$. The result is thus a direct consequence of Theorem~\ref{thm:lc} and Theorem~\ref{thm:klt}.
\end{proof}

\begin{proof}[Proof of Corollary~\ref{cor:Fano}] 
  Lemma~\ref{lem:MKE} yields
  $M_B^{\NA}=H_B^{\NA}-(I^{\NA}-J^{\NA})$.
  The K-semistability of $(X,B)$ thus means that $H_B^{\NA}\ge I^{\NA}-J^{\NA}$, 
  and hence $H_B^{\NA}\ge\frac1nJ^{\NA}$ by Proposition~\ref{prop:J}. By Theorem~\ref{thm:klt}, this implies that $(X,B)$ is klt. 
\end{proof}

The following result gives a slightly more precise version of the computations of~\cite[Theorem 1.4]{OSa} and~\cite[Theorem 3.24]{Der1}. 
\begin{prop}\label{prop:alpha} 
  Let $B$ be an effective boundary on $X$ such that $(X,B)$ is klt and 
  $L:=-K_{(X,B)}$ is ample.
  Assume also that $\e:=\lct((X,B);L)-\frac{n}{n+1}>0$. Then we have 
  \begin{equation*}
    M_B^{\NA}\ge\e I^{\NA}\ge\frac{n+1}{n}\e J^{\NA}.
  \end{equation*}
  In particular, the polarized pair $((X,B);L)$ is uniformly K-stable.  
\end{prop}

\begin{proof} By Proposition~\ref{prop:lct} we have $H^{\NA}\ge\left(\frac{n}{n+1}+\e\right)I^{\NA}$, and hence
$$
M_B^{\NA}\ge\e I^{\NA}+\left(J^{\NA}-\frac{1}{n+1}I^{\NA}\right). 
$$
The result follows since we have
$$
\frac{1}{n+1} I^{\NA}\le J^{\NA}\le\frac{n}{n+1}I^{\NA}
$$
by Proposition~\ref{prop:J}.
\end{proof}
%
%
%
%
\appendix
\section{Asymptotic Riemann-Roch on a normal variety}\label{sec:appA}
%
%
The following result is of course well-known, but we provide a proof
for lack of suitable reference. In particular, the sketch provided
in~\cite[Lemma 3.5]{Oda2} assumes that the line bundle in question is ample, which is not enough for the application to the intersection theoretic formula for the Donaldson-Futaki invariant (cf.~(iv) in Proposition~\ref{prop:DHDFsemi}). 

\begin{thm}\label{thm:RR} If $Z$ is a proper normal variety of
  dimension $d$, defined over an
  algebraically closed field $k$, and $L$ is a line bundle on $Z$, then 
  $$
  \chi(Z,mL)=(L^d)\frac{m^d}{d!}-(K_Z\cdot L^{d-1})\frac{m^{d-1}}{2(d-1)!}+O(m^{d-2}). 
  $$
\end{thm}

\begin{proof}[A proof in characteristic $0$] When $Z$ is smooth, the result follows from the Riemann-Roch formula, which reads
$$
\chi(Z,mL)=\int\left(1+c_1(mL)+\dots+\frac{c_1(mL)^d}{d!}\right)\left(1+\frac{c_1(Z)}{2}+\dots\right). 
$$
Assume now that $Z$ is normal, pick a resolution of singularities
$\mu\colon Z'\to Z$ and set $L':=\mu^*L$. The Leray spectral sequence and the projection formula imply that 
$$
\chi(Z',mL')=\sum_j(-1)^j\chi\left(Z,\cO_Z(mL)\otimes R^j\mu_*\cO_{Z'}\right). 
$$
Since $Z$ is normal, $\mu$ is an isomorphism outside a set of codimension at least $2$. As a result, for each $j\ge 1$ the support of the coherent sheaf $R^j\mu_*\cO_{Z'}$ has codimension at least $2$, and hence $\chi\left(Z,\cO_Z(mL)\otimes R^j\mu_*\cO_{Z'}\right)=O(m^{d-2})$ (cf.~\cite[\S1]{Kle}). We thus get
$$
\chi(Z,mL)=\chi(Z',mL')+O(m^{d-2})
=(L'^d)\frac{m^d}{d!}-(K_{Z'}\cdot L'^{d-1})\frac{m^{d-1}}{2(d-1)!}+O(m^{d-2}),
$$
and the projection formula yields the desired result, since $\mu_*K_{Z'}=K_Z$ as cycle classes.
\end{proof}

\begin{proof}[The general case] By Chow's lemma, there exists a
  birational morphism $Z'\to Z$ with $Z'$ projective and normal. By
  the same argument as above, it is enough to prove the result for
  $Z'$, and we may thus assume that $Z$ is projective to begin
  with. 

  We argue by induction on $d$. The case $d=0$ is clear, so 
  assume $d\ge 1$ and let 
  $H$ be a very ample line bundle on $Z$ such that $L+H$ is also very ample. By the Bertini type theorem for normality of~\cite[Satz 5.2]{Fle}, general elements $B\in|H|$ and $A\in|L+H|$ are also normal, with $L=A-B$. The short exact sequence
$$
0\to\cO_Z((m+1)L-A)\to\cO_Z(mL)\to\cO_B(mL)\to 0
$$
shows that 
$$
\chi(Z,(m+1)L-A)=\chi(Z,mL)-\chi(B,mL).
$$
We similarly find
$$
\chi(Z,(m+1)L)=\chi(Z,(m+1)L-A)+\chi(A,(m+1)L),
$$
and hence 
$$
\chi(Z,(m+1)L)-\chi(Z,mL)=\chi(A,(m+1)L)-\chi(B,mL).
$$
Since $A$ and $B$ are normal Cartier divisors on $X$, the adjunction
formulae $K_A=(K_Z+A)|_A$ and $K_B=(K_Z+B)|_B$ hold, as they
are equalities between Weil divisor classes on a normal variety that
hold outside a closed subset of codimension at least $2$. By the
induction hypothesis, we thus get
\begin{multline*}
  \chi(Z,(m+1)L)-\chi(Z,mL)\\
  =(L^{d-1}\cdot A)\left(\frac{m^{d-1}}{(d-1)!}
    +\frac{m^{d-2}}{(d-2)!}\right)-\left((K_Z+A)\cdot A\cdot
    L^{d-2}\right)\frac{m^{d-2}}{2(d-2)!}\\
  -(L^{d-1}\cdot B)\frac{m^{d-1}}{(d-1)!}
  +\left((K_Z+B)\cdot B\cdot L^{d-2}\right)\frac{m^{d-2}}{2(d-2)!}+O(m^{d-3})\\
  =(L^d)\frac{m^{d-1}}{(d-1)!}+\left[(L^d)-\tfrac 1 2(K_Z\cdot
    L^{d-1})\right]\frac{m^{d-2}}{(d-2)!}
  +O(m^{d-3})\\
  =P(m+1)-P(m)+O(m^{d-3}),
\end{multline*}
with 
$$
P(m):=(L^d)\frac{m^d}{d!}-(K_Z\cdot L^{d-1})\frac{m^{d-1}}{2(d-1)!}. 
$$
The result follows. 
\end{proof}
%
%
\section{The equivariant Riemann-Roch theorem for schemes}
%
%
We summarize the general equivariant Riemann-Roch theorem for schemes, which extends to the equivariant setting the results of~\cite[Chap. 18]{Ful}, and is due to Edidin-Graham~\cite{EG1,EG2}. We then use the case $G=\G_m$ to provide an alternative proof of Theorem~\ref{thm:equivRR}.

Let $G$ be a linear algebraic group, and $X$ be a scheme with a
$G$-action. The Grothendieck group $K^0_G(X)$ of virtual
$G$-linearized vector bundles forms a commutative ring with respect to
tensor products, and is functorial under pull-back by $G$-equivariant
morphisms. On the other hand, the Grothendieck group $K^G_0(X)$ of
virtual $G$-linearized coherent sheaves on $X$ is a $K^0_G(X)$-module
with respect to tensor products, and every proper $G$-equivariant
morphism $f\colon X\to Y$ induces a push-forward homomorphism
$f_!\colon K^G_0(X)\to K^G_0(Y)$ defined by
$$
f_![\cF]\=\sum_{q\in\N}(-1)^q[R^qf_*\cF].
$$
Note that $K^0_G(\Spec k)=K_0^G(\Spec k)$ can be identified with the representation ring $R(G)$, so that all the above abelian groups are in particular $R(G)$-modules. 

\smallskip

Equivariant Chow homology and cohomology groups are constructed in~\cite{EG1}, building on an idea of Totaro. The $G$-equivariant Chow cohomology ring
$$
\CH^\bullet_G(X)=\bigoplus_{d\in\N}\CH^d_G(X)
$$ 
can have $\CH^d_G(X)\ne 0$ for infinitely many $d\in\N$, and we set
$$
\hCH^\bullet_G(X)=\prod_{d\in\N}\CH^d_G(X).
$$

The $G$-equivariant first Chern class defines a morphism $c_1^G\colon\Pic^G(X)\to\CH^1_G(X)$, which is an isomorphism when $X$ is smooth~\cite[Corollary 1]{EG1}. In particular, we have natural isomorphisms
$$
\Hom(G,\G_m)\simeq\Pic^G(\Spec k)\simeq\CH^1_G(\Spec k). 
$$
The $G$-equivariant Chern character is a ring homomorphism
$$
\ch^G\colon K^0_G(X)\to\hCH^\bullet_G(X)_\Q, 
$$
functorial with respect to pull-back and such that 
$$
\ch^G(L)=e^{c_1^G(L)}=\left(\frac{c_1^G(L)^d}{d!}\right)_{d\in\N}
$$ 
for a $G$-linearized line bundle $L$. 

\smallskip

On the other hand, the $G$-equivariant Chow homology group 
$$
\CH^G_\bullet(X)=\bigoplus_{p\in\Z}\CH^G_p(X) 
$$
is a $\CH^\bullet_G(X)$-module, with $\CH^d_G(X)\cdot\CH_p^G(X)\subset\CH_{p-d}(X)$. While $\CH_p^G(X)=0$ for $p>\dim X$, it is in general non-zero for infinitely many (negative) $p$ in general, and we set again
$$
\hCH^G_\bullet(X)=\prod_{p\in\Z}\CH^G_p(X), 
$$
a $\hCH^\bullet_G(X)$-module. 

By definition, we have an isomorphism 
$$
\CH_{\dim X}^G(X)\simeq\CH_{\dim X}(X)=\bigoplus_i\Z[X_i]
$$
with $X_i$ the top-dimensional irreducible components of $X$. When $X$ is smooth and pure dimensional, the action of $\CH^d_G(X)$ on the equivariant fundamental class $[X]_G\in\CH^G_{\dim X}(X)$ defines a `Poincar\'e duality' isomorphism
$$
\CH^d_G(X)\simeq\CH_{\dim X-d}^G(X). 
$$  
Via the Chern character, both $K^G_0(X)$ and $\hCH^G_\bullet(X)_\Q$ become $K^0_G(X)$-modules, and the general Riemann-Roch theorem of~\cite[Theorem 3.1]{EG2} constructs a $K^0_G(X)$-module homomorphism
$$
\tau^G\colon K_0^G(X)\to\hCH_\bullet^G(X)_\Q=\prod_{p\le\dim X}\CH_p^G(X)_\Q, 
$$
functorial with respect to push-forward under proper equivariant morphisms, and normalized by $\tau(1)=1$ on $K^0_G(\Spec k)$, so that $\tau^G=\ch^G$ on $R(G)$. The \emph{equivariant Todd class} of $X$ is defined 
$\Td^G(X):=\tau^G(\cO_X)$, with top-dimensional part $\Td^G(X)_{\dim X}=[X]_G\in\CH_{\dim X}^G(X)_\Q$. 

When $X$ is proper, the \emph{equivariant Euler characteristic} of a $G$-linearized coherent sheaf $\cF$ on $X$ is defined as
$$
\chi^G(X,\cF):=\ch^G(\pi_![\cF])\in\hCH^\bullet_G(\Spec k)_\Q\simeq\prod_{d\in\N}\CH^G_{-d}(\Spec k)_\Q
$$
with $\pi\colon X\to\Spec k$ the structure morphism. Equivalently, it can be viewed  The functionality of $\tau^G$ with respect to push-forward by $\pi$ then yields the \emph{equivariant Riemann-Roch formula}, which reads
\begin{equation}\label{equ:RRequiv}
\chi^G(X,E)=\pi_*\left(\ch^G(E)\cdot\Td^G(X)\right)
\end{equation}
for every $G$-linearized vector bundle $E$ on $X$.

\begin{proof}[Alternative proof of Theorem~\ref{thm:equivRR}] Let $(X,L)$ be a polarized scheme with a $\G_m$-action. The argument consists in unraveling the above general results when $G=\G_m$. By~\cite[Lemma 2]{EG}, we have a canonical identification
$$
\CH_{-d}^{\G_m}(\Spec k)\simeq\Z 
$$ 
for all $d\in\N$, with respect to which the equivariant Chern character
$$
\ch^{\G_m}\colon R(\G_m)\to\prod_{d\in\N}\CH_{-d}^{\G_m}(\Spec k)_\Q\simeq\Q^\N
$$
sends a $\G_m$-module $V=\bigoplus_{\la\in\Z} V_\la$ to the sequence $\left(\sum_{\la\in\Z}\frac{\la^d}{d!}\dim V_\la\right)_{d\in\N}$. 

Since $H^q(X,mL)=0$ for $q>0$ and $m\gg 1$, the equivariant Euler characteristic of a $\G_m$-linearized coherent sheaf $\cF$ on $X$ is given by
\begin{equation}\label{equ:RR}
\chi^{\G_m}(X,\cF)=\left(\sum_{\la\in\Z}\frac{\la^d}{d!}\dim H^0(X,mL)_\la\right)_{d\in\N}, 
\end{equation}
and the equivariant Riemann-Roch formula (\ref{equ:RRequiv}) therefore shows that
$$
\sum_{\la\in\Z}\frac{\la^d}{d!}\dim H^0(X,mL)_\la=\left(\pi_*\left(e^{mc_1^{\G_m}(L)}\cdot\Td^{\G_m}(X)\right)\right)_{-d}
$$
in $\CH_{-d}^{\G_m}(\Spec k)_\Q\simeq\Q$, with $\pi\colon X\to\Spec k$ is the structure morphism. Since the top-dimensional part of $\Td^{\G_m}(X)$ is the equivariant fundamental cycle $[X]_{\G_m}\in\CH_n^{\G_m}(X)$, we get a polynomial expansion 

\begin{equation}\label{equ:polyequi}
\sum_{\la\in\Z}\frac{\la^d}{d!}\dim H^0(X,mL)_\la=\frac{m^{n+d}}{(n+d)!}\pi_*\left(c_1^{\G_m}(L)^{n+d}\cdot[X]_{\G_m}\right)+O(m^{n+d-1})
\end{equation}
with $\pi_*\left(c_1^{\G_m}(L)^{n+d}\cdot[X]_{\G_m}\right)\in\CH^{\G_m}_{-d}(\Spec k)\simeq\Z$.
\end{proof}

\begin{rmk}\label{rmk:coef} Comparing the two proofs of
  Theorem~\ref{thm:equivRR}, we see in particular that 
$$
\pi_*\left(c_1^{\G_m}(L)^{n+d}\cdot[X]_{\G_m}\right)=c_1(L_d)^{n+d}\cdot[X_d]. 
$$
This equality probably follows directly from the construction of equivariant cohomology, since~\cite[\S3.1]{EG1} implies that 
$$
\CH_i^{\G_m}(X)\simeq\CH_{i+d}(X_d)
$$
for $i\ge n-d$.
\end{rmk}

\end{document}